\documentclass[10pt]{article}
\usepackage{hyperref}
\usepackage{amsmath}
\usepackage{amsfonts}
\usepackage{amssymb}
\usepackage{amsthm}

\usepackage[shortlabels]{enumitem}
\usepackage{microtype}
\usepackage{fullpage}
\usepackage{mathrsfs}
\usepackage{mleftright}
\mleftright

\NewDocumentCommand{\R}{}{\mathbb{R}}

\NewDocumentCommand{\Z}{}{\mathbb{Z}}

\NewDocumentCommand{\Zg}{}{\Z_{>}}
\NewDocumentCommand{\Zgeq}{}{\Z_{\geq}}

\NewDocumentCommand{\Span}{}{\mathrm{span}}
\NewDocumentCommand{\QCompact}{}{\mathcal{Q}}
\NewDocumentCommand{\Compact}{}{\mathcal{K}}

\NewDocumentCommand{\sS}{}{\mathcal{S}}
\NewDocumentCommand{\sD}{}{\mathcal{D}}

\NewDocumentCommand{\sI}{}{\mathcal{I}}
\NewDocumentCommand{\bt}{}{\tilde{b}}
\NewDocumentCommand{\ct}{}{\tilde{c}}
\NewDocumentCommand{\psih}{}{\hat{\psi}}
\NewDocumentCommand{\Psih}{}{\widehat{\Psi}}
\NewDocumentCommand{\rhot}{}{\tilde{\rho}}
\NewDocumentCommand{\deltah}{}{\hat{\delta}}
\NewDocumentCommand{\xih}{}{\hat{\xi}}
\NewDocumentCommand{\hh}{}{\hat{h}}
\NewDocumentCommand{\ch}{}{\hat{c}}
\NewDocumentCommand{\at}{}{\tilde{a}}

\NewDocumentCommand{\Omegat}{}{\widetilde{\Omega}}
\NewDocumentCommand{\phit}{}{\tilde{\phi}}

\NewDocumentCommand{\supp}{o}{\mathrm{supp}\IfValueT{#1}{(#1)}}

\NewDocumentCommand{\Manifold}{m}{\mathfrak{#1}}

\NewDocumentCommand{\ManifoldM}{}{\Manifold{M}}
\NewDocumentCommand{\ManifoldN}{}{\Manifold{N}}
\NewDocumentCommand{\ManifoldNt}{}{\widetilde{\ManifoldN}}
\NewDocumentCommand{\ManifoldMt}{}{\widetilde{\ManifoldM}}
\NewDocumentCommand{\ManifoldMh}{}{\widehat{\ManifoldM}}
\NewDocumentCommand{\BoundaryN}{}{\partial \ManifoldN}
\NewDocumentCommand{\BoundaryOmega}{}{\partial \Omega}
\NewDocumentCommand{\InteriorN}{}{\mathrm{Int}(\ManifoldN)}

\NewDocumentCommand{\BoundaryNnc}{}{\BoundaryN^{\mathrm{nc}}}
\NewDocumentCommand{\BoundaryNt}{}{\partial \ManifoldNt}

\NewDocumentCommand{\ManifoldNncWWd}{}{\ManifoldN^{\mathrm{nc}}_{\WWd}}
\NewDocumentCommand{\BoundaryNncWWd}{}{\BoundaryN^{\mathrm{nc}}_{\WWd}}
\NewDocumentCommand{\ManifoldNncF}{}{\ManifoldN^{\mathrm{nc}}_{\FilteredSheafF}}
\NewDocumentCommand{\ManifoldNncG}{}{\ManifoldN^{\mathrm{nc}}_{\FilteredSheafG}}
\NewDocumentCommand{\BoundaryNncF}{}{\BoundaryN^{\mathrm{nc}}_{\FilteredSheafF}}

\NewDocumentCommand{\VectorFields}{m}{\mathscr{X}(#1)}
\NewDocumentCommand{\VectorFieldsN}{}{\VectorFields{\ManifoldN}}
\NewDocumentCommand{\VectorFieldsM}{}{\VectorFields{\ManifoldM}}
\NewDocumentCommand{\VectorFieldsBoundaryN}{}{\VectorFields{\BoundaryN}}

\NewDocumentCommand{\SheafVectorFields}{o}{\mathscr{X}\IfValueT{#1}{(#1)}}

\NewDocumentCommand{\TangentSpace}{m m}{T_{#2}#1}
\NewDocumentCommand{\NTangentSpace}{m}{\TangentSpace{\ManifoldN}{#1}}

\NewDocumentCommand{\CmSpace}{m o o}{C^{#1}\IfValueT{#2}{(#2 \IfValueT{#3}{;#3} )}}
\NewDocumentCommand{\CinftySpace}{o o}{\CmSpace{\infty}[#1][#2]}
\NewDocumentCommand{\CzinftySpace}{o}{C_0^{\infty}\IfValueT{#1}{(#1)}}
\NewDocumentCommand{\CmbSpace}{m o o}{C^{#1}_b\IfValueT{#2}{(#2 \IfValueT{#3}{;#3} )}}
\NewDocumentCommand{\CinftybSpace}{o o}{\CmbSpace{\infty}[#1][#2]}
\NewDocumentCommand{\CmbNorm}{m m o o}{\left\|#1\right\|_{\CmbSpace{#2}[#3][#4]}}

\NewDocumentCommand{\LpSpace}{m o o}{L^{#1}\IfValueT{#2}{(#2\IfValueT{#3}{;#3})}}

\NewDocumentCommand{\LpNorm}{m m o o}{\left\| #1 \right\|_{\LpSpace{#2}[#3][#4]}}

\NewDocumentCommand{\Vd}{}{d\hspace*{-0.08em}\widetilde{}\hspace*{0.1em}}
\NewDocumentCommand{\Xd}{}{{d\hspace*{-0.08em}\bar{}\hspace*{0.1em}}}
\NewDocumentCommand{\Wd}{}{d}
\NewDocumentCommand{\Yd}{}{{d\hspace*{-0.10em}\widehat{}\hspace*{0.1em}}}
\NewDocumentCommand{\Zd}{}{{d\hspace*{-0.08em}\widetilde{}\hspace*{0.1em}}}
\NewDocumentCommand{\Wdx}{}{\Wd^x}
\NewDocumentCommand{\Wdxalpha}{}{\Wd^{x(\alpha)}}

\NewDocumentCommand{\Wh}{}{\widehat{W}}
\NewDocumentCommand{\Xh}{}{\widehat{X}}
\NewDocumentCommand{\Zh}{}{\widehat{Z}}
\NewDocumentCommand{\Yh}{}{\widehat{Y}}
\NewDocumentCommand{\Wx}{}{W^x}
\NewDocumentCommand{\Whx}{}{\Wh^x}

\NewDocumentCommand{\XXd}{}{( X, \Xd)}
\NewDocumentCommand{\WWd}{}{( W, \Wd)}
\NewDocumentCommand{\WWdx}{}{( \Wx, \Wdx)}
\NewDocumentCommand{\WhWd}{}{( \Wh, \Wd)}
\NewDocumentCommand{\ZhZd}{}{( \Zh, \Zd)}
\NewDocumentCommand{\XhXd}{}{(\Xh,\Xd)}
\NewDocumentCommand{\VVd}{}{( V, \Vd)}

\NewDocumentCommand{\ZZd}{}{( Z, \Zd)}
\NewDocumentCommand{\partialo}{}{(\partial, 1)}

\NewDocumentCommand{\WWo}{}{( W, 1)}

\NewDocumentCommand{\WhWo}{}{( \Wh, 1)}

\NewDocumentCommand{\Vol}{o}{\mathrm{Vol}\IfValueT{#1}{\left(#1\right)}}
\NewDocumentCommand{\Volh}{o}{\widehat{\mathrm{Vol}}\IfValueT{#1}{\left(#1\right)}}
\NewDocumentCommand{\LebDensity}{o}{\sigma_{\mathrm{Leb}}\IfValueT{#1}{\left(#1\right)}}

\NewDocumentCommand{\psiWhWd}{}{(\psi_{x,\delta}^{*} \Wh,\Wd)}
\NewDocumentCommand{\psiXhXd}{}{(\psi_{x,\delta}^{*} \Xh,\Xd)}
\NewDocumentCommand{\psiVVd}{}{(\psi_{x,\delta}\big|_{\nCube{\sigma_1}}^{*}\delta^{\Vd}V, \Vd)}

\NewDocumentCommand{\BpsideltaWhWd}{m m}{B_{(\psi_{x,\delta}^{*}\delta^{\Wd}\Wh, \Wd)}(#1,#2)}
\NewDocumentCommand{\BpsideltaWWd}{m m}{B_{(\psi_{x,\delta}^{*}\delta^{\Wd}W, \Wd)}(#1,#2)}
\NewDocumentCommand{\BpsideltaXhXd}{m m}{B_{(\psi_{x,\delta}^{*}\delta^{\Xd}\Xh, \Xd)}(#1,#2)}
\NewDocumentCommand{\BpsideltaXXd}{m m}{B_{(\psi_{x,\delta}^{*}\delta^{\Xd}X, \Xd)}(#1,#2)}
\NewDocumentCommand{\BpsideltaVVd}{m m}{B_{(\psi_{x,\delta}\big|_{\nCube{\sigma_1}}^{*}\delta^{\Vd}V, \Vd)}(#1,#2)}

\NewDocumentCommand{\BWWd}{m m}{B_{\WWd}(#1,#2)}
\NewDocumentCommand{\BWWo}{m m}{B_{\WWo}(#1,#2)}
\NewDocumentCommand{\BXXd}{m m}{B_{\XXd}(#1,#2)}
\NewDocumentCommand{\BVVd}{m m}{B_{\VVd}(#1,#2)}
\NewDocumentCommand{\BZZd}{m m}{B_{\ZZd}(#1,#2)}
\NewDocumentCommand{\BWhWd}{m m}{B_{\WhWd}(#1,#2)}
\NewDocumentCommand{\BXhXd}{m m}{B_{\XhXd}(#1,#2)}
\NewDocumentCommand{\BZhZd}{m m}{B_{\ZhZd}(#1,#2)}
\NewDocumentCommand{\BWhWo}{m m}{B_{\WhWo}(#1,#2)}
\NewDocumentCommand{\MetricVFs}{m o o}{\rho_{#1} \IfValueT{#2}{(#2\IfValueT{#3}{,#3})}}
\NewDocumentCommand{\MetricWWd}{o o}{\MetricVFs{\WWd}[#1][#2]}
\NewDocumentCommand{\MetricVVd}{o o}{\MetricVFs{\VVd}[#1][#2]}
\NewDocumentCommand{\MetricXXd}{o o}{\MetricVFs{\XXd}[#1][#2]}
\NewDocumentCommand{\MetricXhXd}{o o}{\MetricVFs{\XhXd}[#1][#2]}
\NewDocumentCommand{\MetricWhWd}{o o}{\MetricVFs{\WhWd}[#1][#2]}
\NewDocumentCommand{\MetricZZd}{o o}{\MetricVFs{\ZZd}[#1][#2]}
\NewDocumentCommand{\MetricZhZd}{o o}{\MetricVFs{\ZhZd}[#1][#2]}

\NewDocumentCommand{\Gen}{m}{\mathrm{Gen}(#1)}
\NewDocumentCommand{\GenWWd}{}{\mathrm{Gen}(W,\Wd)}
\NewDocumentCommand{\GenXXd}{}{\mathrm{Gen}(X,\Xd)}
\NewDocumentCommand{\GenWhWd}{}{\mathrm{Gen}(\Wh,\Wd)}
\NewDocumentCommand{\GenXhXd}{}{\mathrm{Gen}(\Xh,\Xd)}

\NewDocumentCommand{\degParams}{m m o}{\mathrm{deg^{#1}_{#2}}\IfValueT{#3}{(#3)}}
\NewDocumentCommand{\degBoundaryN}{m o}{\mathrm{deg^{\BoundaryN}_{#1}}\IfValueT{#2}{(#2)}}
\NewDocumentCommand{\degBoundaryOmega}{m o}{\mathrm{deg^{\BoundaryOmega}_{#1}}\IfValueT{#2}{(#2)}}
\NewDocumentCommand{\degBoundaryOmegaWWd}{o}{\degBoundaryOmega{\WWd}[#1]}
\NewDocumentCommand{\degBoundaryNWWd}{o}{\degBoundaryN{\WWd}[#1]}
\NewDocumentCommand{\degBoundaryNF}{o}{\degBoundaryN{\FilteredSheafF}[#1]}

\NewDocumentCommand{\nBall}{m}{B^n(#1)}
\NewDocumentCommand{\nmoBall}{m}{B^{n-1}(#1)}
\NewDocumentCommand{\nBallgeq}{m o}{B^n_{\geq \IfValueT{#2}{#2}}(#1)}
\NewDocumentCommand{\nBalleq}{m o}{B^n_{= \IfValueT{#2}{#2}}(#1)}
\NewDocumentCommand{\nUnitBall}{}{\nBall{1}}
\NewDocumentCommand{\nmoUnitBall}{}{\nmoBall{1}}
\NewDocumentCommand{\nUnitBallgeq}{o}{\nBallgeq{1}[#1]}
\NewDocumentCommand{\nUnitBalleq}{o}{\nBalleq{1}[#1]}

\NewDocumentCommand{\CubeCentered}{m m}{Q^n(#1,#2)}
\NewDocumentCommand{\CubeCenteredgeq}{m m o}{Q^n_{\geq \IfValueT{#3}{#3}}(#1,#2)}
\NewDocumentCommand{\CubeCenteredeq}{m m o}{Q^n_{= \IfValueT{#3}{#3}}(#1,#2)}

\NewDocumentCommand{\nCube}{m}{Q^n(#1)}
\NewDocumentCommand{\nmoCube}{m}{Q^{n-1}(#1)}
\NewDocumentCommand{\nCubegeq}{m o}{Q^n_{\geq \IfValueT{#2}{#2}}(#1)}
\NewDocumentCommand{\nCubeeq}{m o}{Q^n_{= \IfValueT{#2}{#2}}(#1)}
\NewDocumentCommand{\nUnitCube}{}{\nCube{1}}
\NewDocumentCommand{\nUnitCubegeq}{o}{\nCubegeq{1}[#1]}
\NewDocumentCommand{\nUnitCubeeq}{o}{\nCubeeq{1}[#1]}
\NewDocumentCommand{\Rn}{}{\R^n}
\NewDocumentCommand{\Rngeq}{o}{\Rn_{\geq \IfValueT{#1}{#1}}}
\NewDocumentCommand{\Rnmo}{}{\R^{n-1}}

\NewDocumentCommand{\RestrictFilteredSheaf}{m m o}{%
  #1\big|_{#2}^{\#}\IfValueT{#3}{(#3)}
}

\NewDocumentCommand{\RestrictFilteredPreSheaf}{m m o o}{%
  \mathrm{Res}_{#2}%
    [#1]_{%
    \IfValueTF{#4}{#4}{\bullet}%
  }%
  \IfValueT{#3}{(#3)}
}

\NewDocumentCommand{\RestrictFilteredPreSheafF}{m o o}{\RestrictFilteredPreSheaf{\FilteredSheafF}{#1}[#2][#3]}

\NewDocumentCommand{\RestrictPreSheaf}{m m o}{\mathrm{Res}_{#2}[#1]\IfValueT{#3}{(#3)}}
\NewDocumentCommand{\RestrictPreSheafF}{m o}{\RestrictPreSheaf{\SheafF}{#1}[#2]}

\NewDocumentCommand{\RestrictSheaf}{m m o}{#1\big|_{#2}^{\#}\IfValueT{#3}{(#3)}}
\NewDocumentCommand{\RestrcitSheafF}{m o}{\RestrictSheaf{\SheafF}{#1}[#2]}

\NewDocumentCommand{\LieFilteredSheaf}{m o o}{%
  \mathrm{Lie}%
    (#1)_{%
    \IfValueTF{#3}{#3}{\bullet}%
  }%
  \IfValueT{#2}{(#2)}
}

\NewDocumentCommand{\LieFilteredSheafF}{o o}{\LieFilteredSheaf{\FilteredSheafF}[#1][#2]}

\NewDocumentCommand{\SheafGenBy}{m o}{\left\langle #1\right\rangle \IfValueT{#2}{(#2)}}

\NewDocumentCommand{\FilteredSheafGenBy}{m o o}{\left\langle #1\right\rangle_{%
\IfValueTF{#3}{#3}{\bullet}%
}%
\IfValueT{#2}{(#2)}}

\NewDocumentCommand{\Sheaf}{m o}{%
  \mathscr{#1}\IfValueT{#2}{(#2)}%
}
\NewDocumentCommand{\SheafHat}{m o}{%
  \widehat{\mathscr{#1}}\IfValueT{#2}{(#2)}%
}

\NewDocumentCommand{\SheafF}{o}{\Sheaf{F}[#1]}
\NewDocumentCommand{\SheafG}{o}{\Sheaf{G}[#1]}
\NewDocumentCommand{\SheafGh}{o}{\SheafHat{G}[#1]}

\NewDocumentCommand{\FilteredModule}{m o}{%
{#1}_{%
  \IfValueTF{#2}{#2}{\bullet}%
}%
}
\NewDocumentCommand{\FilteredModuleM}{o}{\FilteredModule{M}[#1]}

\NewDocumentCommand{\FilteredSheaf}{m o o}{%
  \mathscr{#1}_{%
    \IfValueTF{#3}{#3}{\bullet}%
  }%
  \IfValueT{#2}{(#2)}%
}

\NewDocumentCommand{\FilteredSheafHat}{m o o}{%
  \widehat{\mathscr{#1}}_{%
    \IfValueTF{#3}{#3}{\bullet}%
  }%
  \IfValueT{#2}{(#2)}%
}

\NewDocumentCommand{\FilteredSheafTilde}{m o o}{%
  \widetilde{\mathscr{#1}}_{%
    \IfValueTF{#3}{#3}{\bullet}%
  }%
  \IfValueT{#2}{(#2)}%
}

\NewDocumentCommand{\FilteredSheafNoSet}{m o}{%
  \mathscr{#1}_{%
    \IfValueTF{#2}{#2}{\bullet}%
  }%
}

\NewDocumentCommand{\FilteredSheafTildeNoSet}{m o}{%
  \widetilde{\mathscr{#1}}_{%
    \IfValueTF{#2}{#2}{\bullet}%
  }%
}

\NewDocumentCommand{\FilteredSheafNoSetF}{o}{\FilteredSheafNoSet{F}[#1]}
\NewDocumentCommand{\FilteredSheafNoSetFt}{o}{\FilteredSheafTildeNoSet{F}[#1]}

\NewDocumentCommand{\FilteredSheafF}{o o}{\FilteredSheaf{F}[#1][#2]}
\NewDocumentCommand{\FilteredSheafG}{o o}{\FilteredSheaf{G}[#1][#2]}
\NewDocumentCommand{\FilteredSheafFh}{o o}{\FilteredSheafHat{F}[#1][#2]}
\NewDocumentCommand{\FilteredSheafFt}{o o}{\FilteredSheafTilde{F}[#1][#2]}

\NewDocumentCommand{\GeneratedModule}{m o o}{\left\langle #1  \right\rangle\IfValueT{#2}{_{#2\IfValueTF{#3}{,#3}{,\bullet}}}}

\NewDocumentCommand{\UnderlyingSheaf}{m o}{\mathscr{U}\left[#1\right]\IfValueT{#2}{(#2)}}

\NewDocumentCommand{\rhoF}{}{\rho_{\FilteredSheafF}}

\NewDocumentCommand{\opL}{}{\mathscr{L}}
\NewDocumentCommand{\opP}{}{\mathscr{P}}
\NewDocumentCommand{\opE}{}{\mathscr{E}}

\NewDocumentCommand{\Distributions}{o}{\mathscr{D}'\IfValueT{#1}{(#1)}}
\NewDocumentCommand{\degWd}{m}{\mathrm{deg}_{\Wd}(#1)}

\begin{document}

\newtheorem{theorem}{Theorem}[section]
\newtheorem{corollary}[theorem]{Corollary}
\newtheorem{proposition}[theorem]{Proposition}
\newtheorem{lemma}[theorem]{Lemma}
\newtheorem{conjecture}[theorem]{Conjecture}
\newtheorem{problem}[theorem]{Problem}

\theoremstyle{remark}
\newtheorem{remark}[theorem]{Remark}

\theoremstyle{definition}
\newtheorem{definition}[theorem]{Definition}
\newtheorem{notation}[theorem]{Notation}
\newtheorem{construction}[theorem]{Construction}

\theoremstyle{remark}
\newtheorem{example}[theorem]{Example}

\numberwithin{equation}{section}

\title{Carnot--Carath\'eodory Balls on Manifolds with Boundary}

\author{Brian Street\footnote{The author was partially supported by National Science Foundation Grant 2153069.}}
\date{}

\maketitle

\begin{abstract}
    Nagel, Stein, and Wainger introduced a detailed quantitative study of Carnot--Carath\'eodory balls on a 
smooth manifold without boundary. Most importantly, they introduced scaling maps adapted to Carnot--Carath\'eodory balls
and H\"ormander vector fields. Their work
was extended by many authors and
has since become a key tool in the study of the interior theory of subelliptic
PDEs; in particular, the study of maximally subelliptic PDEs. 
We introduce a generalization of this quantitative theory to manifolds with boundary, where we have scaling maps
both on the interior and on the part of the boundary which is non-characteristic with respect to the vector fields. 
This is the first paper in a forthcoming series devoted to studying maximally subelliptic boundary value problems.

\end{abstract}

\section{Introduction}
In \cite{NagelSteinWaingerBallsAndMetricsDefinedByVectorFieldsI}, Nagel, Stein, and Wainger
gave a detailed quantitative study of Carnot--Carath\'eodory balls on a smooth manifold without boundary.
In particular, they introduced scaling maps adapted to a Carnot--Carath\'eodory geometry which have
since become a central tool in the interior theory of subelliptic PDEs.
The goal of this paper is to generalize this theory to manifolds with boundary,
both on the interior of the manifold and on the ``non-characteristic boundary.''
Just as Nagel, Stein, and Wainger were motivated by applications to interior subelliptic PDEs,
we are motivated by subelliptic boundary value problems. In fact, this is the first paper in a forthcoming
series devoted to introducing a general theory of maximally subelliptic boundary value problems.
\textit{The reader who is familiar with the interior theory may wish to skip straight to 
Sections \ref{Section::Defns} and \ref{Section::GlobalCors} where we give the basic definitions
and some easy to understand corollaries of our main result.}  The main technical result (the scaling theorem)
can be found in Section \ref{Section::Scaling}.
Before entering into technical details, in this introduction we begin by describing some ideas
from the more well-studied interior theory, as that is the best way to explain the motivation 
for the results in this paper.

In the seminal paper \cite{HormanderHypoellipticSecondOrderDifferentialEquations}, H\"ormander introduced
an important class of examples of subelliptic operators. Let \(X_0,X_1,\ldots, X_p\) be smooth vector fields
on a manifold without boundary, \(\ManifoldM\), satisfying \textit{H\"ormander's condition}:
 the Lie algebra generated by \(X_0,X_1,\ldots, X_p\) spans the tangent space at every point (see Definition \ref{Defn::BasicDefns::HormandersCondition}).
Let
\begin{equation*}
    \opL=X_0+X_1^2+X_2^2+\cdots+X_p^2.
\end{equation*}
\(\opL\) (and other maximally subelliptic operators, as described in Definition \ref{Defn::Intro::MaxSub}) arises in a number of settings,
including stochastic calculus and several complex variables--see the introduction of \cite{BramantiBrandoliniHormanderOperators}
for a friendly account.
H\"ormander showed that \(\opL\) is \textit{subelliptic}: roughly speaking that if \(u\in \Distributions[\ManifoldM]\)
and \(\opL u\in L^2_s\) near a point \(x\), then \(u\in L^2_{s+\epsilon}\) near \(x\), where \(L^2_s\)
is the \(\LpSpace{2}\)-Sobolev space of order \(s\), and \(\epsilon=\epsilon(x)>0\).\footnote{In the introduction,
all function spaces are with respect to some chosen smooth, strictly positive density. It does not matter which
such density is used.}
Informally, \(u\) is smoother than \(\opL u\). This is a powerful insight: \(\opL\) can be nowhere elliptic,
and the classical elliptic theory it not useful when studying \(\opL\); nevertheless,
one can still obtain subellipticity for \(\opL\).

Rothschild and Stein \cite{RothschildSteinHypoellipticDifferentialOperatorsAndNilpotentGroups} showed that more is true: 
\(\opL\) is \textit{maximally subelliptic}\footnote{A term which had not yet been defined or studied.
The works of 
Folland and Stein \cite{FollandSteinEstimatesForTheBarPartialBComplex},
Folland \cite{FollandSubellipticEstimatesAndFunctionSpacesOnNilpotentLieGroups},
and 
Rothschild and Stein \cite{RothschildSteinHypoellipticDifferentialOperatorsAndNilpotentGroups}
were the first examples what eventually became general definition. A general definition, under the French name
\textit{hypoellipticit\'e maximale}, first appeared in
work of Helffer and Nourrigat
\cite{HelfferNourrigatHypoellipticiteMaximalePourDesOperateursPolynomesDeChampsDeVecteurs}. See Definition \ref{Defn::Intro::MaxSub}.}. 
Let \(\WWd=\left\{ (W_1,\Wd_1),\ldots, \left( W_r, \Wd_r \right) \right\}\)
be H\"ormander vector fields on \(\ManifoldN\), each paired with a formal degree \(\Wd_j\in \Zg\).
In the case of \(\opL\) above, we take \(\WWd=\left\{ (X_0,2), (X_1,1),\ldots, (X_p,1) \right\}\).
For a list \(\alpha=(\alpha_1,\alpha_2,\ldots,\alpha_L)\in \left\{ 1,\ldots, r \right\}^L\)
let \(W^{\alpha}=W_{\alpha_1}W_{\alpha_2}\cdots W_{\alpha_L}\) and \(\degWd{\alpha}=\Wd_{\alpha_1}+\cdots +\Wd_{\alpha_L}\).

\begin{definition}
    \label{Defn::Intro::MaxSub}
    Let \(\kappa\in \Zg\) be such that \(\Wd_j\) divides \(\kappa\) for every \(1\leq j\leq r\).
    Let \(\opP\) be a partial differential operator of the form
    \begin{equation}\label{Eqn::Intro::MaxSub::opP}
        \opP=\sum_{\degWd{\alpha}\leq \kappa} a_\alpha W^{\alpha}, \quad a_\alpha\in \CinftySpace[\ManifoldM].
    \end{equation}
    We say \(\opP\) is \textit{maximally subelliptic\footnote{Also know as maximally hypoelliptic.} with respect to }\(\WWd\) if for all
    \(\Omega\Subset \ManifoldM\) open and relatively compact,\footnote{We write \(A\Subset B\) if \(A\) is relatively compact
    in \(B\); i.e., if the closure of \(A\) as a subspace of \(B\) is compact.} we have
    \begin{equation*}
        \sum_{j=1}^r\LpNorm{W_j^{\kappa/\Wd_j} f}{2} \lesssim \LpNorm{\opP f}{2} + \LpNorm{f}{2},\quad \forall f\in \CzinftySpace[\Omega].
    \end{equation*}
\end{definition}

Maximally subelliptic operators are subelliptic, but as we will explain much more is true.
Before we do so, we look back to the well-studied area of elliptic operators.
There is by now a vast theory for elliptic PDEs which has been distilled in countless textbooks
and monographs: see, for example the three volume series
\cite{TaylorPartialDifferentialEquationsI,TaylorPartialDifferentialEquationsII,TaylorPartialDifferentialEquationsIII},
or for the related function spaces the many textbooks of Triebel (like \cite{TriebelTheoryOfFunctionSpaces}).
There are many others, too many to list here.
This theory covers linear PDEs, PDEs with rough coefficients, nonlinear (including fully nonlinear) PDEs,
boundary value problems, and much more, where sharp results are known in a variety of the classical function spaces.
Many proofs in the  elliptic theory seem short at first glance, but rely on a huge foundation
of function spaces (like Besov and Triebel--Lizorkin spaces), operators (like pseudodifferential operators),
 geometry (usually Euclidean or Riemannian), and classical transforms like the Fourier transform.

If one takes the H\"ormander vector fields with formal degrees on \(\R^n\) given by
\begin{equation*}
    \partialo:=\left\{ (\partial_{x_1},1),\ldots, (\partial_{x_n},1) \right\},
\end{equation*}
then an operator \(\opP\) is maximally subelliptic with respect to \(\partialo\) if and only if it is (locally) elliptic.
While subellipticity is a \textit{weaker} condition than ellipticity,
maximal subellipticity is instead a \textit{generalization} of ellipticity: ellipticity is the special case of maximal subellipticity
when one choses \(\WWd=\partialo\).  If one changes \(\WWd\), maximal subellipticity becomes different condition, but 
not a weaker one.
When viewed through the lens of the classical elliptic theory, maximally subelliptic operators can look
very degenerate (and are often nowhere elliptic--see \cite[Remark 8.2.7]{StreetMaximalSubellipticity}). However, they satisfy a different condition which in some ways
is as strong as ellipticity.

In light of this, given any classical result from the (vast) elliptic theory, one can ask if it is possible
to prove a more general result for maximally subelliptic operators which specializes to classical result without weakening
the conclusions.
A wide-ranging program was initiated to adapt the elliptic theory to this more general and more degenerate setting;
this has spanned thousands of papers and many books (see \cite{BramantiAnInvitationToHypoellipticOperators,BramantiBrandoliniHormanderOperators} for friendly
introductions and \cite{StreetMaximalSubellipticity} for general results and a further history).
A major difficulty is that the foundation used for elliptic operators (classical function spaces, operators,
geometry, and transforms) is not as useful when studying maximally subelliptic PDEs.
Because of this, to generalize elliptic results, one needs a more general foundation on which to work.
There are many complications in doing so: one major complication is that the Fourier transform is no longer as useful,
and much of the foundational theory for elliptic operators rests on the Fourier transform.
This requires new definitions, new operators, and new proofs.

Despite this major difficulty,
there is now a detailed study for the interior theory of maximally subelliptic PDEs which in many ways
generalizes the classical interior elliptic theory. This includes not only linear operators with smooth coefficients,
but operators with rough coefficients, and nonlinear operators (including general fully nonlinear operators).
Results are often in terms of function spaces which are adapted to \(\WWd\).
See \cite{StreetMaximalSubellipticity} for many of these results, along with a more detailed history.

There is also a deep theory of elliptic boundary value problems; see, for example, \cite{AgranovichEgorovShubinPartialDifferentialEquationsIX}.
While there have been many examples of
what might be called
maximally subelliptic boundary value problems studied
(the first of which may be the work of Jerison \cite{JerisonDirichletProblemForTheKohnLaplacianI}),
there is not yet even a  definition of \textit{maximally subelliptic boundary value probelms}
which generalizes the elliptic case; let alone any general results which generalize the elliptic theory.
The main reason there is currently no such theory is that the huge foundational theory of function
spaces and operators on which the elliptic theory rests does not apply to the more general maximally subelliptic case.
This is the first in a series of papers aimed to fill this gap, and provide a general theory of maximally subelliptic
boundary value problems; here we are given a list of H\"ormander vector fields with formal degrees
\(\WWd\) now on a manifold \textit{with boundary}, \(\ManifoldN\).

One important concept, first pointed out by Kohn and Nirenberg \cite{KohnNirenbergNonCoerciveBoundaryValueProblems},
Derridj \cite{DerridjSurUnTheoremeDeTraces},
and Jerison \cite{JerisonDirichletProblemForTheKohnLaplacianI,JerisonDirichletProblemForTheKohnLaplacianII}
is that even in the simplest cases, the interaction between \(\WWd\) and \(\BoundaryN\) is important.
They showed that whether or not the boundary was ``characteristic'' with respect to \(\WWd\)
had significant impacts on the study of boundary value problems.
When \(\Wd_j=1\), \(\forall 1\leq j\leq r\), we say \(x_0\in \BoundaryN\) is \(\WWo\)-non-characteristic
if \(\exists j\) with \(W_j(x_0)\not \in \TangentSpace{\BoundaryN}{x_0}\). The definition
when some of the \(\Wd_j\) are not equal to \(1\) is a bit more complicated; see Definition \ref{Defn::BasicDefns::NoncharteristicPoint}.

Returning to the interior theory of maximally subelliptic PDEs, the first step towards a general
theory was taken in the important work of Nagel, Stein, and Wainger \cite{NagelSteinWaingerBallsAndMetricsDefinedByVectorFieldsI}
who gave a detailed study of Carnot--Carath\'eodory balls (sometimes known as sub-Riemannian balls).
These are metric balls on a manifold induced by the collection of H\"ormander vector fields with formal degrees
\(\WWd\) (see \eqref{Eqn::BasicDefns::UnitCCBall} and \eqref{Eqn::BasicDefns::CCBallAndMetric}) and are of central importance
in the interior theory of maximally subelliptic PDEs.
Where elliptic PDEs are intimately connected to Riemannian geometry, maximally subelliptic PDEs are intimately connected
to Carnot--Carath\'eodory geometry.
Nagel, Stein, and Wainger showed:
\begin{itemize}
    \item The balls can be defined in several different ways; each definition gives locally comparable balls, 
    and locally Lipschitz equivalent metrics. See Theorem \ref{Thm::Metrics::Results::LocalWeakEquivImpliesEquiv}.
    \item If \(\Vol\) is a smooth, strictly positive density,\footnote{A smooth, strictly positive density can be thought of as a measure which in any local coordinate system is given by integration against a smooth, strictly positive function.
    See the beginning of \cite[Section 3.1]{StreetMaximalSubellipticity} for a quick overview of the properties we will use and references for further reading.}
     then the balls satisfy a local doubling condition with respect to \(\Vol\):
        \(\Vol[B(x,2\delta)]\lesssim \Vol[B(x,\delta)]\), uniformly for \(x\) in a compact set and \(\delta\) small (see Theorem \ref{Thm::Scaling::MainResult} \ref{Item::Scaling::MainResult::Doubling}).
        This implies Carnot-Carath\'eodory manifolds are locally spaces of homogeneous type in the sense of Coifmann and Weiss \cite{CoifmanWeissAnalyseHarmoniqueNonCommutative}.
    \item Most importantly, 
    they introduced scaling maps\footnote{While \cite{NagelSteinWaingerBallsAndMetricsDefinedByVectorFieldsI} did not emphasize the scaling
    nature of the maps they introduced, it was later used by many authors; see \cite{KoenigMaximalSobolevAndHolderEstimatesForTheTangentialCauchyRiemannOperatorAndBoundaryLaplacian}
    for one explicit such example.} which turned small Carnot-Carath\'eodory scales into the unit scale (see Section \ref{Section::Intro::Scaling} and Theorem \ref{Thm::Scaling::MainResult}).
    These scaling maps are the central tool when pursuing a quantitative study of maximally subelliptic PDEs.
\end{itemize}
There have been several works generalizing and refining the results of \cite{NagelSteinWaingerBallsAndMetricsDefinedByVectorFieldsI};
for example, 
\cite{TaoWrightLpImprovingBoundsForAveragesAlongCurves,
FeffermanPhongSubellipticEigenvalueProblems,
FeffermanSanchezCalleFundamentalSolutionsForSecondOrderSubellipticOperators,
MontanariMorbidelliNonsmoothHormanderVectorFieldsAndTheirControlBalls,
StovaStreetCoordinatesAdaptedToVectorFieldsCanonicalCoordinates,
StreetCoordinatesAdaptedToVectorFieldsII,
StreetCoordinatesAdaptedToVectorFieldsIII,
StreetSubHermitianGeometryAndTheQuantitativeNewlanderNirenbergTheorem}.

In this paper, we generalize the theory of \cite{NagelSteinWaingerBallsAndMetricsDefinedByVectorFieldsI}
(incorporating ideas from \cite{StovaStreetCoordinatesAdaptedToVectorFieldsCanonicalCoordinates}) to manifolds
with boundary near a non-characteristic point. Much as \cite{NagelSteinWaingerBallsAndMetricsDefinedByVectorFieldsI}
opened the door to the interior theory of maximally subelliptic PDEs in generality,
this paper is the first step towards obtaining a sharp and general theory of maximally subelliptic boundary value problems,
which will be the topic of future papers.
Analogous to \cite{NagelSteinWaingerBallsAndMetricsDefinedByVectorFieldsI}, we show
\begin{itemize}
    \item The Carnot--Carath\'eodory balls can be defined in several different ways; each definition gives locally comparable balls, 
    and locally Lipschitz equivalent metrics. See Theorems \ref{Thm::Metrics::Results::LocalWeakEquivImpliesEquiv} and
    \ref{Thm::Metrics::Results::ExtendedVFsGiveSameMetric}.
\item If \(\Vol\) is a smooth, strictly positive density,
     then the balls satisfy a local doubling condition with respect to \(\Vol\):
        \(\Vol[B(x,2\delta)]\lesssim \Vol[B(x,\delta)]\), uniformly for \(x\) in a compact set of the interior and \(\WWd\)-non-characteristic boundary and \(\delta\) small (see Theorem \ref{Thm::Scaling::MainResult} \ref{Item::Scaling::MainResult::Doubling}).

\item Most importantly, we introduce scaling maps which turn small Carnot--Carath\'eodory scales (both on the interior
and on near the \(\WWd\)-non-characteristic boundary) into the unit scale (see Theorem \ref{Thm::Scaling::MainResult}).
\end{itemize}
We also show,
\begin{itemize}
\item \(\WWd\) induces a natural Carnot--Carath\'eodory geometry on the  \(\WWd\)-non-characteristic boundary (see Theorem \ref{Thm::Sheaves::Metrics::BoundaryMetricEquivalence}).
\end{itemize}

When studying maximally subelliptic PDEs one often starts with a list of vector fields with formal degrees
\(\WWd\). However, there are many different equivalent lists which give rise to the same definitions
(like Definition \ref{Defn::Intro::MaxSub}). For example, one often only cares about the local Lipschitz equivalence 
class (see Definition \ref{Defn::Metrics::Intro::LocallyEquivalent}) 
of the Carnot--Carath\'eodory metric when studying maximally subelliptic PDEs.
Many works (like \cite{StreetMaximalSubellipticity}) start by making such a choice
of \(\WWd\) and then describe which other choices are equivalent (see Definition \ref{Defn::BasicDefns::StrongWeakEquivalnce}).
Unfortunately, this becomes too unwieldy when studying boundary value problems, because the data on the boundary
is only well-defined up to this equivalence.
To address this problem, in Section \ref{Section::Sheaves} we introduce a different framework,
involving \(\Zg\)-filtrations of sheaves of vector fields on the manifold.
Instead of giving \(\WWd\), one can just give the filtration of sheaves of vector fields
generated by \(\WWd\). This removes the arbitrary choice and allows us to state our results
on boundary metrics (like Theorem \ref{Thm::Sheaves::Metrics::BoundaryMetricEquivalence})
in a natural way.  Nearly all our results are perhaps best stated in the new langugae, except notably
our most important result: the scaling maps in Theorem \ref{Thm::Scaling::MainResult}.
Since these scaling maps are quantitative, it is important we pick a particular choice of \(\WWd\)
with which to state the quantitative estimates.
Thus, we start by working in the familiar setting where \(\WWd\) is given, and where we prove
Theorem \ref{Thm::Scaling::MainResult} and its consequences.
Then, in Section \ref{Section::Sheaves}, we restate many of these consequences in the more abstract
setting of filtrations of sheaves of vector fields.

    \subsection{The importance of scaling}\label{Section::Intro::Scaling}
    The most technical result in this paper, the scaling result (Theorem \ref{Thm::Scaling::MainResult}),
is also the most important for applications to PDEs.  In fact, the other results in this paper are corollaries  
of Theorem \ref{Thm::Scaling::MainResult}, and Theorem \ref{Thm::Scaling::MainResult} is useful in other ways as well.
Similar remarks hold for previous related papers on manifolds without boundary (like the foundational work of Nagel, Stein, and Wainger
\cite{NagelSteinWaingerBallsAndMetricsDefinedByVectorFieldsI}). While it is tempting to focus on the easier to understand
corollaries about metrics (see Section \ref{Section::GlobalCors}), this misses the main use of these types of results.
In this section, we present an overview of why the scaling maps are useful; starting with the ellipic setting,
and then moving to the more general maximally subelliptic setting.

In the elliptic setting, scaling is so simple it is often used implicitly.
Let \(\opE=\sum_{|\alpha|\leq \kappa} a_\alpha(x) \partial_x^\alpha\) be an elliptic operator on \(\Rn\)
with smooth coefficients. Consider the maps \(\Phi_{x,\delta}(t)=x+\delta t:B^n(0,1)\xrightarrow{\sim}B^n(x,\delta)\).
Then, we have
\begin{equation*}
    \opE_{x,\delta}:=\Phi_{x,\delta}^{*} \delta^{\kappa} \opE \left( \Phi_{x,\delta} \right)_{*} = \sum_{|\alpha|\leq \kappa} a_\alpha (x+\delta t) \delta^{\kappa-|\alpha|} \partial_{t}^{\alpha}.
\end{equation*}
\(\opE_{x,\delta}\) is an elliptic operator (in the \(t\)-variable) on \(B^{n}(0,1)\), uniformly as 
\(x\) ranges over a compact set and \(\delta\in (0,1]\).
Thus, when studying the elliptic operator \(\opE\) on a small scale \(B(x,\delta)\), it often suffices to instead study
a different elliptic operator, \(\opE_{x,\delta}\), at the unit scale.
This scale-invariance is central to the interior study of elliptic PDEs; including in creating adapted function spaces
(like the Besov and Triebel--Liorkin spaces).
Similar scaling works for boundary value problems; for example on \(\Rngeq=\left\{ (x',x_n)\in \R^{n-1}\times \R : x_n\geq 1 \right\}\).
Here, when near the boundary, one can rescale to a set of the form\footnote{When working near the boundary,
cubes are sometimes more  convenient to work with (rather than balls). This is merely a convenience
and not a central point.}
\(\nCubegeq{1}:=\left\{ (x',x_n) : |x'|_{\infty}<1, x_n\in [0,1) \right\}\).
This kind of scaling is central to the study of elliptic boundary value problems.

Turning to the maximally subelliptic setting, let \(\opP\) be given by \eqref{Eqn::Intro::MaxSub::opP}
on a manifold without boundary, \(\ManifoldM\), and suppose \(\opP\) is maximally subelliptic 
with respect to \(\WWd\)
(see Definition \ref{Defn::Intro::MaxSub}).
For each \(x\in \ManifoldM\) and \(\delta \in (0,1]\),
Nagel, Stein, and Wainger \cite{NagelSteinWaingerBallsAndMetricsDefinedByVectorFieldsI}
introduced\footnote{The perspective used here is not the one presented in \cite{NagelSteinWaingerBallsAndMetricsDefinedByVectorFieldsI},
though their results can be easily reformulated as described here. See \cite{StovaStreetCoordinatesAdaptedToVectorFieldsCanonicalCoordinates}
for a paper which presents the results in this way.} 
maps \(\Phi_{x,\delta}:B^n(0,1)\xrightarrow{\sim}\Phi_{x,\delta}(B^n(0,1))\)
satisfying the following
\begin{itemize}
    \item \(\Phi_{x,\delta}(B^n(0,1))\) is comparable to the Carnot--Carath\'eodory ball \(\BWWd{x}{\delta}\)
        as defined in \eqref{Eqn::BasicDefns::CCBallAndMetric}.  See Theorem \ref{Thm::Scaling::MainResult} \ref{Item::Scaling::MainResult::Containments}
        for the corresponding result on manifolds with boundary.
    \item \(\Phi_{x,\delta}^{*}\delta^{\Wd_j}W_j\) are smooth vector fields on \(B^n(0,1)\) and satisfy H\"ormander's condition;
        and both are true \textit{uniformly} for \(\delta>0\) small and as \(x\) ranges over a compact set.
        See Theorem \ref{Thm::Scaling::MainResult} \ref{Item::Scaling::MainResults::UniformHormander}.
    \item  \(\opP_{x,\delta}:=\Phi_{x,\delta}^{*} \delta^{\kappa}\opP \left( \Phi_{x,\delta} \right)_{*}\) is maximally subelliptic
        with respect to \(\left\{ \left( \Phi_{x,\delta}^{*}\delta^{\Wd_1}W_1,\Wd_1 \right),\ldots, \left( \Phi_{x,\delta}^{*} \delta^{\Wd_r}W_r, \Wd_r \right) \right\}\),
        \textit{uniformly} for \(\delta>0\) small and as \(x\) ranges over a compact set.
        See \cite[Section 3.3.1]{StreetMaximalSubellipticity}.
\end{itemize}
In light of this, generalizing the elliptic setting, if one wishes to study the maximally subelliptic operator
\(\opP\) on a small Carnot--Carath\'eodory ball, \(\BWWd{x}{\delta}\), it often suffices to instead
study a different maximally subelliptic operator on the unit ball.
Leveraging this scaling, much of the classical interior theory of elliptic PDEs can be generalized to the maximally
subelliptic setting. This includes the sharp regularity theory (both for linear and even fully nonlinear equations),
and adapted function spaces (generalizing the classical Besov and Triebel--Lizorkin spaces).
See \cite{StreetMaximalSubellipticity} for these results, along with a detailed history.

Turing to maximally subelliptic boundary value problems, the first main hurdle is that no analogous scaling results
are known near the boundary. The goal of this paper is to fill this gap, providing scaling maps near the non-characteristic
boundary. These will be used in forthcoming papers to develop the sharp theory of maximally subelliptic boundary value problems.


\section{Basic definitions}\label{Section::Defns}
Let \(\ManifoldN\) be a smooth manifold with boundary; we denote by \(\BoundaryN\) the boundary and \(\InteriorN\) the interior.
We write \(\VectorFieldsN\) for the space of smooth vector fields on \(\ManifoldN\).

\begin{definition}\label{Defn::BasicDefns::HormandersCondition}
    Fix \(m\in \Zg=\left\{ 1,2,3,\ldots \right\}\) and \(W_1,W_2,\ldots, W_r\in \VectorFieldsN\).
    \begin{itemize}
        \item We say \(W_1,\ldots, W_r\) satisfy \textit{H\"ormander's condition of order \(m\) at \(x\in \ManifoldN\)} if:
        \begin{equation*}
            W_1(x), W_2(x), \ldots, W_r(x), \ldots,\,
            \underbrace{[W_i,W_j](x), \ldots}_{\substack{\text{commutators}\\\text{of order 2}}},\,
            \underbrace{[W_i,[W_j,W_k]](x), \ldots}_{\substack{\text{commutators}\\\text{of order 3}}},\,
            \ldots, \ldots,\,
            \text{commutators of order } m
          \end{equation*}
        span \(\NTangentSpace{x}\).
        \item We say \(W_1,\ldots, W_r\) satisfy \textit{H\"ormander's condition of order \(m\) on \(\ManifoldN\)}
          if \(W_1,\ldots, W_r\) satisfy H\"ormander's condition of order \(m\) at \(x\), \(\forall x\in \ManifoldN\).
        \item We say \(W_1,\ldots, W_r\) satisfy \textit{H\"ormander's condition at \(x\in \ManifoldN\)} if \(\exists m\)
          such that \(W_1,\ldots,W_r\) satisfy H\"ormander's condition of order \(m\) at \(x\).
        \item We say \(W_1,\ldots, W_r\) satisfy \textit{H\"ormander's condition on \(\ManifoldN\)}
        if \(W_1,\ldots, W_r\) satisfy H\"ormander's condition at \(x\), \(\forall x\in \ManifoldN\).
        In this case, we say \(W_1,\ldots, W_r\) are \textit{H\"ormander vector fields on \(\ManifoldN\)}.
    \end{itemize}
\end{definition}

\begin{definition}
    Let \(W=\left\{ W_1,\ldots, W_r \right\}\subset \VectorFieldsN\) be H\"ormander vector fields on \(\ManifoldN\).
    Assign to each \(W_j\) a ``formal degree'' \(\Wd_j\in \Zg\). We write
    \begin{equation*}
        \WWd=\left\{ \left( W_1,\Wd_1 \right),\ldots, \left( W_r,\Wd_r \right) \right\}
    \end{equation*}
    and call \(\WWd\) \textit{H\"ormander vector fields with formal degrees}.
\end{definition}

Let \(\WWd=\left\{ \left( W_1,\Wd_1 \right),\ldots, \left( W_r,\Wd_r \right) \right\}\subset \VectorFieldsN\times \Zg\)
be H\"ormander vector fields with formal degrees on \(\ManifoldN\). 
For \(\delta>0\) we write
\(\delta^{\Wd} W=\left\{ \delta^{\Wd_1}W_1,\ldots, \delta^{\Wd_r}W_r \right\}\).
Set, for \(x\in \ManifoldN\),
\begin{equation}\label{Eqn::BasicDefns::UnitCCBall}
\begin{split}
     B_W(x):=\bigg\{
        y\in \ManifoldN \: \bigg|\:& 
        \exists \gamma:[0,1]\rightarrow \ManifoldN, \gamma(0)=x, \gamma(1)=y,
        \\&\gamma\text{ is absolutely continuous},
        \\&\gamma'(t)=\sum_{j=1}^r a_j(t) W_j(\gamma(t)) \text{ for almost every } t,
        \\&a_j\in \LpSpace{\infty}[{[0,1]}], \bigg\|\sum_{j=1}^r |a_j|^2\bigg\|_{\LpSpace{\infty}[{[0,1]}]}<1
        \bigg\},
\end{split}
\end{equation}
\begin{equation}\label{Eqn::BasicDefns::CCBallAndMetric}
    \BWWd{x}{\delta}:=B_{\delta^{\Wd}W}(x), \quad \MetricWWd[x][y]:=\inf\left\{ \delta>0 : y\in \BWWd{x}{\delta} \right\}.
\end{equation}
It follows easily from the definitions that \(\MetricWWd\) is an extended metric\footnote{An extended metric satisfies all the same
axioms as a metric, but may take the value \(\infty\).} on \(\ManifoldN\).

When \(\ManifoldN\) is a connected manifold without boundary, it is a result of Chow \cite{ChowUberSystemeVonLinearenPartiellenDifferentialgleichungenErsterOrdnung}
that \(\MetricWWd\) is a metric on \(\ManifoldN\) and the metric topology induced by \(\MetricWWd\)
corresponds with the usual topology on \(\ManifoldN\) as a manifold; see \cite[Lemma 3.1.7]{StreetMaximalSubellipticity}
for an exposition of this classical result and Theorem \ref{Thm::Metrics::Results::GivesUsualTopology} for a discussion when \(\ManifoldN\) has boundary.

\begin{definition}
    Let \(\sS\subseteq \VectorFieldsN\times \Zg\) be a set. We let \(\Gen{\sS}\subseteq \VectorFieldsN\times \Zg\) be the smallest set
    such that
    \begin{itemize}
        \item \(\sS\subseteq \Gen{\sS}\),
        \item \((X_1,\Xd_1),(X_2,\Xd_2)\in \Gen{\sS}\implies \left( [X_1,X_2],\Xd_1+\Xd_2 \right)\in \Gen{\sS}\).
    \end{itemize}
\end{definition}

\begin{definition}\label{Defn::BasicDefns::StrongWeakEquivalnce}
    Let \(\sS_1,\sS_2\subseteq \VectorFieldsN\times \Zg\). We say:
    \begin{itemize}
        \item \(\sS_1\) \textit{locally strongly controls} \(\sS_2\) on \(\ManifoldN\), if \(\forall x\in \ManifoldN\),
            there exists an open neighborhood \(U\) of \(x\) such that \(\forall (Z,\Zd)\in \sS_2\),
            \(Z\big|_U\) is in the \(\CinftySpace[U]\) module generated by \(\left\{ Y\big|_U : (Y,\Yd)\in \sS_1, \Yd\leq \Zd \right\}\).
        \item \(\sS_1\) and \(\sS_2\) are \textit{locally strongly equivalent} on \(\ManifoldN\) if \(\sS_1\) locally strongly controls \(\sS_2\) on \(\ManifoldN\)
            and \(\sS_2\) locally strongly controls \(\sS_1\) on \(\ManifoldN\).
        \item \(\sS_1\) \textit{locally weakly controls} \(\sS_2\) on \(\ManifoldN\) if \(\Gen{\sS_1}\) locally strongly controls \(\sS_2\) on \(\ManifoldN\).
        \item  \(\sS_1\) and \(\sS_2\) are \textit{locally weakly equivalent} on \(\ManifoldN\) if \(\sS_1\) locally weakly controls \(\sS_2\) and \(\sS_2\) locally weakly controls \(\sS_1\).
    \end{itemize}
\end{definition}

One main conclusion of this paper is that the important properties we discuss only depend on the local weak equivalence
class of \(\WWd\).

\begin{remark}
    Definition \ref{Defn::BasicDefns::StrongWeakEquivalnce} is more simply stated in the language of filtrations of sheaves of vector fields;
    see Section \ref{Section::Sheaves::Control}.
\end{remark}

    \subsection{Non-characteristic points}\label{Section::Defns::NonChar}
    \begin{definition}
    Let \(\WWd=\left\{ \left( W_1,\Wd_1 \right),\ldots, \left( W_r,\Wd_r \right) \right\}\subset \VectorFieldsN\times \Zg\) be H\"ormander vector fields with formal degrees
    on \(\ManifoldN\). We define
    \begin{equation*}
        \degBoundaryNWWd 
        :\BoundaryN\rightarrow \Zg
    \end{equation*}
    by
    \begin{equation*}
        \degBoundaryNWWd[x]=\min\left\{ \Zd : \exists (Z,\Zd)\in \GenWWd, Z(x)\not \in \TangentSpace{\BoundaryN}{x} \right\}.
    \end{equation*}
\end{definition}

    Because the vector fields \(W_1,\ldots, W_r\) are assumed to satisfy H\"ormander's condition, 
    \(\degBoundaryNWWd[x]\) is always finite.
    It is easy to see that \(x\mapsto \degBoundaryNWWd[x]\), \(\BoundaryN\rightarrow \Zg\) is upper semi-continuous.
Moreover, \(\degBoundaryNWWd[x]\) depends only on the local weak equivalence class of \(\WWd\).

\begin{definition}\label{Defn::BasicDefns::NoncharteristicPoint}
    We say \(x_0\in \BoundaryN\) is \textit{\(\WWd\)-non-characteristic} if \(x\mapsto \degBoundaryNWWd[x]\),
    \(\BoundaryN\rightarrow \Zg\) is continuous at \(x_0\). I.e., if \(\degBoundaryNWWd[x]\) is 
    constant on a \(\BoundaryN\)-neighborhood of \(x_0\).
\end{definition}

\begin{remark}\label{Rmk::BasicDefns::NonCharDependsOnlyOnWeakEquiv}
    Since \(\degBoundaryNWWd[x]\) depends only on the weak local equivalence class of \(\WWd\),
    whether a point is \(\WWd\)-non-characteristic also only depends on the weak local equivalence class of \(\WWd\).
\end{remark}

\begin{remark}\label{Rmk::BasicDefns::NonCharInTermsOfWj}
    A simple proof shows \(x_0\in \BoundaryN\) is \(\WWd\)-non-characteristic if and only if \(\exists j\)
    with \(\Wd_j=\degBoundaryNWWd[x_0]\) and \(W_j(x_0)\not \in \TangentSpace{\BoundaryN}{x_0}\).
\end{remark}

\begin{example}\label{Example::BasicDefns::NonCharExamples}
    \begin{enumerate}[(i)]
        \item If \(\exists (W_j,1)\in \WWd\) with \(W_{j}(x_0)\not\in \TangentSpace{\BoundaryN}{x_0}\), then \(x_0\)
            is \(\WWd\)-non-characteristic.
        \item In the case \(\WWd=\WWo=\left\{ \left( W_1,1 \right),(W_2,1),\ldots, (W_r,1) \right\}\), then \(x_0\)
            is \(\WWo\)-non-characteristic if and only if \(x_0\) is non-characteristic for the sub-Laplacian \(\sum W_j^{*}W_j\).
            This is the notion used by Jerison \cite{JerisonDirichletProblemForTheKohnLaplacianI,JerisonDirichletProblemForTheKohnLaplacianII}.
        \item Let \(\ManifoldM\) be a manifold without boundary, and let \(W_1,\ldots, W_r\) be H\"ormander vector fields on \(\ManifoldM\).
            Consider the manifold with \([0,\infty)\times \ManifoldM\) with coordinates \((t,x)\); and let
            \(\ZZd=\left\{ \left( W_1,1 \right),\ldots, \left( W_r,1 \right), \left( \partial_t,2 \right) \right\}\).
            Then every point of the form \((0,x)\in \left\{ 0 \right\}\times \ManifoldM\) is \(\ZZd\)-non-characteristic.
            This is the setting which arises when studying the heat operator \(\partial_t+\sum W_j^{*} W_j\).
        \item\label{Item::BasicDefns::NonCharExamples::Elliptic} Let \(\ManifoldN = \R^{n-1}\times [0,\infty)\) and let \(\partialo:=\left\{ \left( \partial_{x_1},1 \right), \left( \partial_{x_2},1 \right),\ldots, \left( \partial_{x_n},1 \right) \right\}\).
            Then every point of \(\BoundaryN\) is \(\partialo\)-non-characteristic. This is the setting which arises when studying elliptic operators (see \cite[Example 1.1.10(i)]{StreetMaximalSubellipticity}).
            Thus, when studying elliptic boundary value problems, every boundary point is automatically non-characteristic.
        \item Let \(\ManifoldN=\left\{ (x,y)\in \R^2 : y\geq x^2\right\}\) and \(\WWd=\left\{ \left( \partial_x,1 \right),\left( x\partial_y,1 \right) \right\}\).
            Then, \((x,x^2)\) is \(\WWd\)-non-characteristic if and only if \(x\ne 0\).
            In the second paper in this series \cite{StreetFunctionSpacesAndTraceTheoremsForMaximallySubellipticBoundaryValueProblems}
            we show that even in this simple case, the trace map has a complicated image near \(x=0\), and falls outside
            our theory.

    \end{enumerate}
\end{example}

Let \(\BoundaryNncWWd:=\left\{ x\in \BoundaryN : x\text{ is }\WWd\text{-non-characteristic} \right\}\)
which is an open subset of \(\BoundaryN\). Set \(\ManifoldNncWWd:=\BoundaryNncWWd\cup \InteriorN\) with manifold
structure given as an open submanifold of \(\ManifoldN\).
Every point of \(\BoundaryNncWWd\) is \(\WWd\)-non-characterstic. If we are only interested in results near a \(\WWd\)-non-characterstic
point of \(\BoundaryN\), we may usually replace \(\ManifoldN\) with \(\ManifoldNncWWd\) and may thereby assume that
all boundary points are \(\WWd\)-non-characterstic. However, even if \(\ManifoldN\) is compact,
\(\ManifoldNncWWd\) need not be compact.

\section{Global corollaries}\label{Section::GlobalCors}
The main results of this paper are local, taking place near a \(\WWd\)-non-characteristic point of \(\BoundaryN\).
In this section, we state some simple corollaries on certain compact manifolds which are perhaps easier to understand
on a first reading.

\noindent\textbf{Temporary\footnote{In the rest of the paper, these global assumptions are replaced with local assumptions on a possibly non-compact manifold.} Global Assumptions:} Throughout this section, \(\ManifoldN\) is a connected, compact manifold with boundary
of dimension \(n\geq 1\),
\(\WWd=\left\{ \left( W_1,\Wd_1 \right),\ldots, \left( W_r,\Wd_r \right) \right\}\subset\VectorFieldsN\times\Zg\) are H\"ormander
vector fields with formal degrees on \(\ManifoldN\), and \textbf{every point of \(\BoundaryN\) is assumed
to be \(\WWd\)-non-characteristic}.

\begin{corollary}
    \(\MetricVFs{\WWd}\) is a metric on \(\ManifoldN\) and the Lipschitz equivalence class (see Definition \ref{Defn::Metrics::Intro::Equivalent}) of this metric depends only
    on the weak local equivalence class of \(\WWd\).
\end{corollary}
\begin{proof}
    This follows from Theorems \ref{Thm::Metrics::Results::GivesUsualTopology} and \ref{Thm::Metrics::Results::LocalWeakEquivImpliesEquiv}.
\end{proof}

\begin{corollary}\label{Cor::GlobalCor::EquivTop}
    The metric topology on \(\ManifoldN\) induced by \(\MetricVFs{\WWd}\) is the same as the topology on \(\ManifoldN\)
    as a manifold.
\end{corollary}
\begin{proof}
    This follows from Theorem \ref{Thm::Metrics::Results::GivesUsualTopology}.
\end{proof}

As described in Section \ref{Section::BndryVfsWithFormaLDegrees} and Example \ref{Example::Sheaves::Metrics::CompactWithNCBdry}, we define H\"ormander vector fields with formal degrees on \(\BoundaryN\):
\(\VVd:=\left\{ \left( V_1,\Vd_1 \right),\ldots, \left( V_q,\Vd_q \right) \right\}\subset \VectorFieldsBoundaryN\times \Zg\).

\begin{corollary}\label{Cor::GlobalCor::EquivBoundary}
    \(\MetricVFs{\VVd}\) and \(\MetricVFs{\WWd}\big|_{\BoundaryN\times \BoundaryN}\) are Lipschitz equivalent
     (see Definition \ref{Defn::Metrics::Intro::Equivalent})
    on connected components of \(\BoundaryN\).
    In particular, the metric geometry on \(\BoundaryN\) induced by \(\MetricVFs{\WWd}\) is a Carnot-Carath\'eodory
    geometry.
\end{corollary}
\begin{proof}
    See Example \ref{Example::Sheaves::Metrics::CompactWithNCBdry}.
\end{proof}

For the next results, fix \(\Vol\), a strictly positive, smooth density on \(\ManifoldN\).

\begin{corollary}\label{Cor::GlobalCor::EsimateVol}
    Fix \(m\in \Zg\) such that
    \(W_1,\ldots,W_r\) satisfy H\"ormander's condition of order \(m\) on \(\ManifoldN\).
    Then, for \(x\in \ManifoldN\) and \(\delta\in (0,1]\),
    \begin{equation}\label{Eqn::GlobalCor::EsimateVol}
        \Vol[\BWWd{x}{\delta}] \approx \max\left\{ \Vol(x)\left( \delta^{\Xd_1}X_1(x),\ldots, \delta^{\Xd_n}X_n(x) \right) \: | \: \left( X_1,\Xd_1\right),\ldots,\left( X_1,\Xd_n \right)\in \GenWWd, \Xd_j\leq m \max\{\Wd_j\}  \right\},
    \end{equation}
    where on the left-hand side we are treating \(\Vol\) as a measure, and on the right-hand side as a density.
\end{corollary}
\begin{proof}
    Let \(\delta_0>0\) be as in Theorem \ref{Thm::Scaling::MainResult}. Then, for \(\delta\leq \delta_0\),
    this follows from Theorem \ref{Thm::Scaling::MainResult} \ref{Item::Scaling::MainResult::EstimateVolume}.
    For \(\delta\in [\delta_0\wedge 1, 1]\), we use that, by compactness, both sides of \eqref{Eqn::GlobalCor::EsimateVol}
    are \(\approx 1\).
\end{proof}

\begin{corollary}
    For \(x\in \ManifoldN\), \(\delta>0\),
    \(
        \Vol[\BWWd{x}{2\delta}]\lesssim \Vol[\BWWd{x}{\delta}].
    \)    
\end{corollary}
\begin{proof}
    For \(\delta\in (0,1/2]\), this follows immediately from Corollary \ref{Cor::GlobalCor::EsimateVol}.
    The result for \(\delta\geq 1/2\), follows from the next equation:
    \begin{equation}\label{Eqn::GlobalCor::VolLargeBallsIs1}
        \Vol[\BWWd{x}{\delta}]\approx 1,\quad \forall x\in \ManifoldN, \delta\geq 1/2.
    \end{equation}
    To see \eqref{Eqn::GlobalCor::VolLargeBallsIs1}, note that \(\Vol[\BWWd{x}{\delta}]\leq \Vol[\ManifoldN]\lesssim 1\), \(\forall x\in \ManifoldN\), \(\forall \delta>0\), by
    the compactness of \(\ManifoldN\).
    For the reverse inequality, we have
    \begin{equation*}
    \begin{split}
         &\inf_{\substack{x\in \ManifoldN\\ \delta\geq 1/2}} \Vol[\BWWd{x}{\delta}]
         = \inf_{x\in \ManifoldN} \Vol[\BWWd{x}{1/2}]\gtrsim 1,
    \end{split}
    \end{equation*}
    where the final inequality uses \eqref{Eqn::GlobalCor::EsimateVol} and the compactness of \(\ManifoldN\).
\end{proof}

\begin{corollary}\label{Cor::GlobalCor::InsideAndOutsideMetricsAreTheSame}
    Suppose \(\ManifoldN\) is a co-dimension \(0\) closed, embedded submanifold of a manifold without boundary \(\ManifoldM\),
    and \(\WhWd=\left\{ ( \Wh_1,\Wd_1 ),\ldots, ( \Wh_r,\Wd_r ) \right\}\subset \VectorFieldsM\times \Zg\)
    are H\"ormander vector fields with formal degrees on \(\ManifoldM\) such that \(\Wh_j\big|_{\ManifoldN}=W_j\).
    Then, the metrics \(\MetricVFs{\WWd}\) and \(\MetricVFs{\WhWd}\big|_{\ManifoldN\times\ManifoldN}\) are Lipschitz  equivalent (see Definition \ref{Defn::Metrics::Intro::Equivalent}).
\end{corollary}
\begin{proof}
    See Theorem \ref{Thm::Metrics::Results::ExtendedVFsGiveSameMetric}.
\end{proof}

\begin{remark}
    In Corollary \ref{Cor::GlobalCor::InsideAndOutsideMetricsAreTheSame}, the definitions of
    \(\MetricVFs{\WWd}\) and \(\MetricVFs{\WhWd}\big|_{\ManifoldN\times\ManifoldN}\)
    are similar. The difference is that when defining \(\MetricVFs{\WhWd}[x][y]\) (for \(x,y\in\ManifoldN\)),
    the paths are allowed to leave \(\ManifoldN\); whereas in \(\MetricVFs{\WWd}[x][y]\) the paths
    are restricted to stay inside of \(\ManifoldN\). Because of this,
    the inequality \(\MetricVFs{\WhWd}\big|_{\ManifoldN\times\ManifoldN}(x,y)\leq \MetricVFs{\WWd}[x][y]\)
    is obvious. Corollary \ref{Cor::GlobalCor::InsideAndOutsideMetricsAreTheSame} shows that not much is gained
    by allowing the paths to leave \(\ManifoldN\).
\end{remark}

Most importantly, we introduce scaling maps adapted to this geometry. See Theorem \ref{Thm::Scaling::MainResult} for a precise statement.
For \(\delta>0\), let 
\begin{equation}\label{Eqn::GlobalCor::QnDefn}
    \nCube{\delta}:=\left\{ x\in \R^n : |x|_{\infty}<\delta \right\}, \quad \nCubegeq{\delta}[c] := \left\{ x=(x_1,\ldots,x_n)\in \nCube{\delta} : x_n\geq c \right\}, \quad \nCubegeq{\delta}:=\nCubegeq{\delta}[0].
\end{equation}
Roughly speaking, for each \(x\in \ManifoldN\) and each \(\delta>0\) small, we provide
coordinate charts \(\psi_{x,\delta}\), whose domain are either \(\nUnitCube\) or \(\nUnitCubegeq[-c]\) for some \(c>0\),
and whose image is comparable to \(\BWWd{x}{\delta}\). Moreover,
\(\psi_{x,\delta}^{*}\delta^{\Wd_1}W_1,\ldots, \psi_{x,\delta}^{*}\delta^{\Wd_r}W_r\) are H\"ormander vector fields, uniformly for \(x\in \ManifoldN\) 
and \(\delta>0\) small.
Informally, this allows us to rescale small Carnot--Carath\'eodory scales (either on the interior or at a \(\WWd\)-non-characteristic point
on the boundary) to turn them into the unit scale, where quantitative methods can be applied.

\begin{remark}
    In this paper, whenever we use 
    the cubes
    \(\nCube{\eta}\) and \(\nCubegeq{\eta}\), or the ball analogs
    \(\nBall{\eta}=\{x\in \Rn : |x|<\eta\}\) and \(\nBallgeq{\eta}=\{x\in \nBall{\eta}:x_n\geq 0\}\), 
    we think of \(\eta\) as \(\approx 1\) and treat this as the ``unit-scale.''
    The choice between using the balls and cubes is always done for convenience and is not a central point in any of our results.
\end{remark}

\section{Boundary vector fields with formal degrees}\label{Section::BndryVfsWithFormaLDegrees}
Let \(\ManifoldN\) be a smooth manifold with boundary, and let
\(\WWd=\left\{ \left( W_1,\Wd_1 \right),\ldots, \left( W_r,\Wd_r \right) \right\}\subset \VectorFieldsN\times \Zg\)
be H\"ormander vector fields with formal degrees on \(\ManifoldN\).
Fix \(x_0\in \BoundaryNncWWd\). We define a list of vector fields with formal degrees,
\(\VVd\), near \(x_0\) on \(\BoundaryN\) by the following procedure:

\begin{enumerate}[(1)]
    \item Let \(\Omega_1\subset\ManifoldN\) be an open \(\ManifoldN\)-neighborhood of \(x_0\) and \(m\in \Zg\) 
        be such that \(W_1,\ldots,W_r\)
        satisfy H\"ormander's condition of order \(m\) on \(\Omega_1\).

    \item\label{Item::BoundaryVfs::DefineZZd} Let 
    \begin{equation*}
        \ZZd:=\left\{ \left( Z_1,\Zd_1 \right),\ldots, \left( Z_q,\Zd_q \right) \right\}
        :=\left\{ \left( Y,e \right)\in \GenWWd : e\leq m\max\{\Wd_1,\ldots, \Wd_r\} \right\}.
    \end{equation*}
    By H\"ormander's condition, \(Z_1(x),\ldots, Z_q(x)\) span \(\TangentSpace{\ManifoldN}{x}\), \(\forall x\in \Omega_1\).
    Note that \(\WWd\subseteq \ZZd\) and
    \begin{equation}\label{Eqn::BoundaryVfs::NSW::Z}
        [Z_j,Z_k]=\sum_{\Zd_l\leq \Zd_j+\Zd_k} c_{j,k}^l Z_l,\quad c_{j,k}^l\in \CinftySpace[\Omega_1].
    \end{equation}
    Indeed, when \(\Zd_j+\Zd_k \leq m\max\{\Wd_1,\ldots, \Wd_r\}\), then \(\left( [Z_j,Z_k],\Zd_j+\Zd_k \right)\in \ZZd\)
    and \eqref{Eqn::BoundaryVfs::NSW::Z} is obvious. When \(\Zd_j+\Zd_k >  m\max\{\Wd_1,\ldots, \Wd_r\}\) then we
    use that \(Z_1(x),\ldots, Z_q(x)\) span \(\TangentSpace{\ManifoldN}{x}\), \(\forall x\in \Omega_1\),
    and \eqref{Eqn::BoundaryVfs::NSW::Z} follows.

    \item\label{Item::BoundaryVfs::FindWj0} Since \(x_0 \in\BoundaryNncWWd\), \(\exists j_0\) with \(\Wd_{j_0}=\degBoundaryNWWd[x_0]\) and \(W_{j_0}(x_0)\not \in \TangentSpace{\BoundaryN}{x_0}\); see Remark \ref{Rmk::BasicDefns::NonCharInTermsOfWj}.
        Let \(\Omega_2\subseteq \Omega_1\) be an open \(\ManifoldN\)-neighborhood of \(x_0\) so small that
        \(\forall x'\in \Omega_2\cap \BoundaryN\) we have:
        \begin{itemize}
            \item \(W_{j_0}(x)\not \in \TangentSpace{\BoundaryN}{x}\),
            \item \(\Wd_{j_0}=\degBoundaryNWWd[x]\).
        \end{itemize}

    \item Set \(X_0:=W_{j_0}\), \(\Xd_0:=\Wd_{j_0}\).
    \item By construction, \(\exists! b_j\in \CinftySpace[\Omega_2\cap \BoundaryN]\) such that
        \begin{equation*}
            Z_j(x')-b_j(x')X_0(x')\in \TangentSpace{\BoundaryN}{x'},\quad \forall x'\in \Omega_2\cap \BoundaryN.
        \end{equation*}
        Note that \(b_j\equiv 0\) if \(\Zd_j<\Xd_0=\Wd_{j_0}=\degBoundaryNWWd[x']\).
    
    \item Extend each \(b_j\) to a function \(\bt_j\in \CinftySpace[\Omega_2]\) with \(\bt_j\equiv 0\) if \(\Zd_j<\Xd_0\).
    
    \item\label{Item::BoundaryVfs::DefineXj} Set, for \(1\leq j\leq q\), \(\Xd_j:=\Zd_j\) and
        \begin{equation*}
            X_j(x):=Z_j(x)-\bt_j(x)X_0(x)
            =\begin{cases}
                Z_j(x)-\bt_j(x) X_0(x), & \Zd_j\geq \Xd_0,\\
                Z_j(x), & \Zd_j<\Xd_0,
            \end{cases}
        \end{equation*}
        so that \(X_j\in \VectorFields{\Omega_2}\)
        and \(X_j(x)\in \TangentSpace{\BoundaryN}{x'}\), \(\forall x'\in \Omega_2\cap \BoundaryN\). One of the vector fields \(X_j\)  is identically \(0\) (the one corresponding to \(Z_j=W_{j_0}\)), but this is not important for what follows.
        Set \(\XXd=\left\{ \left( X_0,\Xd_0 \right),\left( X_1,\Xd_1 \right),\ldots, \left( X_q,\Xd_q \right) \right\}\).
        Note that \(\XXd\) and \(\ZZd\) are locally strongly equivalent on \(\Omega_2\).
        It follows from this strong equivalence 
        that \(X_0(x),\ldots, X_q(x)\) span \(\TangentSpace{\ManifoldN}{x}\), \(\forall x\in \Omega_2\)
        and using \eqref{Eqn::BoundaryVfs::NSW::Z} 
        \begin{equation}\label{Eqn::BoundaryVfs::NSW::X}
            [X_j,X_k]=\sum_{\Xd_l\leq \Xd_j+\Xd_k} \ct_{j,k}^l X_l,\quad \ct_{j,k}^l\in \CinftySpace[\Omega_2].
        \end{equation}

    \item\label{Item::BoundaryVfs::DefineVj} Set, for \(1\leq j\leq q\), \(\Vd_j:=\Xd_j\),
        \begin{equation*}
            V_j:=X_j\big|_{\Omega_2\cap \BoundaryN}\in \VectorFields{\Omega_2\cap\BoundaryN},
        \end{equation*}
        and \(\VVd:=\left\{ \left( V_1,\Vd_1 \right),\ldots, \left( V_q,\Vd_q \right) \right\}\).
        For \(x'\in \Omega_2\cap \BoundaryN\),
        \begin{equation*}
            \dim\Span \left\{ V_1(x'),\ldots, V_q(x') \right\}\geq \dim \Span\left\{ X_0(x'),\ldots, X_q(x') \right\}-1=\dim \TangentSpace{\ManifoldN}{x'}-1=\dim \TangentSpace{\BoundaryN}{x'},
        \end{equation*}
        and therefore \(V_1(x'),\ldots, V_q(x')\) span \(\TangentSpace{\BoundaryN}{x'}\), \(\forall x'\in \Omega_2\cap \BoundaryN\).

    \item In light of \eqref{Eqn::BoundaryVfs::NSW::X}, the fact that \([V_j,V_k](x')\in \TangentSpace{\BoundaryN}{x'}\), \(\forall x'\in \Omega_2\cap \BoundaryN\),
        and that \(X_0(x')\) is the only \(X_j(x')\) which is not tangent to \(\BoundaryN\) at each \(x'\in \Omega_2\cap\BoundaryN\), we see
        we have
        \begin{equation}\label{Eqn::BoundaryVfs::NSW::V}
            [V_j,V_k]=\sum_{\Vd_l\leq \Vd_j+\Vd_k} \ct_{j,k}^l\big|_{\Omega_2\cap\BoundaryN} V_l.
        \end{equation}
\end{enumerate}

\begin{remark}
    The local strong equivalence class of \(\VVd\) does not depend on any choices in the above construction,
    and depends only on the local weak equivalence class of \(\WWd\). Indeed, if one replaced \(\ZZd\)
    with a different locally strongly equivalent list on a neighborhood of \(x_0\)  and
    went through the above process, then the corresponding \(\XXd\) and \(\VVd\) would be locally strongly equivalent to the
    ones constructed above (on appropriate neighborhoods of \(x_0\)).
    This is perhaps better understood in the language of sheaves; see Construction \ref{Construction::Sheaves::Restrict::Bndry::VVd}.
\end{remark}

\begin{remark}
    Suppose \(Y\in \VectorFields{\Omega_2}\), \(e\in \Zgeq\), \(Y=\sum_{\Zd_l\leq e} f_l Z_l\), \(f_l\in \CinftySpace[\Omega_2]\),
    and \(Y(x')\in \TangentSpace{\BoundaryN}{x'}\), \(\forall x'\in \Omega_2\cap \BoundaryN\).
    Then,
    \begin{equation}\label{Eqn::BoundaryVFs::RestrictGenerators}
        Y\big|_{\Omega_2\cap\BoundaryN} = \sum_{\Vd_l \leq e} g_l V_l, \quad g_l\in \CinftySpace[\Omega_2\cap \BoundaryN].
    \end{equation}
    Indeed, by the construction of \(\XXd\), we have
    \(Y=\sum_{\Xd_l\leq e} h_l X_l\), \(h_l\in \CinftySpace[\Omega_2]\). Using that \(Y(x')\in \TangentSpace{\BoundaryN}{x'}\), \(\forall x'\in \Omega_2\cap \BoundaryN\)
    and \(X_0(x')\) is the only \(X_j(x')\) which is not tangent to \(\BoundaryN\) at each \(x'\in \Omega_2\cap\BoundaryN\),
    we have 
    \begin{equation*}
        Y\big|_{\Omega_2\cap\BoundaryN} = \sum_{\Vd_l \leq e} h_l\big|_{\Omega_2\cap \BoundaryN} V_l,
    \end{equation*}
    establishing \eqref{Eqn::BoundaryVFs::RestrictGenerators}.
\end{remark}

\begin{remark}
    The above procedure and definitions could have been completed on any co-dimension \(1\) embedded submanifold 
    of \(\ManifoldN\); it need not be the boundary.
\end{remark}

\section{Scaling}\label{Section::Scaling}
The heart of this paper are scaling maps which behave well near a non-characteristic point of the boundary.
In this section, we state and prove these scaling results. Roughly speaking, we give scaling maps
which turn small Carnot--Carath\'eodory scales into the ``unit scale.''
In Section \ref{Section::Scaling::UnitScale} we describe the unit scale before stating the main scaling
theorem (Theorem \ref{Thm::Scaling::MainResult}) in Section \ref{Section::Scaling::MainResult}. The remainder of this section is devoted to the proof
of Theorem Theorem \ref{Thm::Scaling::MainResult} and establishing some other technical points.

    \subsection{The unit scale}\label{Section::Scaling::UnitScale}
Our main scaling result (Theorem \ref{Thm::Scaling::MainResult}) allows us to rescale H\"ormander vector fields at a small
scale, so that they are H\"ormander vector fields at the unit scale ``uniformly'' in various parameters.
This uniformity is explained in this section.

For \(U\subseteq \Rn\) open
and \(L\in \Zgeq=\left\{ 0,1,2,\ldots \right\}\), set
\begin{equation*}
    \CmbSpace{L}[U]:=\left\{ f\in \CmSpace{L}[U] : \partial_x^{\alpha}\text{ is bounded }\forall |\alpha|\leq L \right\},
\end{equation*}
\begin{equation*}
    \CmbNorm{f}{L}[U]:=\sup_{x\in U}\max_{|\alpha|\leq L} \left| \partial_x^{\alpha}f(x) \right|.
\end{equation*}
Set \(\CinftybSpace[U]:=\bigcap_{L}\CmbSpace{L}[U]\), which we give the usual Fr\'echet topology.
We similarly define the vector valued analogs \(\CmbSpace{L}[U][\R^m]\) and \(\CinftybSpace[U][\R^m]\).

For a vector field \(X\in \VectorFields{U}\), we write \(X\) with respect to the standard basis,
\(X(x)=\sum a_j(x) \partial_{x_j}\), and identify \(X\) with the function \(x\mapsto (a_1(x),\ldots, a_n(x))\in \CinftySpace[U][\R^n]\).
It therefore makes sense to ask if \(X\in \CinftybSpace[U][\R^n]\) and to consider the norms
\(\CmbNorm{X}{L}[U][\R^n]\).

\begin{definition}\label{Defn::Scaling::UnitScale::UniformSpan}
    Let \(\sI\) be an index set, and for \(\iota \in \sI\) let \(X_1^{\iota},\ldots, X_{q_\iota}^{\iota}\in \CinftySpace[U][\R^n]\)
    be smooth vector fields on \(U\). We say \(X_1^{\iota},\ldots, X_q^{\iota}\) \textit{span the tangent space to \(U\), uniformly for \(\iota\in \sI\)}
    if
    \begin{equation*}
        \inf_{\iota\in \sI} \inf_{x\in U} \max_{j_1,\ldots,j_n\in \{1,\ldots, q_{\iota}\}} \left| \det\left( X_{j_1}^{\iota}(x)| X_{j_2}^{\iota}(x)|\cdots| X_{j_n}^{\iota}(x) \right) \right|>0,
    \end{equation*}
    where \(\left( X_{j_1}^{\iota}(x)| X_{j_2}^{\iota}(x)|\cdots| X_{j_n}^{\iota}(x) \right)\) is the \(n\times n\) matrix whose \(k\)-th column is \(X_{j_k}^{\iota}(x)\)
    written as a column vector in the standard basis.
\end{definition}

\begin{definition}\label{Defn::Scaling::UnitScale::UniformHormander}
    Let \(\sI\) be an index set, and for \(\iota\in \sI\) let \(W_1^{\iota},\ldots, W_{r_\iota}^\iota\in \CinftybSpace[U][\R^n]\)
    be smooth vector fields on \(U\). We say \(W_1^{\iota},\ldots, W_{r_\iota}^{\iota}\) are \textit{H\"ormander vector fields on \(U\),
    uniformly for \(\iota\in \sI\)} if:
    \begin{itemize}
        \item \(\sup_{\iota\in \sI} r_\iota<\infty\).
        \item \(\left\{ W_j^{\iota} : \iota\in \sI, 1\leq j\leq r_\iota \right\}\subseteq \CinftybSpace[U][\R^n]\) is a bounded set.
        \item \(\exists m\in \Zg\) (independent of \(\iota\)) such that if \(X_1^{\iota},\ldots, X_{q_\iota}^{\iota}\) is
            an enumeration of the commutators of \(W_1^{\iota},\ldots, W_{r_\iota}^{\iota}\) up to order \(m\),
            then \(X_1^{\iota},\ldots, X_{q_\iota}^{\iota}\) span the tangent space to \(U\), uniformly for \(\iota\in \sI\).
    \end{itemize}
\end{definition}

\begin{remark}
    The main relevance of Definition \ref{Defn::Scaling::UnitScale::UniformHormander} is that most standard proofs
    regarding H\"ormander vector fields go through ``uniformly'' in \(\iota\) when the hypotheses of
    Definition \ref{Defn::Scaling::UnitScale::UniformHormander} are satisfied.
\end{remark}

    \subsection{The main scaling result}\label{Section::Scaling::MainResult}
    Let \(\ManifoldN\) be a smooth manifold with boundary, let
\(\WWd=\left\{ \left( W_1, \Wd_1 \right),\ldots, \left( W_r,\Wd_r \right) \right\}\subset \VectorFieldsN\times \Zg\)
be H\"ormander vector fields with formal degrees on \(\ManifoldN\), and let \(\Vol\)
be a smooth, strictly positive density on \(\ManifoldM\).

Let \(\ManifoldM\) be a smooth manifold (without boundary)
with smooth, strictly positive density \(\Volh\) such that \(\ManifoldN\) is a co-dimension \(0\), closed,
embedded submanifold of \(\ManifoldM\) with \(\Vol=\Volh\big|_{\ManifoldN}\), and such that there are H\"ormander vector fields with formal degrees
\(\WhWd=\left\{ ( \Wh_1,\Wd_1 ),\ldots, ( \Wh_r,\Wd_r ) \right\}\subset \VectorFieldsM\times \Zg\)
satisfying \(\Wh_j\big|_{\ManifoldN}=W_j\).
Such a \(\ManifoldM\), \(\WhWd\), and \(\Volh\) always exist; see Remark \ref{Rmk::Scaling::MainResult::DontNeedAmbientMfld}. Thus, even if they are not given in an application,
the parts of Theorem \ref{Thm::Scaling::MainResult} which only refer to \(\WWd\) and \(\ManifoldN\) still apply;
see Remark \ref{Rmk::Scaling::MainResult::RestrictingHats} for further comments.

For \(\alpha=(\alpha_1,\ldots, \alpha_L)\in \{1,\ldots, r\}^L\), we write
\begin{equation}\label{Eqn::Scaling::OrderedMultiIndex}
    W^{\alpha}=W_{\alpha_1}W_{\alpha_2}\cdots W_{\alpha_r}, \quad |\alpha|=L.
\end{equation}

Let \(\Compact\Subset \ManifoldNncWWd\) be compact\footnote{Here, and in the rest of the paper, we write
\(A\Subset B\) to mean that \(A\) is a relatively compact subset of \(B\).}
 and let \(\Omega\Subset\ManifoldM\) be an open, relatively compact
subset of \(\ManifoldM\) with \(\Compact\Subset \Omega\).
Let \(m\in \Zg\) be such that \(W_1,\ldots, W_r\) satisfy H\"ormander's condition of order \(m\)
on \(\overline{\Omega}\), and set for \(x\in \Omega\), \(\delta>0\),
\begin{equation*}
    \Lambda(x,\delta):=\max\left\{ \Vol(x)\left( \delta^{\Xd_1}X_1(x),\ldots, \delta^{\Xd_n}X_n(x) \right) \: | \: \left( X_1,\Xd_1\right),\ldots,\left( X_n,\Xd_n \right)\in \GenWWd, \Xd_j\leq m \max\{\Wd_j\}  \right\}.
\end{equation*}

\begin{remark}
    By definition (see \eqref{Eqn::BasicDefns::UnitCCBall} and \eqref{Eqn::BasicDefns::CCBallAndMetric}), \(\BWWd{x}{\delta}\subseteq\ManifoldN\) while \(\BWhWd{x}{\delta}\subseteq \ManifoldM\)--the balls
    \(\BWhWd{x}{\delta}\) are allowed to go outside of \(\ManifoldN\).
    We also have \(\BWWd{x}{\delta}\subseteq \BWhWd{x}{\delta}\cap \ManifoldN\), though this may be a strict containment.
\end{remark}
Let \(\nCube{\delta}\) and \(\nCubegeq{\delta}[c_0]\) be as in \eqref{Eqn::GlobalCor::QnDefn}
and set \(\nCubeeq{\delta}[c_0]=\left\{ x=(x_1,\ldots, x_n)\in \nCube{\delta} : x_n=c_0 \right\}\). Note that if
\(c_0\leq -1\), then \(\nCubegeq{1}[c_0]=\nCube{1}\) and \(\nCubeeq{1}[c_0]=\emptyset\).

We write \(A\lesssim B\) for \(A\leq CB\) where \(C\geq 1\) does not depend on the variables \(x\) and \(\delta\).
We write \(A\approx B\) for \(A\leq C B\) and \(B\leq C A\).

\begin{theorem}\label{Thm::Scaling::MainResult}
    \(\exists \delta_1\in (0,1]\), \(\forall x\in \Compact\), \(\forall \delta\in (0,\delta_1]\),
    \(\exists \Psi_{x,\delta}:\nUnitCube\rightarrow \BWhWd{x}{\delta}\) such that:
    \begin{enumerate}[(a)]
        \item\label{Item::Scaling::MainResult::InOmega} \(\forall x\in \Compact\), \(\delta\in (0,\delta_1]\), \(\BWhWd{x}{\delta}\subseteq \Omega\).
        \item\label{Item::Scaling::MainResult::EstimateVolume}  \(\forall x\in \Compact\), \(\forall \delta\in (0,\delta_1]\),
            \begin{equation*}
                \Vol[\BWWd{x}{\delta}] \approx \Lambda(x,\delta).
            \end{equation*}

            \item\label{Item::Scaling::MainResult::EstimateVolumeWedge1}  \(\forall x\in \Compact\), \(\forall \delta>0\),
            \begin{equation*}
                \Vol[\BWWd{x}{\delta}]\wedge 1 \approx \Lambda(x,\delta)\wedge 1\approx \Vol[\BWWd{x}{\delta\wedge \delta_1}].
            \end{equation*}

        \item\label{Item::Scaling::MainResult::Doubling} \(\forall x\in \Compact\), \(\forall \delta\in (0,\delta_1]\),
            \begin{equation*}
                \Vol[\BWWd{x}{2\delta}] \lesssim  \Vol[\BWWd{x}{\delta}].
            \end{equation*}
        
        \item\label{Item::Scaling::MainResult::DoublingWedge} \(\forall x\in \Compact\), \(\forall \delta>0\),
            \begin{equation*}
                \Vol[\BWWd{x}{2\delta}]\wedge 1 \lesssim  \Vol[\BWWd{x}{\delta}]\wedge 1,\quad 
                \Vol[\BWWd{x}{(2\delta)\wedge \delta_1}] \lesssim  \Vol[\BWWd{x}{\delta\wedge \delta_1}].
            \end{equation*}
        
        \item\label{Item::Scaling::MainResult::ImageOf0} \(\Psi_{x,\delta}(0)=x\).
        \item\label{Item::Scaling::MainResult::Diffeo} \(\Psi_{x,\delta}(\nUnitCube)\subseteq \Omega\) is open and
            \(\Psi_{x,\delta}:\nUnitCube\xrightarrow{\sim}\psi_{x,\delta}(\nUnitCube)\) is a \(\CinftySpace\)-diffeomorphism.
        
        \item\label{Item::Scaling::MainResult::Existencec0} \(\forall x\in \Compact\), \(\forall \delta\in (0,\delta_1]\), \(\exists c_0=c_0(x,\delta)\in [-1,0]\) such that
            \begin{equation*}
                \Psi_{x,\delta}(\nUnitCube)\cap \ManifoldN=\Psi_{x,\delta}(\nUnitCubegeq[c_0]),
                \quad \Psi_{x,\delta}(\nUnitCube)\cap \BoundaryN=\Psi_{x,\delta}(\nUnitCubeeq[c_0]).
            \end{equation*}
            Here, \(c_0=-1\) corresponds to the case \(\Psi_{x,\delta}(\nUnitCube)\subseteq \InteriorN\)
            and \(\Psi_{x,\delta}(\nUnitCube)\cap \BoundaryN=\emptyset\), and
            \(c_0=1\) corresponds to the case \(x\in \BoundaryN\).
        
        \item\label{Item::Scaling::MainResult::Containments} \(\forall \eta_1\in (0,1/2)\), \(\exists \xi_1\in (0,1]\), \(\forall x\in \Compact\), \(\forall \delta\in (0,\delta_1]\),
            \begin{equation*}
                \BWhWd{x}{\xi_1\delta}\subseteq \Psi_{x,\delta}(\nCube{\eta_1})\subseteq \Psi_{x,\delta}(\nUnitCube)\subseteq \BWhWd{x}{\delta},
            \end{equation*}
            and with \(c_0=c_0(x,\delta)\) as in \ref{Item::Scaling::MainResult::Existencec0},
            \begin{equation*}
                \BWhWd{x}{\xi_1\delta}\cap \ManifoldN \subseteq \Psi_{x,\delta}(\nCubegeq{\eta_1}[c_0])
                \subseteq \Psi_{x,\delta}(\nUnitCubegeq[c_0]) \subseteq \BWWd{x}{\delta}.
            \end{equation*}
    \end{enumerate}
    Let \(\Wh_j^{x,\delta}:=\Psi_{x,\delta}^{*}\delta^{\Wd_j}\Wh_j\). 
    \begin{enumerate}[(a),resume]
        \item\label{Item::Scaling::MainResults::UniformHormander} \(\Wh_1^{x,\delta},\ldots, \Wh_{r}^{x,\delta}\) are H\"ormander vector fields on \(\nUnitCube\), uniformly
            for \(x\in \Compact\) and \(\delta\in (0,\delta_1]\) (see Definition \ref{Defn::Scaling::UnitScale::UniformHormander}).

        \item\label{Item::Scaling::MainResults::PulledBackIsPartialxn} \(\exists D\geq 1\), \(\forall x\in \Compact\), \(\forall \delta\in (0,\delta_1]\), if
            \(\Psi_{x,\delta}(\nUnitCube)\cap \BoundaryN\neq \emptyset\), then \(\exists j_0=j_0(x,\delta)\in\{1,\ldots,r\}\) with
            \begin{equation*}
                \Wh_{j_0}^{x,\delta}=\Psi_{x,\delta}^{*} \delta^{\Wd_{j_0}} \Wh_{j_0}=\pm D \partial_{t_n}.
            \end{equation*}

    \end{enumerate}
    Define \(h_{x,\delta}\in \CinftySpace[\nCube{1}]\) by \(\Psi_{x,\delta}^{*}\Volh=h_{x,\delta} \Lambda(x,\delta)\LebDensity\),
    where \(\LebDensity\) denotes the usual Lebesgue density on \(\Rn\).
    \begin{enumerate}[(a),resume]
        \item\label{Item::Scaling::MainResults::hxdeltaSmooth} \(\left\{ h_{x,\delta} : x\in \Compact, \delta\in (0,1] \right\}\subset \CinftybSpace[\nUnitCube]\) is a bounded set.
        \item\label{Item::Scaling::MainResults::hxdeltaPositive} \(\inf_{x\in \Compact} \inf_{\delta\in (0,\delta_1]} \inf_{t\in \nUnitCube} h_{x,\delta}(t)>0\).
    \end{enumerate}
\end{theorem}

\begin{remark}\label{Rmk::Scaling::MainResult::RestrictingHats}
    Let \(c_0=c_0(x,\delta)\) as in \ref{Item::Scaling::MainResult::Existencec0}.
    Note that \(\Psi_{x,\delta}\big|_{\nUnitCubegeq[c_0]}^{*} \delta^{\Wd_j}W_j = \Wh_j^{x,\delta}\big|_{\nUnitCubegeq[c_0]}\)
    and \(\Psi_{x,\delta}\big|_{\nUnitCubegeq[c_0]}^{*} \Vol = \Psi_{x,\delta}^{*} \Volh\big|_{\nUnitCubegeq[c_0]} = h_{x,\delta}\big|_{\nUnitCubegeq[c_0]} \Lambda(x,\delta)\LebDensity\).
    Thus, \ref{Item::Scaling::MainResults::UniformHormander} and \ref{Item::Scaling::MainResults::PulledBackIsPartialxn}
    imply corresponding results about \(W_j\) and \ref{Item::Scaling::MainResults::hxdeltaSmooth}
    and \ref{Item::Scaling::MainResults::hxdeltaPositive} imply corresponding results about \(\Vol\).
    This is especially useful when combined with Remark \ref{Rmk::Scaling::MainResult::DontNeedAmbientMfld}.
\end{remark}

\begin{remark}\label{Rmk::Scaling::MainResult::DontNeedAmbientMfld}
    Theorem \ref{Thm::Scaling::MainResult} assumes the existence of an ambient manifold \(\ManifoldM\)
    with corresponding ambient vector fields \(\Wh_j\) and density \(\Volh\). In some applications, these
    are given. However, in some applications one is only given the data on \(\ManifoldN\).
    In this case, Theorem \ref{Thm::Scaling::MainResult} still applies: there always exists an appropriate choice
    of \(\ManifoldM\), \(\Wh_1,\ldots, \Wh_r\), and 
    \(\Volh\).
    To see this, let \(\ManifoldMh\) be any smooth manifold without boundary such that \(\ManifoldN\subseteq \ManifoldM\)
    is a co-dimension \(0\) closed embedded submanifold (for example, \(\ManifoldMh\) could be the double of \(\ManifoldN\)).
    Then, extend \(W_1,\ldots, W_r\) to smooth vector fields \(\Wh_1,\ldots, \Wh_r\) on \(\ManifoldMh\)
    and extend \(\Vol\) to a smooth density \(\Volh\) on \(\ManifoldMh\). By continuity, \(\Wh_1,\ldots, \Wh_r\)
    satisfy H\"ormander's condition on a neighborhood of \(\ManifoldN\) and \(\Volh\) is a strictly positive
    on a neighborhood of \(\ManifoldN\). Letting \(\ManifoldM\) be this common neighborhood, Theorem \ref{Thm::Scaling::MainResult}
    applies.
    This is particularly useful in combination with Remark \ref{Rmk::Scaling::MainResult::RestrictingHats}.
\end{remark}

\begin{remark}
    \ref{Item::Scaling::MainResults::PulledBackIsPartialxn} is only established in the case 
    \(\Psi_{x,\delta}(\nUnitCube)\cap \BoundaryN\neq \emptyset\). It is straight-forward to 
    modify the proof and establish
    the same result in the case \(\Psi_{x,\delta}(\nUnitCube)\cap \BoundaryN= \emptyset\);
    however, there is no reason to distinguish the variable \(t_n\) in that case, so such a conclusion
    is usually not relevant.
\end{remark}

The proof of Theorem \ref{Thm::Scaling::MainResult} is separated into two parts:
when \(x\) is near the boundary (relative to \(\delta\)) and when \(x\) is far away.
When \(x\) is far from the boundary,
the boundary can be ignored and  Theorem \ref{Thm::Scaling::MainResult} has
been established by Nagel, Stein, and Wainger \cite{NagelSteinWaingerBallsAndMetricsDefinedByVectorFieldsI}
and the author and Stovall \cite{StovaStreetCoordinatesAdaptedToVectorFieldsCanonicalCoordinates}; this is described
in Section \ref{Section::Scaling::WithoutBoundary}.

    \subsection{Scaling on manifolds without boundary}\label{Section::Scaling::WithoutBoundary}
    In this section, we review scaling results on manifolds without boundary.
The original results of this type were proved by Nagel, Stein, and Wainger \cite{NagelSteinWaingerBallsAndMetricsDefinedByVectorFieldsI};
however, we more closely follow the work of the author and Stovall \cite{StovaStreetCoordinatesAdaptedToVectorFieldsCanonicalCoordinates}
as presented in \cite[Section 3.5]{StreetMaximalSubellipticity}.

Let \(\ManifoldM\) be a smooth manifold without boundary, \(\Vol\) a smooth, strictly positive density on \(\ManifoldM\),
and \(\WWd=\left\{ ( W_1,\Wd_1 ),\ldots, ( W_r,\Wd_r ) \right\}\subset \VectorFieldsM\times \Zg\)
H\"ormander vector fields with formal degrees on \(\ManifoldM\).

Fix \(\Compact\Subset\Omega\Subset\ManifoldM\) open with 
\(\Compact\) compact and \(\Omega\) open and relatively compact in \(\ManifoldM\).
Let \(\XhXd=\left\{ ( X_1,\Xd_1 ),\ldots,( X_q,\Xd_q ) \right\}\subset \VectorFields{\Omega}\times \Zg\) satisfy:
\begin{itemize}
    \item \(\forall x\in \Omega\), \(\Span\left\{ X_1(x),\ldots, X_q(x) \right\}=\TangentSpace{\ManifoldM}{x}\),
    \item \(\left[ X_j,X_k \right]=\sum_{\Xd+l\leq \Xd_j+\Xd_k} c_{j,k}^l X_l\), \(c_{j,k}^l\in \CinftySpace[\Omega]\),
    \item \(\XXd\) locally strongly controls \(\WWd\) on \(\Omega\),
    \item \(\WWd\) locally weakly controls \(\WWd\) on \(\Omega\).
\end{itemize}

\begin{remark}\label{Rmk::Scaling::WithoutBoundary::XXdAlwaysExists}
    Such a choice of \(\XXd\) always exists. Indeed, take \(m\in \Zg\) such that
    \(W_1,\ldots, W_r\) satisfy H\"ormader's condition of order \(m\) on \(\Omega\).
    Then, we may take
    \begin{equation*}
        \XXd = \left\{ (Z,e)\in \GenWWd : e\leq m\max\{\Wd_j\} \right\}.
    \end{equation*}
    See the proof of \cite[Lemma 3.3.5]{StreetMaximalSubellipticity}.
\end{remark}

Set
\begin{equation*}
    \Lambda(x,\delta):= \max_{j_1,\ldots, j_n\in \left\{ 1,\ldots, q \right\}}
    \Vol(x)\left( \delta^{\Xd_{j_1}}X_{j_1}(x),\ldots, \delta^{\Xd_{j_n}}X_{j_n}(x) \right).
\end{equation*}
Since \(X_1(x),\ldots, X_q(x)\) span \(\TangentSpace{\ManifoldM}{x}\), \(\forall x\in \Omega\),
and \(\Vol\) is a strictly positive density, it follows that \(\Lambda(x,\delta)>0\) \(\forall x\in \Omega\), \(\delta>0\).

Fix \(\zeta\in (0,1]\) and for each \(x\in \Compact\), \(\delta\in (0,1]\)
pick \(j_1=j_1(x,\delta),\ldots, j_n=j_n(x,\delta)\in \{1,\ldots,q\}\)
such that \(\Span\left\{ X_{j_1}(x),\ldots, X_{j_n}(x) \right\}=\TangentSpace{\ManifoldM}{x}\)
and moreover,
\begin{equation}\label{Eqn::Scaing::WithoutBoundary::PickingXBasis}
    \max_{k_1,\ldots, k_n\in {1,\ldots, q}}
    \left| \frac{\delta^{\Xd_{k_1}}X_{k_1}(x)\wedge \cdots \wedge \delta^{\Xd_{k_n}}X_{k_n}(x) }{ \delta^{\Xd_{j_1}}X_{j_1}(x)\wedge\cdots\wedge \delta^{\Xd_{j_n}}X_{j_n}(x)  } \right|
    \leq \zeta^{-1}.
\end{equation}
See \cite[Definition 3.6.2]{StreetMaximalSubellipticity} for an explanation of the quotient in \eqref{Eqn::Scaing::WithoutBoundary::PickingXBasis}.
It is always possible to pick \(j_1(x,\delta),\ldots, j_n(x,\delta)\) so that the left-hand side of \eqref{Eqn::Scaing::WithoutBoundary::PickingXBasis}
equals \(1\); however, it is sometimes convienient to take \(\zeta<1\).

We write \(A\lesssim B\) for \(A\leq CB\) where \(C\geq 1\) is a constant which does not depend on \(x\) or \(\delta\);
we write \(A\approx B\) for \(A\lesssim B\) and \(B\lesssim A\).

\begin{theorem}\label{Thm::Scaling::WithoutBoundary::MainThm}
    \(\exists \delta_0>0\) such that
    \begin{enumerate}[(a)]
        \item\label{Item::Scaling::WithoutBdry::BallsInOmega} \(\forall x\in \Compact\), \(\forall \delta\in (0,\delta_0]\), 
        \begin{equation*}
            \bigcup_{\substack{x\in \Compact\\ \delta\in (0,\delta_0]}}\BWWd{x}{\delta}\cup \BXXd{x}{\delta}\Subset \Omega.
        \end{equation*}
        \item \(\forall x\in \Compact\), \(\delta\in (0,\delta_0]\),
            \begin{equation*}
                \Vol[\BWWd{x}{\delta}]\approx \Vol[\BXXd{x}{\delta}]\approx \Lambda(x,\delta).
            \end{equation*}
    \end{enumerate}
    For each \(x\in \Compact\) and \(\delta\in (0,\delta_0]\), there exists \(\Phi_{x,\delta}:\nUnitBall\rightarrow \BWWd{x}{\delta}\cap \BXXd{x}{\delta}\)
    and \(\xi_0\in (0,1)\) such that
    \begin{enumerate}[(a),resume]
        \item\label{Item::Scaling::WithoutBdry::PhixdeltaOf0} \(\Phi_{x,\delta}(0)=x\).
        \item \(\Phi_{x,\delta}(\nUnitBall)\subseteq \Omega\) is open and \(\Phi_{x,\delta}:\nUnitBall\xrightarrow{\sim}\Phi_{x,\delta}(\nUnitBall)\)
            is a \(\CinftySpace\) diffeomorphism.
        \item\label{Item::Scaling::WithoutBdry::HardContainment} \(\forall x\in \Compact\), \(\delta\in (0,\delta_0]\), \(\forall \eta_0\in (0,1]\), \(\exists \xi_0\in (0,1]\),
            \begin{equation*}
                \BWWd{x}{\xi_0\delta}\cup \BXXd{x}{\xi_0\delta}
                \subseteq \Phi_{x,\delta}(\nBall{\eta_0})
                \subseteq \Phi_{x,\delta}(\nBall{1})
                \subseteq \BWWd{x}{\delta}\cap \BXXd{x}{\delta}.
            \end{equation*}
    \end{enumerate}
    Set \(W_k^{x,\delta}:=\Phi_{x,\delta}^{*}\delta^{\Wd_k}W_k\) and \(X_k^{x,\delta}:=\Phi_{x,\delta}^{*}\delta^{\Xd_k}X_k\).
    \begin{enumerate}[(a),resume]
        \item\label{Item::Scaling::WithoutBdry::PulledBackSmooth} \(\left\{ W_k^{x,\delta}, X_l^{x,\delta} : x\in \Compact,\delta\in (0,\delta_0], 1\leq k\leq r, 1\leq l\leq q \right\}\subset \CinftybSpace[\nUnitBall]\) is a bounded set.
        \item\label{Item::Scaling::WithoutBdry::PulledBackHormander} \(W_1^{x,\delta},\ldots, W_r^{x,\delta}\) are H\"ormander vector fields on \(\nUnitBall\), uniformly for \(x\in \Compact\), \(\delta\in (0,\delta_0]\),
            in the sense of Definition \ref{Defn::Scaling::UnitScale::UniformHormander}.
        \item\label{Item::Scaling::WithoutBdry::PulledBackSpan} With \(j_1=j_1(x,\delta),\ldots, j_n=j_n(x,\delta)\) as in \eqref{Eqn::Scaing::WithoutBoundary::PickingXBasis}, we have
            \begin{equation*}
                \inf_{\substack{x\in \Compact \\ \delta\in (0,\delta_0]}} \inf_{u\in \nUnitBall} \left| \det\left( X_{j_1}^{x,\delta}(u)| \cdots | X_{j_n}^{x,\delta}(u) \right) \right|>0.
            \end{equation*}
        \item\label{Item::Scaling::WithoutBdry::PulledBackNorm}
            \(\forall x\in \Compact\), \(\forall \delta\in (0,\delta_0]\),
            \begin{equation*}
                \CmbNorm{f}{L}[\nUnitBall] \approx \sum_{|\alpha|\leq L} \CmbNorm{ \left( X^{x,\delta} \right)^{\alpha} f}{0}[\nUnitBall],
            \end{equation*}
            where the implicit constants depend on \(L\in \Zgeq\), but not on \(x\in \Compact\), \(\delta\in (0,\delta_0]\),
            or \(f\in \CmbSpace{L}[\nUnitBall]\).
    \end{enumerate}
    Define \(h_{x,\delta}\in \CinftySpace[\nUnitBall]\) by
    \(\Phi_{x,\delta}^{*}\Vol=h_{x,\delta}\Lambda(x,\delta)\LebDensity\). Then,
    \begin{enumerate}[(a),resume]
        \item\label{Item::Scaling::WithoutBdry::hIsSmooth}  \(\left\{ h_{x,\delta}:x\in \Compact,\delta\in (0,\delta_0] \right\}\subset \CinftybSpace[\nUnitBall]\) is a bounded set.
        \item\label{Item::Scaling::WithoutBdry::hIsPositive} \(\inf_{\substack{x\in \Compact \\ \delta\in (0,\delta_0]}} \inf_{t\in \nUnitBall}h_{x,\delta}(t)>0\).
    \end{enumerate}
\end{theorem}
\begin{proof}[Comments on the proof]
   This is largely contained in \cite[Theorems 3.5.1 and 3.5.4]{StreetMaximalSubellipticity}
   and follows from \cite[Theorem 3.6.5]{StreetMaximalSubellipticity}, which was originally proved in \cite{StovaStreetCoordinatesAdaptedToVectorFieldsCanonicalCoordinates}. 
   We make a few comments on the differences between the statement here the statements of
   \cite[Theorems 3.5.1 and 3.5.4]{StreetMaximalSubellipticity}.

   \ref{Item::Scaling::WithoutBdry::PulledBackSpan} is stated in a slightly stronger way than
   \cite[Theorem 3.5.1(h)]{StreetMaximalSubellipticity}, but the proof their establishes 
   \ref{Item::Scaling::WithoutBdry::PulledBackSpan} directly.
    In the proof of \cite[Theorems 3.5.1 and 3.5.4]{StreetMaximalSubellipticity},
   \(\zeta\) is chosen to be \(1/2\), but that is not relevant to the
   proof, and the application of \cite[Theorem 3.6.5]{StreetMaximalSubellipticity} goes through with any fixed \(\zeta\).

   More significantly, the relationship between \(\WWd\) and \(\XXd\) here is less explicit
   than the one in \cite[Section 3.5]{StreetMaximalSubellipticity}. In light of this,
   the proof of \cite[Theorems 3.5.1 and 3.5.4]{StreetMaximalSubellipticity} goes through to prove
   all of the parts of Theorem \ref{Thm::Scaling::WithoutBoundary::MainThm} concerning \(\XXd\),
   but it does not directly apply to \(\WWd\).

    \ref{Item::Scaling::WithoutBdry::BallsInOmega} follows from the Picard--Lindel\"of Theorem by taking \(\delta_0\) small.
    Since \(\XXd\) locally strongly controls \(\WWd\) on \(\Omega\), it follows from \ref{Item::Scaling::WithoutBdry::BallsInOmega}
    that \ref{Item::Scaling::WithoutBdry::PulledBackSmooth} for \(W_k^{x,\delta}\)
    follows from the corresponding (already proved) result for \(X_k^{x,\delta}\), and
    \(\exists c_1\gtrsim 0\) such that \(\forall x\in \Compact\) and \(\delta>0\) small
    \begin{equation}\label{Eqn::Scaling::WithoutBdry::WWdBallsInXXdBalls}
        \BWWd{x}{c_1\delta}\subseteq \BXXd{x}{\delta}.
    \end{equation}

    Since \(\WWd\) locally weakly controls \(\XXd\) on \(\Omega\), \ref{Item::Scaling::WithoutBdry::PulledBackHormander} follows from \ref{Item::Scaling::WithoutBdry::BallsInOmega}
    and a straight-forward modification of the proof of \cite[Theorem 3.5.1(i)]{StreetMaximalSubellipticity}.
    From here, that \(\exists c_2\gtrsim 1\) such that \(\forall x\in \Compact\), \(\delta>0\) small
    \begin{equation}\label{Eqn::Scaling::WithoutBdry::XXdBallsInWWdBalls}
        \BWWd{x}{c_1\delta}\subseteq \BXXd{x}{\delta},
    \end{equation}
    follows just as in \cite[(3.25)]{StreetMaximalSubellipticity}.

    \eqref{Eqn::Scaling::WithoutBdry::WWdBallsInXXdBalls} and \eqref{Eqn::Scaling::WithoutBdry::XXdBallsInWWdBalls}
    show that the balls \(\BXXd{x}{\delta}\) and \(\BWWd{x}{\delta}\) are comparable.
    From here, the remainder of the theorem follows by using
    \(\Phi_{x,c_2\delta}\) in place of \(\Phi_{x,\delta}\), where
    \(\Phi_{x,\delta}\)
    is the map from \cite[Theorems 3.5.1 and 3.5.4]{StreetMaximalSubellipticity}, and some simple modifications
    of the proof, which we leave to the reader.
\end{proof}

    \subsection{Quantitative equivalence of topologies}
    As described in Corollary \ref{Cor::GlobalCor::EquivTop} and Theorem \ref{Thm::Metrics::Results::GivesUsualTopology},
the metric topology on \(\ManifoldNncWWd\) induced by \(\MetricWWd\) is the same as the standard topology on
\(\ManifoldNncWWd\) as manifold. In this section we prove the central result needed to establish this equality;
as will be required in the proof of Theorem \ref{Thm::Scaling::NearBdry::MainThm}, we state and prove a quantitative version of this.
See the proof of Theorem \ref{Thm::Metrics::Results::GivesUsualTopology} for how the results in this section imply the topologies are the same.

We let \(\CubeCentered{y}{a}=\left\{ x\in \R^n : |x-y|_{\infty}<a \right\}\) and
\(\CubeCenteredgeq{y}{a}=\left\{ x=(x_1,\ldots,x_n)\in \CubeCentered{y}{a}  : x_n\geq 0 \right\}\).
Note that \(\nCube{a}=\CubeCentered{0}{a}\)
and \(\nCubegeq{a}=\CubeCenteredgeq{0}{a}\).

Fix \(a\in (0,1]\). Let \(\Wh_0,\Wh_1,\ldots, \Wh_r\in \CinftybSpace[\nCube{a}][\R^n]\)
be smooth vector fields on \(\nCube{a}\). Assign to each \(\Wh_j\) a formal degree
\(\Wd_j\in \Zg\). Set
\begin{equation*}
    \WhWd=\left\{ (\Wh_0,\Wd_0),(\Wh_1,\Wd_1),\ldots, (\Wh_r,\Wd_r) \right\}\subset \CinftybSpace[\nCube{a}][\R^n]\times \Zg.
\end{equation*}

Fix \(m\in \Zg\) and let \(\Xh_1,\ldots, \Xh_q\in \CinftybSpace[\nCube{a}][\R^n]\) denote the commutators
of \(\Wh_0,\ldots, \Wh_r\) up to order \(m\). We assume that \(\Wh_0,\ldots, \Wh_r\)
satisfy H\"ormander's condition at \(0\) of order \(m\) in the sense that:
\begin{equation}\label{Eqn::Scaling::EquivTop::HorCond}
    \gamma_0:=\max_{j_1,\ldots, j_n\in \{1,\ldots,q\}} \left| \det\left( \Xh_{j_1}(0)| \Xh_{j_2}(0)|\cdots|\Xh_{j_n}(0) \right) \right|>0.
\end{equation}

\begin{definition}
    For a parameter \(\iota\), we say \(C\geq 0\) is an \textit{\(\iota\)-admissible constant} if 
    \(\exists L\in \Zgeq\), depending only on upper bounds for \(n\), \(r\), and \(m\),
    such that \(C\) can be chosen to depend only on \(\iota\) and upper bounds
    for \(n\), \(r\), \(m\), \(\gamma_0^{-1}\), 
    \(\max\{\Wd_j\}\),
    and
    \begin{equation*}
        \max_{0\leq j\leq r} \CmbNorm{\Wh_j}{L}[\nCube{a}][\R^n]. 
    \end{equation*}
    We say \(C\geq 0\) is an \textit{admissible constant}, if it is a \(0\)-admissible constant.
\end{definition}

\begin{proposition}\label{Prop::Scaling::EquivTop::Interior}
    \(\exists\) an admissible constant \(a_1\in (0,a]\), \(\forall a_2\in (0,a_1]\),
    \(\exists\) an \(a_2\)-admissible constant \(\delta_1\in(0,1]\),
    \(\forall \delta\in (0,\delta_1]\),
    \(\exists\) an \((a_2,\delta)\)-admissible constant \(a_3\in(0,a_2/2]\),
    \(\forall y\in \nCube{a_2/2}\),
    \begin{equation*}
        \CubeCentered{y}{a_3}\subseteq \BWWd{y}{\delta}\subseteq \nCube{a_2}.
    \end{equation*}
\end{proposition}

For the next result, we suppose in addition that
the boundary \(\{x_n=0\}\) is \textbf{non-characteristic} for \(\Wh_0\) near \(0\) in the sense that
\(\Wh_0=\sum_{k=1}^n b_0^k(x) \partial_{x_k}\) and
\(|b_0^n(0)|\geq c_0>0\).  Define \(W_j:=\Wh_j\big|_{\nCubegeq{a}}\).

\begin{proposition}\label{Prop::Scaling::EquivTop::Boundary}
    \(\exists\) a \(c_0\)-admissible constant \(a_1\in (0,a]\),
    \(\forall a_2\in (0,a_1]\),
    \(\exists\) an \((a_2,c_0)\)-admissible constant \(\delta_1\in (0,1]\),
    \(\forall \delta\in (0,\delta_1]\),
    \(\exists\) an \((a_2,c_0,\delta)\)-admissible constant \(a_3\in (0,a_2/2]\),
    \(\forall y\in \nCubegeq{a_2/2}\),
    \begin{equation}\label{Eqn::Scaling::EquivTop::Boundary::MainContainment}
        \CubeCenteredgeq{y}{a_3} \subseteq \BWWd{y}{\delta}\subseteq \nCubegeq{a_2}.
    \end{equation}
\end{proposition}

\begin{remark}
    Proposition \ref{Prop::Scaling::EquivTop::Interior} could be stated more simply
    by considering only \(y=0\); and the more general result is a simple corollary.
    However, this is not true of Proposition \ref{Prop::Scaling::EquivTop::Boundary},
    where \(x_n=0\) is distinguished to be the boundary. By stating 
    Proposition \ref{Prop::Scaling::EquivTop::Interior} in the same format
    as Proposition \ref{Prop::Scaling::EquivTop::Boundary}, later proofs are simplified as we are
    able to address both cases simultaneously.
\end{remark}

Proposition \ref{Prop::Scaling::EquivTop::Interior} is a result of Nagel, Stein, and Wainger \cite{NagelSteinWaingerBallsAndMetricsDefinedByVectorFieldsI};
see \cite[Lemma 3.2.4]{StreetMaximalSubellipticity} for an exposition (though the statement is slightly different,
the proof can be easily adapted to prove Proposition \ref{Prop::Scaling::EquivTop::Interior}). 
Proposition \ref{Prop::Scaling::EquivTop::Boundary} has a similar, though more complicated, proof, which we present.
The proof we present can be easily simplified to give Proposition \ref{Prop::Scaling::EquivTop::Interior} (an in doing so
one would recreate the proof from \cite{NagelSteinWaingerBallsAndMetricsDefinedByVectorFieldsI}).

The remainder of this section is devoted to the proof of Proposition \ref{Prop::Scaling::EquivTop::Boundary}.

The second containment in \eqref{Eqn::Scaling::EquivTop::Boundary::MainContainment}
follows by taking \(\delta_1\) small and applying the Picard--Lindel\"of Theorem,
so we focus only on the first. We write \(A\lesssim B\) for \(A\leq CB\) where \(C\geq 1\) is an
appropriate admissible constant.  Write \(\Wh_j=\sum_{k} b_j^k(x)\partial_{x_n}\).
We begin with some reductions:
\begin{enumerate}[(I)]
    \item By shirking \(a>0\) to a \(c_0\)-admissible constant, we may assume \(|b_0^n(x)|\geq c_0/2\gtrsim 1\),
        \(\forall x\in \nCube{a}\).
        
    \item\label{Item::Scaling::EquivTop::b0nequals1} By replacing \(\Wh_0\) with \((b_0^n)^{-1}\Wh_0\), we may assume \(b_0^n\equiv 1\).
    \item\label{Item::Scaling::EquivTop::bjnequals1} By replacing \(\Wh_j\) with \(\Wh_j-b_j^n \Wh_0\), we may assume \(b_j^n\equiv 0\), \(\forall j\geq 1\).
    \item Since \(\BWWo{y}{\delta}\subseteq \BWWd{y}{ \delta^{1/\max{\Wd_j}}}\), it suffices to prove the first containment
        in \eqref{Eqn::Scaling::EquivTop::Boundary::MainContainment} when \(\Wd_j=1\), \(\forall j\).
\end{enumerate}

For smooth vector fields \(S_1,S_2,\ldots\), and for \(t\in \R\) small enough that the expressions make sense, we set
\begin{equation*}
    C_1(t,S_1)=e^{t S_1}, \quad C_2(t,S_1,S_2)=e^{-tS_2}e^{-tS_1}e^{tS_2}e^{tS_1},
\end{equation*}
and  recursively,
\begin{equation}\label{Eqn::Scaling::EquivTop::Boundary::FormulaForCl}
    C_l(t,S_1,\ldots, S_l)=e^{-tS_l} C_{l-1}(t,S_1,\ldots, S_{l-1})^{-1}e^{t S_l} C_{l-1}(t, S_1,\ldots, S_{l-1}),
\end{equation}
so that
\begin{equation}\label{Eqn::Scaling::EquivTop::Boundary::FormulaForClInv}
    C_l(t,S_1,\ldots, S_l)^{-1} = C_{l-1}(t, S_1,\ldots, S_{l-1})^{-1} e^{-tS_l} C_{l-1}(t,S_1,\ldots, S_{l-1})e^{tS_l}
\end{equation}

\begin{lemma}\label{Lemma::Scaling::EquivTop::SjminusIsCancelled}
    Each time \(e^{-tS_j}\) appears in either \(C_l(t,S_1,\ldots, S_l)\) (when \(l\geq 1\))
    or \(C_l(t,S_1,\ldots, S_l)^{-1}\) (when \(l\geq 2\)), a copy of \(e^{tS_j}\) appears to the right of it.
\end{lemma}
\begin{proof}
    This is clearly true for \(l=1\) and \(l=2\). The full claim follows by induction and 
    \eqref{Eqn::Scaling::EquivTop::Boundary::FormulaForCl} and \eqref{Eqn::Scaling::EquivTop::Boundary::FormulaForClInv}.
\end{proof}

Note that, \(C_l(t,S_1,\ldots, S_l)x\) is jointly smooth in \(t\) and \(x\) (on its domain) and for \(1\leq k\leq l\),
\begin{equation}\label{Eqn::Scaling::EquivTop::DerivOfC}
    \partial_t^k\big|_{t=0} C_l(t,S_1,\ldots, S_l)x=
    \begin{cases}
        0, &k<l,\\
        [S_1,[S_2,[\cdots,[S_{l-1},S_l]]]](x), &k=l, l>1,\\
        S_1(x), & k=1, l=1.
    \end{cases}
\end{equation}

Let \(Y_1=\Wh_0\). In light of the reductions \ref{Item::Scaling::EquivTop::b0nequals1} and \ref{Item::Scaling::EquivTop::bjnequals1}
and H\"ormander's condition \eqref{Eqn::Scaling::EquivTop::HorCond}, we may pick \(Y_2,\ldots, Y_n\) of the form
\begin{equation}\label{Eqn::Scaling::EquivTop::DefineYj}
    Y_j = \left[\Wh_{j_1^j},
    \left[\Wh_{l_2^j},
    \left[\cdots
    \left[\Wh_{l_{k_{j-1}-1}^j}, \Wh_{l_{k_j}^j}
    \right] 
    \right] 
    \right]  
    \right]
\end{equation}
with \(1\leq k_j\lesssim 1\) and such that there exists a \(c_0\)-admissible constant \(a_1\in (0,a]\)
with
\begin{equation}\label{Eqn::Scaling::EquivTop::DetYs}
    \left| \det\left( Y_1(y)|\cdots|Y_n(y) \right) \right|\gtrsim 1, \quad \forall |y|\leq a_0.
\end{equation}

Set, for \(j\geq 2\),
\begin{equation}\label{Eqn::Scaling::EquivTop::DefineDjPlus}
    D_j^{+}(t)=C_{k_j}(t, \Wh_{l_1^j}, \Wh_{l_2^j},\ldots, \Wh_{l_{k_j}^j})
\end{equation}
So that \(D_j^{+}(t)x\) is smooth in \(t\) and \(x\) on its domain and by 
\eqref{Eqn::Scaling::EquivTop::DerivOfC} and
\eqref{Eqn::Scaling::EquivTop::DefineYj} we have
\begin{equation}\label{Eqn::Scaling::EquivTop::DerivDjPlus}
    \partial_t^{l}\big|_{t=0} D_j^{+}(t)x=
    \begin{cases}
        0, & 1\leq l\leq k_j-1,\\
        Y_j(x),&l=k_j.
    \end{cases}
\end{equation}

For \(j\geq 2\) at least one of the vector fields in \eqref{Eqn::Scaling::EquivTop::DefineYj}
is not easy to \(\Wh_0\); i.e., \(l_k^j\ne 0\) for some \(k\).
Pick such a 
\(k\) and define \(D_j^{-}(t)\) by replacing \(\Wh_{l_k^j}\) with \(-\Wh_{l_k^j}\)
in \eqref{Eqn::Scaling::EquivTop::DefineDjPlus}.
For example, if \(l_1^j\ne 0\) then we could take
\begin{equation*}
    D_j^{-}(t)=C_{k_j}(t, -\Wh_{l_1^j}, \Wh_{l_2^j},\ldots, \Wh_{l_{k_j}^k}).
\end{equation*}
By \eqref{Eqn::Scaling::EquivTop::DerivOfC} and \eqref{Eqn::Scaling::EquivTop::DefineYj}, we have
\begin{equation}\label{Eqn::Scaling::EquivTop::DerivDjMinus}
    \partial_t^{l}\big|_{t=0} D_j^{-}(t)x=
    \begin{cases}
        0, & 1\leq k\leq k_j-1,\\
        -Y_j(x),&l=k_j.
    \end{cases}
\end{equation}

\begin{lemma}\label{Lemma::Scaling::EquivTop::UnderstandDs}
    \(\exists N\lesssim 1\) such that for \(t_2,\ldots, t_n\geq 0\) (small enough that everything makes sense),
    and \(y\in \nCubegeq{a/2}\), if
    \begin{equation}\label{Eqn::Scaling::EquivTop::ItteratedDs}
        y'=D_2^{\pm}(t_2) D_3^{\pm}(t_3)\cdots D_n^{\pm}(t_n)y,
    \end{equation}
    then
    \begin{equation}\label{Eqn::Scaling::EquivTop::yprimeInBall}
        y'\in \BWWo{x}{N(t_2+\cdots+t_n)}
    \end{equation}
    and \(y_n'=y_n\).
    In \eqref{Eqn::Scaling::EquivTop::ItteratedDs}, each \(D_j^{\pm}\) can be either \(D_j^{+}\) or \(D_j^{-}\).
\end{lemma}
\begin{proof}
    The choice of \(\pm\) does not affect the argument which follows so we consider only the case
    \begin{equation*}
        y'=D_2^{+}(t_2) D_3^{+}(t_3)\cdots D_n^{+}(t_n)y.
    \end{equation*}
    There we define a path \(\gamma:[0,t_1+\cdots+t_n]\rightarrow \nCube{a}\), with \(\gamma(0)=y\)
    and \(\gamma(t_2+\cdots+t_n)=y'\) by
    \begin{equation}\label{Eqn::Scaling::EquivTop::gammaPath}
        \gamma(s)=
        \begin{cases}
            D_n^{+}(s)y, &0\leq s< t_n,\\
            D_{n-1}^{+}(s-t_n)D_{n}^{+}(t_n)y, &t_n\leq s<t_n+t_{n-1},\\
            \cdots,\\
            D_2^{+}(s-t_n-t_{n-1}-\cdots-t_3) D_3^{+}(t_3)\cdots D_n^{+}(t_n)y, & t_n+t_{n-1}+\cdots+t_3\leq s\leq t_n+t_{n-1}+\cdots +t_2.
        \end{cases}
    \end{equation}
    Since each \(D_j^{+}(t)\) is a composition of exponentials of \(\Wh_j\),
    this shows \(y'\in \BWhWo{y}{N(t_2+\cdots+t_n)}\) for some \(N\lesssim 1\).
    To establish \eqref{Eqn::Scaling::EquivTop::yprimeInBall} we require that \(\gamma(s)\) never goes below \(x_n=0\).

    By Lemma \ref{Lemma::Scaling::EquivTop::SjminusIsCancelled}, each time
    \(e^{-{t}\Wh_0}\) appears in the expression \eqref{Eqn::Scaling::EquivTop::gammaPath} for \(\gamma\),
    it is preceded by \(e^{t_j\Wh_0}\) for some \(t_j\geq t\).
    Since \(\Wh_0\) has \(\partial_{x_n}\) component equal to \(1\) (by \ref{Item::Scaling::EquivTop::b0nequals1})
    and the other \(\Wh_j\) have \(\partial_{x_n}\) component equal to \(0\) (by \ref{Item::Scaling::EquivTop::bjnequals1})
    it follows that \(\gamma\) always stays above the line \(\{x_n=y_n\}\) and therefore above \(\{x_n=0\}\).
    \eqref{Eqn::Scaling::EquivTop::yprimeInBall} follows.
    Finally, that \(y_n'=y_n\) follows from the same argument, since the \(\partial_{x_n}\) components in the
    exponentials exactly cancel out when considering \(y'=\gamma(t_2+\cdots+t_n)\), by Lemma \ref{Lemma::Scaling::EquivTop::SjminusIsCancelled}.
\end{proof}

Set, for \(j\geq 2\),
\begin{equation*}
    E_j(t)=
    \begin{cases}
        D_j^{+}(|t|^{1/k_j}), & t\geq 0,\\
        D_j^{-}(|t|^{1/k_j}), & t<0.
    \end{cases}
\end{equation*}
By \eqref{Eqn::Scaling::EquivTop::DerivDjPlus} and \eqref{Eqn::Scaling::EquivTop::DerivDjMinus},
we have \(E_j(t)x\) is \(\CmSpace{1}\) in \((t,x)\) and \(\partial_t\big|_{t=0} E_j(t)x=Y_j(x)\).
Set, for \(y\in \nCube{a_1}\),
\begin{equation}\label{Eqn::Scaling::EquivTop::FormulaForF}
    F_y(t_1,\ldots, t_n):=e^{t_1\Wh_0} E_2(t_2)E_3(t_3)\cdots E_n(t_n)y,
\end{equation}
so that by \eqref{Lemma::Scaling::EquivTop::UnderstandDs} and the fact that the \(\partial_{x_n}\)
component of \(\Wh_0\) equals \(1\) (by \ref{Item::Scaling::EquivTop::b0nequals1}),
the \(n\)-th component of
\(F_y(t_1,\ldots, t_n)\)
is \(\geq 0\) if and only if \(-t_1\geq y_n\).

By the definition of \(F_y\), we have \(\CmbNorm{F_y}{1}\lesssim 1\)
and \(d_t\big|_{t=0} F_y(t_1,\ldots,t_n)=\left( Y_1(y)|\cdots|Y_n(y) \right)\).
In light of \eqref{Eqn::Scaling::EquivTop::DetYs}, the Inverse Function Theorem implies
\(\forall \epsilon>0\), \(\exists\) a \((c_0,\epsilon)\)-admissible constant
\(a_3>0\) with
\begin{equation}\label{Eqn::Scaling::EquivTop::CubeIn}
    \CubeCentered{y}{a_3}\subseteq F_y(\nBall{\epsilon})
\end{equation}
and therefore
\begin{equation*}
    \CubeCenteredgeq{y}{a_3}\subseteq F_y\left( \left\{ t\in \nBall{\epsilon} : t_1\geq -y_n \right\} \right).
\end{equation*}
But if \(\epsilon>0\) is a sufficiently small \((\delta,c_0)\)-admissible constant, using \eqref{Eqn::Scaling::EquivTop::yprimeInBall}
and \eqref{Eqn::Scaling::EquivTop::FormulaForF} we see
\begin{equation}\label{Eqn::Scaling::EquivTop::ImageIn}
    F_y\left( \left\{ t\in \nBall{\epsilon} : t_1\geq -y_n \right\} \right)\subseteq \BWWo{y}{\delta}.
\end{equation}
Combining \eqref{Eqn::Scaling::EquivTop::CubeIn} and \eqref{Eqn::Scaling::EquivTop::ImageIn}
completes the proof.

    \subsection{Scaling near the boundary}\label{Section::Scaling::NearBoundary}
    Let
\(\WhWd=\left\{ (\Wh_1,\Wd_1),\ldots, (\Wh_r,\Wd_r) \right\}\subset \CinftybSpace[\nUnitBall][\R^n]\times \Zg\)
be H\"ormander vector fields with formal degrees on \(\nUnitBall\);
where \(\Wh_1,\ldots, \Wh_r\) satisfy H\"ormander's condition of order \(m\in \Zg\) on \(\nUnitBall\). Set
\begin{equation*}
    W_j:=\Wh_j\big|_{\nUnitBallgeq[1]}.
\end{equation*}
Write \(\Wh_j=\sum_{k=1}^n a_j^k \partial_{x_k}\). We assume that the boundary \(\left\{ (x_1,\ldots, x_n)\in \nUnitBall : x_n=0 \right\}\cong \nmoUnitBall\)
is \textbf{non-characteristic} in the sense that
\begin{equation}\label{Eqn::Scaling::NearBdry::WoIsTheNonCharVF}
    |a_1^n(x',0)|\geq c_0>0,\quad \forall x'\in \nmoUnitBall,
\end{equation}
\begin{equation}\label{Eqn::Scaling::NearBdry::NonCharLowerDegreesVanish}
    \Wd_j<\Wd_1\implies a_j^n(x',0)=0, \quad \forall x'\in \nmoUnitBall.
\end{equation}
Let \(h(x)\in \CinftybSpace[\nUnitBall]\) with \(h(x)\geq c_1>0\), \(\forall x\in \nUnitBall\).
Set \(\Vol:=h \LebDensity\); where \(\LebDensity\) is the usual Lebesgue density on \(\Rn\).

We follow the procedure as described in Section \ref{Section::BndryVfsWithFormaLDegrees},
with \(\BoundaryN\) replaced by \(\nmoUnitBall\) and with \(W_{j_0}\) replaced by \(\Wh_0\).
Namely, we set
\begin{equation*}
    \ZhZd:=\left\{(\Zh_0,\Zd_0), ( \Zh_1,\Zd_1 ),\ldots, ( \Zh_q,\Zd_q ) \right\}
            :=\left\{ \left( Y,e \right)\in \GenWhWd : e\leq m\max\{\Wd_1,\ldots, \Wd_r\} \right\}.
\end{equation*}
Ordered so that \((\Zh_0,\Zd_0)=(\Wh_1,\Wd_1)\).
Let \(Z_j=\Zh_j\big|_{\nUnitBallgeq}\) and
\(\ZZd=\left\{ (Z_0,\Zd_0),\ldots, (Z_q,\Zd_q) \right\}\).

Let \((\Xh_0,\Xd_0)=(\Wh_0,\Wd_0)\), and with this choice define \(\bt_j\in \CinftybSpace[\nUnitBall]\)
as in  Section \ref{Section::BndryVfsWithFormaLDegrees}.
Define \(\XhXd=\left\{ (\Xh_0,\Xd_0),\ldots, (\Xh_q,\Xd_q) \right\}\subset \CinftybSpace[\nUnitBall][\Rn]\) as in Step \ref{Item::BoundaryVfs::DefineXj} and 
\(\VVd=\left\{ (V_1,\Vd_1),\ldots, (V_q,\Vd_q) \right\}\subset \CinftybSpace[\nmoUnitBall][\Rnmo]\) as in Step \ref{Item::BoundaryVfs::DefineVj};
so that \((V_j,\Vd_j)=(\Xh_j\big|_{\nmoUnitBall},\Xd_j)\), \(1\leq j\leq q\).
Note that \(\XhXd\) and \(\ZhZd\) are locally strongly equivalent on \(\nUnitBall\).
Set \(X_j:=\Xh_j\big|_{\nUnitBallgeq}\) and \(\XXd:=\left\{ (X_1,\Xd_1),\ldots, (X_q,\Xd_q) \right\}\).

For \(x\in \nmoUnitBall\) and \(\delta>0\), let
\begin{equation*}
    \Lambda(x,\delta):=\max_{j_1,\ldots, j_n\in \{0,\ldots,q\}} \left| \det\left( \delta^{\Xd_{j_1}}\Xh_{j_1}(x)| \delta^{\Xd_{j_2}}\Xh_{j_2}(x)|\cdots|\delta^{\Xd_{j_n}}\Xh_{j_n}(x) \right) \right|.
\end{equation*}
\(\Lambda(x,\delta)>0\) since \(\Xh_0,\ldots, \Xh_q\) span the tangent space at every point (see Step \ref{Item::BoundaryVfs::DefineXj}
in Section \ref{Section::BndryVfsWithFormaLDegrees}).
For each \(x\in \nmoUnitBall\), \(\delta>0\) pick \(j_1=j_1(x,\delta),\ldots, j_n=j_n(x,\delta)\) such that
\begin{equation*}
    \left| \det\left( \delta^{\Xd_{j_1}}\Xh_{j_1}(x)| \delta^{\Xd_{j_2}}\Xh_{j_2}(x)|\cdots|\delta^{\Xd_{j_n}}\Xh_{j_n}(x) \right) \right|
    =\Lambda(x,\delta).
\end{equation*}
By the above construction (see Section \ref{Section::BndryVfsWithFormaLDegrees}), \(\Xh_0\)
is the only \(\Xh_j\) with a non-zero \(\partial_{x_n}\) component on \(\nmoUnitBall\).
Therefore, one of the \(j_k(x,\delta)\) equals \(0\); without loss of generality, let \(j_n(x,\delta)=0\).

Note that \(\BWWd{x}{\delta}\subseteq \nBallgeq{1}\)
while \(\BWhWd{x}{\delta}\subseteq \nBall{1}\) with \(\BWWd{x}{\delta}\subseteq \BWhWd{x}{\delta}\cap \nBallgeq{1}\),
and similarly for \(\XhXd\) and \(\XXd\) and \(\ZhZd\) and \(\ZZd\).

The \(\partial_{x_n}\) component of \(\Xh_{0}\) is either always positive or always negative on \(\nmoUnitBall\);
let \(\omega=\pm 1\) be the sign of this component.
Set, for \(x\in \nmoBall{1/2}\), \(\delta>0\),
\begin{equation}\label{Eqn::Scaling::NearBdry::FormulaForpsixdelta}
    \psi_{x,\delta}(t_1,\ldots, t_n)
    = \exp\left(
        t_n \omega\delta^{\Xd_0} \Xh_0
    \right)
    \exp\left(
        t_1 \delta^{\Xd_{j_1}}\Xh_{j_1} + t_2 \delta^{\Xd_{j_2}}\Xh_{j_2}+\cdots t_{n-1}\delta^{\Xd_{j_{n-1}}}\Xh_{j_{n-1}}
    \right) 
    x.
\end{equation}

\begin{theorem}\label{Thm::Scaling::NearBdry::MainThm}
    \(\exists \delta_0\in (0,1]\), \(\exists \sigma_1\in (0,1]\), \(\forall x\in \nmoBall{1/2}\), \(\forall \delta\in (0,\delta_0]\),
    \begin{enumerate}[(i)]
        \item\label{Item::Scaling::NearBdry::InBall3/4}
            \(
                \BXhXd{x}{\delta}\cup \BZhZd{x}{\delta}\cup \BWhWd{x}{\delta}\subseteq \nBall{3/4}.
            \)
        \item\label{Item::Scaling::NearBdry::MapsAreDiffeos} 
            \begin{enumerate}[(a)]
                \item\label{Item::Scaling::NearBdry::ImagesOpen} \(\psi_{x,\delta}(\nCube{\sigma_1})\subseteq \nBall{3/4}\), \(\psi_{x,\delta}(\nCubegeq{\sigma_1})\subseteq \nBallgeq{3/4}\),
                    and \(\psi_{x,\delta}(\nmoCube{\sigma_1})\subseteq \nmoBall{3/4}\), and each image is an open subset of the respective space.
                \item\label{Item::Scaling::NearBdry::Diffeos} The following maps are smooth diffeomorphims: \(\psi_{x,\delta}:\nCube{\sigma_1}\xrightarrow{\sim}\psi_{x,\delta}(\nCube{\sigma_1})\subseteq \nBall{3/4}\),
                    \(\psi_{x,\delta}:\nCubegeq{\sigma_1}\xrightarrow{\sim}\psi_{x,\delta}(\nCubegeq{\sigma_1})\subseteq \nBallgeq{3/4}\),
                    and \(\psi_{x,\delta}:\nmoCube{\sigma_1}\xrightarrow{\sim}\psi_{x,\delta}(\nmoCube{\sigma_1})\subseteq \nmoBall{3/4}\).
            \end{enumerate}
        \item\label{Item::Scaling::NearBdry::IntersectionCommutes} 
        \(\psi_{x,\delta}(\nCube{\sigma_1})\cap \nUnitBallgeq = \psi_{x,\delta}(\nCubegeq{\sigma_1})\) and
            \(\psi_{x,\delta}(\nCube{\sigma_1})\cap \nmoUnitBall = \psi_{x,\delta}(\nmoCube{\sigma_1})\).
        \item\label{Item::Scaling::NearBdry::UniformlySmooth} The following are bounded sets:
            \begin{equation*}
                \left\{ \psi_{x,\delta}^{*}\delta^{\Wd_j}\Wh_j : x\in \nmoBall{1/2}, \delta\in (0,\delta_0], 1\leq j\leq r  \right\}
                \subset \CinftybSpace[\nCube{\sigma_1}][\R^n],
            \end{equation*}
            \begin{equation*}
                \left\{ \psi_{x,\delta}^{*} \delta^{\Xd_k} \Xh_k : x\in \nmoBall{1/2}, \delta\in (0,\delta_0], 0\leq k\leq q  \right\}
                \subset \CinftybSpace[\nCube{\sigma_1}][\R^n],
            \end{equation*}
            \begin{equation*}
                \left\{ \psi_{x,\delta}^{*}\delta^{\Zd_k}\Zh_l : x\in \nmoBall{1/2}, \delta\in (0,\delta_0], 0\leq k\leq q  \right\}
                \subset \CinftybSpace[\nCube{\sigma_1}][\R^n],
            \end{equation*}
            \begin{equation*}
                \left\{ \psi_{x,\delta}\big|_{\nmoCube{\sigma_1}}^{*} \delta^{\Vd_k}V_k : x\in \nmoBall{1/2}, \delta\in (0,\delta_0], 1\leq k\leq q  \right\}
                \subset \CinftybSpace[\nmoCube{\sigma_1}][\Rnmo].
            \end{equation*}

        \item\label{Item::Scaling::NearBdry::UniformlySpan} 
        In the sense of Definition \ref{Defn::Scaling::UnitScale::UniformSpan},
            \begin{enumerate}[(a)]
                \item\label{Item::Scaling::NearBdry::UniformlySpan::X} \(\psi_{x,\delta}^{*}\delta^{\Xd_0}\Xh_0,\ldots, \psi_{x,\delta}^{*}\delta^{\Xd_q}\Xh_q\) span the tangent
                    space to \(\nCube{\sigma_1}\), uniformly for \(x\in \nmoBall{1/2}\), \(\delta\in (0,\delta_0]\).
                \item\label{Item::Scaling::NearBdry::UniformlySpan::Z} \(\psi_{x,\delta}^{*}\delta^{\Zd_0}\Zh_0,\ldots, \psi_{x,\delta}^{*}\delta^{\Zd_q}\Zh_q\) span the tangent
                    space to \(\nCube{\sigma_1}\), uniformly for \(x\in \nmoBall{1/2}\), \(\delta\in (0,\delta_0]\).
                \item\label{Item::Scaling::NearBdry::UniformlySpan::V} \(\psi_{x,\delta}\big|_{\nmoCube{\sigma_1}}^{*}\delta^{\Vd_1}V_1,\ldots, \psi_{x,\delta}\big|_{\nmoCube{\sigma_1}}^{*}\delta^{\Vd_q}V_q\) span the tangent
                    space to \(\nmoCube{\sigma_1}\), uniformly for \(x\in \nmoBall{1/2}\), \(\delta\in (0,\delta_0]\).
            \end{enumerate}
        
        \item\label{Item::Scaling::NearBdry::UniformHormander} \(\psi_{x,\delta}^{*}\delta^{\Wd_1}\Wh_1,\ldots, \psi_{x,\delta}^{*}\delta^{\Wd_r}\Wh_r\)
            are H\"ormander vector fields on \(\nCube{\sigma_1}\), uniformly for \(x\in \nmoBall{1/2}\), \(\delta\in (0,\delta_0]\)
            in the sense of Definition \ref{Defn::Scaling::UnitScale::UniformHormander}.
        
        \item\label{Item::Scaling::NearBdry::PullBackW1IsDxn} \(\psi_{x,\delta}^{*} \delta^{\Xd_0}\Xh_0=\psi_{x,\delta}^{*}\delta^{\Wd_1}\Wh_1 = \omega \partial_{t_n}\),
            and \(\psi_{x,\delta}^{*} \delta^{\Xd_j}\Xh_j(u',0)\) does not have a \(\partial_{t_n}\) component
            for \(1\leq j\leq q\), \(u\in \nmoCube{\sigma_1}\).
        
        \item\label{Item::Scaling::NearBdry::PullBackDensity} \(\psi_{x,\delta}^{*}\Vol = f_{x,\delta}(t) \Lambda(x,\delta)\LebDensity(t)\), where
            \begin{equation*}
                \left\{ f_{x,\delta} : x\in \nmoBall{1/2}, \delta\in (0,\delta_0] \right\}\subset \CinftybSpace[\nCube{\sigma_1}]
            \end{equation*}
            is a bounded set and
            \begin{equation*}
                \inf_{x\in \nmoBall{1/2}} \inf_{\delta\in (0,\delta_0]} \inf_{t\in \nCube{\sigma_1}} f_{x,\delta}(t)>0.
            \end{equation*}

        \item\label{Item::Scaling::NearBdry::xBallsInImage} 
        \(\forall \sigma_2\in (0,\sigma_1)\), \(\exists c_1\in (0,1)\), \(\forall x\in \nmoBall{1/2}\), \(\forall \delta\in (0,\delta_0]\),
            \begin{equation*}
                \left( \BWhWd{x}{c_1\delta} \cup \BXhXd{x}{c_1\delta} \right)\cap \nUnitBallgeq
                \subseteq
                \psi_{x,\delta}(\nCubegeq{\sigma_2}).
            \end{equation*}

        \item\label{Item::Scaling::NearBdry::Containments} 
        \(\forall \sigma_2\in (0,\sigma_1)\), \(\exists \sigma_3\in (0,\sigma_2/2]\), \(\forall \sigma_4\in (0,\sigma_3)\),
            \(\exists c_2\in (0,1)\), \(\forall x\in \nmoBall{1/2}\), \(\forall \delta\in (0,\delta_0]\),
            \(\forall y\in \psi_{x,\delta}(\nCubegeq{\sigma_2/2})\), we have the following containments:
            \begin{equation*}
                \begin{split}
                    &\left( \BWhWd{y}{c_2\delta}\cup \BXhXd{y}{c_2\delta} \right)\cap \nUnitBallgeq
                    \subseteq \psi_{x,\delta}( \CubeCenteredgeq{\psi_{x,\delta}^{-1}(y)}{\sigma_4})
                    \\&\subseteq \psi_{x,\delta}( \CubeCenteredgeq{\psi_{x,\delta}^{-1}(y)}{\sigma_3})
                    \subseteq \BXXd{y}{\delta}\cap \BWWd{y}{\delta},
                \end{split}            
            \end{equation*}
            \begin{equation*}
            \begin{split}
                 &\BWhWd{y}{c_2\delta}\cup \BXhXd{y}{c_2\delta} 
                 \subseteq \psi_{x,\delta}(\CubeCentered{\psi_{x,\delta}^{-1}(y)}{\sigma_4})
                 \\&\subseteq \psi_{x,\delta}(\CubeCentered{\psi_{x,\delta}^{-1}(y)}{\sigma_3})
                 \subseteq \BXhXd{y}{\delta}\cap \BWhWd{y}{\delta},
            \end{split}
            \end{equation*}
            \begin{equation*}
            \begin{split}
                 &\BVVd{x}{c_2\delta}
                 \cup \left( \BXXd{x}{c_2\delta} \cap \nmoUnitBall \right)
                 \cup \left( \BWWd{x}{c_2\delta} \cap \nmoUnitBall \right)
                 \\&\subseteq \psi_{x,\delta}(\nmoCube{\sigma_4})
                 \subseteq \psi_{x,\delta}(\nmoCube{\sigma_3})
                 \\&\subseteq 
                 \BVVd{x}{\delta} \cap \left( \BXXd{x}{\delta} \cap \nmoUnitBall \right) \cap \left( \BWWd{x}{\delta}\cap \nmoUnitBall \right).
            \end{split}
            \end{equation*}
    \end{enumerate}
\end{theorem}
\begin{proof}
    We proceed in the case \(\omega=1\) (i.e., the \(\partial_{x_n}\) component of \(\Xh_0=\Wh_1\)
    is positive on \(\nmoUnitBall\)). The case when \(\omega=-1\) follows in the same way, merely
    by replacing \(\Wh_1\) and \(\Xh_0\) with \(-\Wh_1\) and \(-\Xh_1\), throughout.

    We apply Theorem \ref{Thm::Scaling::WithoutBoundary::MainThm} with
    \(\Compact=\overline{\nmoBall{1/2}}\), \(\Omega=\nBall{3/4}\), \(\Vol\) as given,
    \(\WWd\) and \(\XXd\) replaced by \(\WhWd\) and \(\XhXd\),
    and \(j_1=j_1(x,\delta),\ldots, j_n=j_n(x,\delta)\) as described above (so that \(j_n=0\)).
    We obtain \(\Phi_{x,\delta}\) and \(\delta_0>0\) as in that theorem.

     \ref{Item::Scaling::NearBdry::InBall3/4} for \(\WhWd\) and \(\XhXd\)
     follows from Theorem \ref{Thm::Scaling::WithoutBoundary::MainThm} \ref{Item::Scaling::WithoutBdry::BallsInOmega}.
     By possibly shrinking \(\delta_0\),  \ref{Item::Scaling::NearBdry::InBall3/4} holds for \(\ZhZd\)
     as well by the Picard--Lindel\"of Theorem.

    Define, for \(t=(t_1,\ldots,t_n)\) sufficiently small,
    \begin{equation*}
        \psih_{x,\delta}(t)=\exp\left( t_n \Phi_{x,\delta}^{*} \delta^{\Xd_0}\Xh_0 \right)
        \exp\left( t_1 \Phi_{x,\delta}^{*} \delta^{\Xd_{j_1}}\Xh_{j_1}+ \cdots+ t_{n-1} \Phi_{x,\delta}^{*} \delta^{\Xd_{j_{n-1}}}\Xh_{j_{n-1}} \right)0.
    \end{equation*}
    Using that \(\Phi_{x,\delta}(0)=x\) (Theorem \ref{Thm::Scaling::WithoutBoundary::MainThm} \ref{Item::Scaling::WithoutBdry::PhixdeltaOf0}),
    we see 
    \begin{equation}\label{Eqn::Scaling::NearBoundary::psiIsCompPsihAndPhi}
        \psi_{x,\delta} = \Phi_{x,\delta}\circ \psih_{x,\delta}.
    \end{equation}

    By Theorem \ref{Thm::Scaling::WithoutBoundary::MainThm} \ref{Item::Scaling::WithoutBdry::PulledBackSmooth},
    \begin{equation*}
        \left\{ \Phi_{x,\delta}^{*}\delta^{\Xd_k}\Xh_k : 0\leq k\leq q, x\in \overline{\nmoBall{1/2}}, \delta\in (0,\delta_0] \right\}\subset \CinftybSpace[\nUnitBall][\Rn]
    \end{equation*}
    is a bounded set, and standard theorems from ODEs apply to show
    \(\exists \sigma_0\in (0,1]\), \(\forall x\in \overline{\nmoBall{1/2}}\), \(\forall \delta\in (0,\delta_0]\),
    \(\psih_{x,\delta}:\nBall{\sigma_0}\rightarrow \nBall{3/4}\) and
    \begin{equation}\label{Eqn::Scaling::NearBoundary::psihIsBoundedSet}
        \left\{ \psih_{x,\delta} : x\in \overline{\nmoBall{1/2}},\delta\in (0,\delta_0] \right\}
        \subset \CinftybSpace[\nBall{\sigma_0}][\Rn]
    \end{equation}
    is a bounded set.

    Note, \(\partial_{t_l}\big|_{t=0} \psih_{x,\delta}(t)=\Phi_{x,\delta}^{*} \delta^{\Xd_{j_l}}\Xh_{j_l}(0)\).
    Combining this with Theorem \ref{Thm::Scaling::MainResult} \ref{Item::Scaling::WithoutBdry::PulledBackSpan}
    and using that \eqref{Eqn::Scaling::NearBoundary::psihIsBoundedSet} is a bounded set,
    the Inverse Function Theorem shows \(\exists \sigma_1\in (0,\sigma_0/\sqrt{n}]\) with
    \begin{equation*}
        \psih_{x,\delta}:\nCube{\sigma_1} \xrightarrow{\sim}\psih_{x,\delta}(\nCube{\sigma_1})\subseteq \nBall{3/4}
    \end{equation*}
    is a smooth diffeomorphism, \(\psih_{x,\delta}(\nCube{\sigma_1})\subseteq \nBall{3/4}\) is open,
    and, \(\forall \alpha\),
    \begin{equation}\label{Eqn::Scaling::NearBoundary::psiInverseSmooth}
        \sup_{\substack{x\in \overline{\nmoBall{1/2}}\\\delta\in (0,\delta_0]}}\: \sup_{t\in \psih_{x,\delta}(\nCube{\sigma_1})} \left| \partial_t^{\alpha} \psi_{x,\delta}^{-1}(t) \right|<\infty.
    \end{equation}
    This establishes part of \ref{Item::Scaling::NearBdry::MapsAreDiffeos} \ref{Item::Scaling::NearBdry::ImagesOpen} and \ref{Item::Scaling::NearBdry::Diffeos}; namely, 
    \(\psi_{x,\delta}(\nCube{\sigma_1})\subseteq \nBall{3/4}\) is open and \(\psi_{x,\delta}:\nCube{\sigma_1}\xrightarrow{\sim}\psi_{x,\delta}(\nCube{\sigma_1})\subseteq \nBall{3/4}\) is a smooth diffeomorphism.

    Recall, \(\Xh_0\) has a positive \(\partial_{x_n}\) component and by Section \ref{Section::BndryVfsWithFormaLDegrees}, Step \ref{Item::BoundaryVfs::DefineXj},
    \(\Xh_j\big|_{\nmoUnitBall}\) does not have a \(\partial_{x_n}\) for \(j\ne 0\).
    Using this and \eqref{Eqn::Scaling::NearBdry::FormulaForpsixdelta},
    we see that the \(n\)-th component of \(\psi_{x,\delta}(t)\) is positive/negative/zero, precisely when
    \(t_n\) is positive/negative/zero. Using this observation, 
    \ref{Item::Scaling::NearBdry::IntersectionCommutes} and
    the remainder of \ref{Item::Scaling::NearBdry::MapsAreDiffeos}
    follow.

    For \(Y\in \VectorFields{\nUnitBall}\), \(\psi_{x,\delta}^{*} Y = \psih_{x,\delta}^{*}\Phi_{x,\delta}^{*}Y\).
    Using this, Theorem \ref{Thm::Scaling::WithoutBoundary::MainThm} \ref{Item::Scaling::WithoutBdry::PulledBackSmooth},
    that \eqref{Eqn::Scaling::NearBoundary::psihIsBoundedSet} is a bounded set, and \eqref{Eqn::Scaling::NearBoundary::psiInverseSmooth},
    we see
    \begin{equation}\label{Eqn::Scaling::NearBoundary::PulledBackXandWAreBounded}
        \left\{ \psi_{x,\delta}^{*} \delta^{\Xd_j}\Xh_j, \psi_{x,\delta}^{*} \delta^{\Wd_k}\Wh_k : x\in \overline{\nmoBall{1/2}}, \delta\in (0,\delta_0], 0\leq j\leq q, 1\leq k\leq r \right\}
        \subset \CinftybSpace[\nCube{\sigma_1}][\Rn]
    \end{equation}
    is a bounded set.
    Since \(\XhXd\) locally strongly controls \(\ZhZd\) on \(\nUnitBall\), we have
    \begin{equation*}
        \delta^{\Zd_j} \Zh_j= \sum_{\Xd_k\leq \Zd_j} \left( \delta^{\Zd_j-\Xd_k} c_{j}^k \right) \delta^{\Xd_k} \Xh_k, \quad c_j^k\in \CinftySpace[\nUnitBall],
    \end{equation*}
    and so
    \begin{equation}\label{Eqn::Scaling::NearBoundary::Tmp::PullBackZ}
        \psi_{x,\delta}^{*} \delta^{\Zd_j} \Zh_j = \sum_{\Xd_k\leq \Zd_j} \left( \delta^{\Zd_j-\Xd_k} c_{j}^k\circ \psi_{x,\delta} \right) \psi_{x,\delta}^{*} \delta^{\Xd_k}\Xh_k.
    \end{equation}
    Using Theorem \ref{Thm::Scaling::WithoutBoundary::MainThm} \ref{Item::Scaling::WithoutBdry::PulledBackNorm}, we have
    \(\forall L\in \Zgeq\), \(x\in \Compact\), \(\delta\in (0,\delta_0]\),
    with implicit constants depending on \(L\), but not on \(x\) or \(\delta\),
    \begin{equation}\label{Eqn::Scaling::NearBoundary::Tmp::ZCoefsPullBackBounded}
    \begin{split}
         &\CmbNorm{\delta^{\Zd_j-\Xd_k} c_{j}^k\circ \psi_{x,\delta}}{L}[\nCube{\sigma_1}]
        \leq \CmbNorm{ c_{j}^k\circ \psi_{x,\delta}}{L}[\nCube{\sigma_1}]
        \approx \sum_{|\alpha|\leq L}\CmbNorm{\left( \psi_{x,\delta}^{*} \delta^{\Xd} \Xh \right)^{\alpha} c_j^k}{0}[\nCube{\sigma_1}]
        \\&\leq \sum_{|\alpha|\leq L} \CmbNorm{\left( \delta^{\Xd}X \right)^{\alpha}c_j^k}{0}[\nBall{3/4}]
        \leq \sum_{|\alpha|\leq L} \CmbNorm{X^{\alpha}c_j^k}{0}[\nBall{3/4}]
        \lesssim \CmbNorm{c_j^k}{L}[\nBall{3/4}]\lesssim 1.
    \end{split}
    \end{equation}
    Combining \eqref{Eqn::Scaling::NearBoundary::Tmp::ZCoefsPullBackBounded} with \eqref{Eqn::Scaling::NearBoundary::Tmp::PullBackZ}
    shows that
    \begin{equation}\label{Eqn::Scaling::NearBoundary::ZInTermsOfX}
        \psi_{x,\delta}^{*} \delta^{\Zd_j}\Zh_j = \sum_{k}c_j^{k,x,\delta} \psi_{x,\delta}^{*} \delta^{\Xd_k}\Xh_k,
    \end{equation}
    where \(\left\{ c_{j}^{k,x,\delta} :0\leq j,k\leq q, x\in \overline{\nmoBall{1/2}}, \delta\in (0,\delta_0] \right\}\subset \CinftybSpace[\nCube{\sigma_1}]\) is a bounded set.
    We similarly have
    \begin{equation}\label{Eqn::Scaling::NearBoundary::XInTermsOfZ}
        \psi_{x,\delta}^{*} \delta^{\Xd_j}\Xh_j = \sum_{k}\ct_j^{k,x,\delta} \psi_{x,\delta}^{*} \delta^{\Zd_k}\Zh_k,
    \end{equation}
    where  \(\left\{ \ct_{j}^{k,x,\delta} :0\leq j,k\leq q, x\in \overline{\nmoBall{1/2}}, \delta\in (0,\delta_0] \right\}\subset \CinftybSpace[\nCube{\sigma_1}]\) is a bounded set.
    Combining \eqref{Eqn::Scaling::NearBoundary::ZInTermsOfX} with the fact that
    \eqref{Eqn::Scaling::NearBoundary::PulledBackXandWAreBounded} is a bounded set shows
    \begin{equation*}
        \left\{ \psi_{x,\delta}^{*} \delta^{\Zd_j}\Zh_j : x\in \overline{\nmoBall{1/2}}, \delta\in (0,\delta_0], 0\leq j\leq q \right\}
        \subset \CinftybSpace[\nCube{\sigma_1}][\Rn]
    \end{equation*}
    is a bounded set. 
    
    Since \(\Xh_{j_k}\) does not have a \(\partial_{x_n}\) component for \(j_k\ne 0\) (i.e., \(k\ne n\)),
    and since \(\psi_{x,\delta}(\nmoCube{\sigma_1})\subseteq \nBallgeq{3/4}\) (by \ref{Item::Scaling::NearBdry::MapsAreDiffeos} \ref{Item::Scaling::NearBdry::ImagesOpen}),
    the formula \eqref{Eqn::Scaling::NearBdry::FormulaForpsixdelta} for \(\psi_{x,\delta}\)
    shows that if \(Y\in \CinftybSpace[\nUnitBall][\Rn]\) is such that
    \(Y\big|_{\nmoUnitBall}\) does not have a \(\partial_{x_n}\) component, then
    \(\psi_{x,\delta}^{*}Y\big|_{\nmoCube{\sigma_1}}\) does not have a \(\partial_{t_n}\) component.
    In particular (using Section \ref{Section::BndryVfsWithFormaLDegrees}, Step \ref{Item::BoundaryVfs::DefineXj}),
    \(\psi_{x,\delta}^{*} \delta^{\Xd_j} X_j\big|_{\nmoCube{\sigma_1}}\) does not have a \(\partial_{t_n}\)
    component for \(j\ne 0\); this is part of \ref{Item::Scaling::NearBdry::PullBackW1IsDxn}.

    Note that, using  \(\psi_{x,\delta}(\nmoCube{\sigma_1})\subseteq \nBallgeq{3/4}\) (by \ref{Item::Scaling::NearBdry::MapsAreDiffeos} \ref{Item::Scaling::NearBdry::ImagesOpen}),
    and the previous paragraph, for \(j\ne 0\) we have
    \begin{equation}\label{Eqn::Scaling::NearBoundary::PullBackVIsPullBackXRestricted}
        \psi_{x,\delta}^{*} \delta^{\Xd_j} \Xh_j\big|_{\nmoCube{\sigma_1}}
        =\psi_{x,\delta}\big|_{\nmoCube{\sigma_1}}^{*} \delta^{\Vd_j} V_j\in \CinftybSpace[\nmoCube{\sigma_1}][\Rnmo].
    \end{equation}
    Combining this the fact that \eqref{Eqn::Scaling::NearBoundary::PulledBackXandWAreBounded}
    is bounded, we see
    \begin{equation*}
        \left\{ \psi_{x,\delta}\big|_{\nmoCube{\sigma_1}}^{*} \delta^{\Vd_j} V_j : x\in \overline{\nmoBall{1/2}}, \delta\in (0,\delta_0], 1\leq j\leq q \right\}
        \subset \CinftybSpace[\nmoCube{\sigma_1}][\Rnmo]
    \end{equation*}
    is a bounded set.  This completes the proof of \ref{Item::Scaling::NearBdry::UniformlySmooth}.

    That \(\psi_{x,\delta}^{*} \delta^{\Xd_0} \Xh_0=\partial_{x_n}\) follows from \eqref{Eqn::Scaling::NearBdry::FormulaForpsixdelta}.
    Since \((\Xh_0,\Xd_0)=(\Wh_1,\Wd_1)\), this completes the proof of \ref{Item::Scaling::NearBdry::PullBackW1IsDxn}.

    \ref{Item::Scaling::NearBdry::UniformHormander} follows from 
    Theorem \ref{Thm::Scaling::WithoutBoundary::MainThm} \ref{Item::Scaling::WithoutBdry::PulledBackHormander},
    \eqref{Eqn::Scaling::NearBoundary::psiIsCompPsihAndPhi},
    that \eqref{Eqn::Scaling::NearBoundary::psihIsBoundedSet} is a bounded set,
    and \eqref{Eqn::Scaling::NearBoundary::psiInverseSmooth}.
    
    \ref{Item::Scaling::NearBdry::UniformlySpan} \ref{Item::Scaling::NearBdry::UniformlySpan::X}
    follows from Theorem \ref{Thm::Scaling::WithoutBoundary::MainThm} \ref{Item::Scaling::WithoutBdry::PulledBackSpan},
    \eqref{Eqn::Scaling::NearBoundary::psiIsCompPsihAndPhi},
    that \eqref{Eqn::Scaling::NearBoundary::psihIsBoundedSet} is a bounded set,
    and \eqref{Eqn::Scaling::NearBoundary::psiInverseSmooth}.
    \ref{Item::Scaling::NearBdry::UniformlySpan} \ref{Item::Scaling::NearBdry::UniformlySpan::Z}
    follows from \ref{Item::Scaling::NearBdry::UniformlySpan} \ref{Item::Scaling::NearBdry::UniformlySpan::X}
    and \eqref{Eqn::Scaling::NearBoundary::XInTermsOfZ}.
    \ref{Item::Scaling::NearBdry::UniformlySpan} \ref{Item::Scaling::NearBdry::UniformlySpan::V}
    follows from \ref{Item::Scaling::NearBdry::UniformlySpan} \ref{Item::Scaling::NearBdry::UniformlySpan::X},
    \ref{Item::Scaling::NearBdry::PullBackW1IsDxn}, and \eqref{Eqn::Scaling::NearBoundary::PullBackVIsPullBackXRestricted}.

    \ref{Item::Scaling::NearBdry::PullBackDensity}: Using \eqref{Eqn::Scaling::NearBoundary::psiIsCompPsihAndPhi}
    and letting \(h_{x,\delta}\) be as in
    Theorem \ref{Thm::Scaling::WithoutBoundary::MainThm}, we have
    \begin{equation}\label{Eqn::Scaling::NearBdry::fInTermsOfh}
        f_{x,\delta}(t)=\left( h_{x,\delta}\circ \psih_{x,\delta}(t) \right) \left| \det d_t \psih_{x,\delta}(t) \right|.
    \end{equation}
    \ref{Item::Scaling::NearBdry::PullBackDensity}
    follows from \eqref{Eqn::Scaling::NearBdry::fInTermsOfh},
    Theorem \ref{Thm::Scaling::WithoutBoundary::MainThm}\ref{Item::Scaling::WithoutBdry::hIsSmooth},\ref{Item::Scaling::WithoutBdry::hIsPositive},
    that \eqref{Eqn::Scaling::NearBoundary::psihIsBoundedSet} is a bounded set,
    and \eqref{Eqn::Scaling::NearBoundary::psiInverseSmooth}.

    From the definitions, we obtain
    \begin{equation}\label{Eqn::Scaling::NearBoundary::ImageOfHatCCBall}
    \begin{split}
         &\psi_{x,\delta}\left( \BpsideltaWhWd{t}{c} \right)
         =\BWhWd{\psi_{x,\delta}(t)}{c\delta},\quad \psi_{x,\delta}\left( \BpsideltaXhXd{t}{c} \right)
         =\BXhXd{\psi_{x,\delta}(t)}{c\delta}.
    \end{split}
    \end{equation}
    Similarly, using \ref{Item::Scaling::NearBdry::IntersectionCommutes},
    \begin{equation}\label{Eqn::Scaling::NearBoundary::ImageOfCCBall}
        \begin{split}
             &\psi_{x,\delta}\left( \BpsideltaWWd{t}{c} \right)
             =\BWWd{\psi_{x,\delta}(t)}{c\delta},\quad \psi_{x,\delta}\left( \BpsideltaXXd{t}{c} \right)
             \BXXd{\psi_{x,\delta}(t)}{c\delta},
             \\&\psi_{x,\delta}\left( \BpsideltaVVd{t}{c} \right)
             =\BVVd{\psi_{x,\delta}(t)}{c\delta}.
        \end{split}
    \end{equation}

    For the remainder of the proof, we let \(\Rngeq=\left\{ (x_1,\ldots, x_n)\in \Rn : x_n\geq 0 \right\}\)
    and identify \(\Rnmo\) with \(\left\{ (x_1,\ldots, x_{n-1}, 0)\in \Rn \right\}\).

    \ref{Item::Scaling::NearBdry::xBallsInImage}:
    For \(\sigma_2\in (0,\sigma_1)\), by \ref{Item::Scaling::NearBdry::UniformlySmooth}
    and the Picard--Lindel\"of Theorem, \(\exists c_1\in (0,1]\),
    \(\forall x\in \nBall{1/2}\), \(\delta\in (0,\delta_0]\),
    \begin{equation*}
    \begin{split}
         &\BpsideltaWhWd{0}{c_1} \cup \BpsideltaXhXd{0}{c_1}
         \subseteq \nCube{\sigma_2/2},
    \end{split}
    \end{equation*}
    and therefore,
    \begin{equation}\label{Eqn::Scaling::NearBoundary::Findc1}
        \begin{split}
             &\left( \BpsideltaWhWd{0}{c_1} \cup \BpsideltaXhXd{0}{c_1} \right)\cap \Rngeq
             \subseteq \nCubegeq{\sigma_2/2},
        \end{split}
        \end{equation}
    Applying \(\psi_{x,\delta}\) to \eqref{Eqn::Scaling::NearBoundary::Findc1},
    using \(\psi_{x,\delta}(0)=x\) (by \eqref{Eqn::Scaling::NearBdry::FormulaForpsixdelta}), 
    \ref{Item::Scaling::NearBdry::IntersectionCommutes},
    and \eqref{Eqn::Scaling::NearBoundary::ImageOfHatCCBall}, \ref{Item::Scaling::NearBdry::xBallsInImage} follows.

    \ref{Item::Scaling::NearBdry::Containments}:
    In light of \ref{Item::Scaling::NearBdry::UniformlySmooth}, \ref{Item::Scaling::NearBdry::UniformlySpan},
    and \ref{Item::Scaling::NearBdry::UniformHormander},
    Proposition \ref{Prop::Scaling::EquivTop::Interior} applies 
    with \(\WhWd\) replaced by each of
    \(\psiWhWd\), \(\psiXhXd\), and \(\psiVVd\),
    and any \(\iota\)-admissible constants as in the conclusion of Proposition \ref{Prop::Scaling::EquivTop::Interior}
    can be taken independent of \(x\in \nmoBall{1/2}\), \(\delta\in (0,\delta_0]\), with \(a>0\)
    in that proposition replaced by \(\sigma_1\). Here, for \(\psiVVd\) we have replaced \(n\) with \(n-1\).
    Let \(a_1^1>0\) be \(a_1>0\) from that proposition, applied to all these instances.

    Similarly, using \ref{Item::Scaling::NearBdry::PullBackW1IsDxn}, Proposition \ref{Prop::Scaling::EquivTop::Boundary}
    applies with \(\WhWd\) replaced by each of
    \(\psiWhWd\) and \(\psiXhXd\)
    and any \((\iota,c_0)\)-admissible constants as in the conclusion of Proposition \ref{Prop::Scaling::EquivTop::Interior}
    can be taken independent of \(x\in \nmoBall{1/2}\), \(\delta\in (0,\delta_0]\), with \(a>0\)
    in that proposition replaced by \(\sigma_1\).
    Let \(a_1^2>0\)  be \(a_1>0\) from that proposition, applied to all these instances.

    Set \(a_1=\min\{a_1^1,a_1^2\}\). 
    For \(\sigma_2\in (0,\sigma_1)\) set \(a_2=\min\{a_1,\sigma_2\}\).
    Let \(\delta_1^1, \delta_1^2\in (0,1]\)  be \(\delta_1\) from the above applications of
    Propositions \ref{Prop::Scaling::EquivTop::Interior} and \ref{Prop::Scaling::EquivTop::Boundary} (respectively)
    and set \(\deltah_1:=\min\{\delta_1^1,\delta_1^2\}\).
    Let \(a_3^1,a_3^2\in (0,a_2/2]\) be \(a_3\) from the above applications of
    Propositions \ref{Prop::Scaling::EquivTop::Interior} and \ref{Prop::Scaling::EquivTop::Boundary} (respectively)
    with \(\delta=\deltah_1\) and set \(\sigma_3:=\min\{a_3^2,a_3^2\}\).

    We have, from 
    Propositions \ref{Prop::Scaling::EquivTop::Interior} and \ref{Prop::Scaling::EquivTop::Boundary},
    the following containments:
    \begin{enumerate}[(A)]
        \item\label{Item::Scaling::NearBdry::TmpA} 
    \(\forall y\in \nCube{\sigma_2/2}\subseteq \nCube{a_2/2}\),
    \begin{equation*}
    \begin{split}
         &\CubeCentered{y}{\sigma_3}
         \subseteq \CubeCentered{y}{a_3^1}
         \subseteq \BpsideltaWhWd{y}{\deltah_1}\cap \BpsideltaXhXd{y}{\deltah_1}
         \\&\subseteq \BpsideltaWhWd{y}{1}\cap \BpsideltaXhXd{y}{1}.
    \end{split}
    \end{equation*}

    \item \(\nmoCube{\sigma_3}\subseteq \BpsideltaVVd{0}{\deltah_1}\subseteq
    \BpsideltaVVd{0}{1}\), and combining this with \ref{Item::Scaling::NearBdry::TmpA},
        we have
        \begin{equation*}
        \begin{split}
             &\nmoCube{\sigma_3}\subseteq
             \BpsideltaVVd{0}{1}\cap \BpsideltaWhWd{0}{1}\cap \BpsideltaXhXd{0}{1}.
        \end{split}
        \end{equation*}

    \item \(\forall y\in \nCubegeq{\sigma_2/2}\),
    \begin{equation*}
        \begin{split}
        &\CubeCenteredgeq{y}{\sigma_3} \subseteq \CubeCenteredgeq{y}{a_3^2}
        \subseteq \BpsideltaWWd{y}{\deltah_1}\cap \BpsideltaXXd{y}{\deltah_1}
        \\&\subseteq \BpsideltaWWd{y}{1}\cap \BpsideltaXXd{y}{1}.
    \end{split}
    \end{equation*}
    \end{enumerate}

    For \(\sigma_4\in (0,\sigma_3)\),
    using \ref{Item::Scaling::NearBdry::UniformlySmooth},
    the Picard--Lindel\"of Theorem shows \(\exists c_2\in (0,1)\)
    such that:
    \begin{enumerate}[(A),resume]
        \item \(\forall y\in \nCube{\sigma_2/2}\),
        \begin{equation*}
            \BpsideltaWhWd{y}{c_2}
            \cup \BpsideltaXhXd{y}{c_2}
            \subseteq \CubeCentered{y}{\sigma_4}.
        \end{equation*}
        \item \(\forall y\in \nCubegeq{\sigma_2/2}\),
        \begin{equation*}
            \left( \BpsideltaWhWd{y}{c_2}\cap \Rngeq \right)\cup \left( \BpsideltaXhXd{y}{c_2}\cap \Rngeq\right)
            \subseteq \CubeCenteredgeq{y}{\sigma_4}.
        \end{equation*}
        \item\label{Item::Scaling::NearBdry::TmpF}  
            \begin{equation*}
                \BpsideltaVVd{0}{c_2} \cup \left( \BpsideltaXhXd{0}{c_2}\cap \Rnmo \right) \cup\left( \BpsideltaWhWd{0}{c_2}\cap \Rnmo \right)
                \subseteq \nmoCube{\sigma_4}.
            \end{equation*}
    \end{enumerate}

    Applying \(\psi_{x,\delta}\) to \ref{Item::Scaling::NearBdry::TmpA}--\ref{Item::Scaling::NearBdry::TmpF},
    using \eqref{Eqn::Scaling::NearBoundary::ImageOfHatCCBall}, \eqref{Eqn::Scaling::NearBoundary::ImageOfCCBall},
    and \ref{Item::Scaling::NearBdry::IntersectionCommutes}, \ref{Item::Scaling::NearBdry::Containments} follows.
\end{proof}

We close this section with a lemma which will be used in the proof of Theorem \ref{Thm::Sheaves::Metrics::BoundaryMetricEquivalence}.

\begin{lemma}\label{Lemma::Scaling::NearBdry::VVdLessThanXXd}
    \(\exists \delta_1>0\), \(\exists C_0\geq 1\), \(\forall x,y\in \nmoBall{1/2}\),
    \begin{equation*}
        \MetricXXd[x][y]<\delta_1\implies \MetricVVd[x][y]\leq C_0 \MetricXXd[x][y].
    \end{equation*}
\end{lemma}
\begin{proof}
    Let \(\delta_0\in (0,1]\), \(\sigma_1\in (0,1]\) be as in Theorem \ref{Thm::Scaling::NearBdry::MainThm}.
    Set \(\sigma_2=\sigma_1/2\), let \(\sigma_3\in (0,\sigma_2/2]\) be as in Theorem \ref{Thm::Scaling::NearBdry::MainThm} \ref{Item::Scaling::NearBdry::Containments},
    take \(\sigma_4=\sigma_3/2\), and \(c_2\in (0,1)\) as in Theorem \ref{Thm::Scaling::NearBdry::MainThm} \ref{Item::Scaling::NearBdry::Containments}.
    
    Suppose \(\MetricXXd[x][y]<c_2\delta_0\) with \(x,y\in \nmoBall{1/2}\).
    Take any \(\delta\in (0,\delta_0]\) with \(\MetricXXd[x_1][x_2]<c_2\delta\).
    Then we have from Theorem \ref{Thm::Scaling::NearBdry::MainThm} \ref{Item::Scaling::NearBdry::Containments}
    \begin{equation*}
    \begin{split}
         & y\in \BXXd{x}{c_2\delta} \cap \nmoBall{1/2}
         \subseteq \BVVd{x}{\delta},
    \end{split}
    \end{equation*}
    and therefore \(\MetricVVd[x][y]<\delta\).
    Taking the infimum over such \(\delta\) shows \(\MetricVVd[x][y]\leq c_2^{-1} \MetricXXd[x][y]\),
    completing the proof with \(\delta_1=c_2\delta_0\) and \(C_0=c_2^{-1}\).
\end{proof}

    \subsection{Proof of \texorpdfstring{Theorem \ref{Thm::Scaling::MainResult}}{Theorem \ref*{Thm::Scaling::MainResult}}}
    \begin{lemma}\label{Lemma::Scaling::Proof::TwoCases}
    \(\exists\) a compact set \(\QCompact\Subset \BoundaryNncWWd\) and \(\delta_2\in (0,1]\)
    such that \(\forall x\in \Compact\), \(\forall \delta\in (0,\delta_2]\) either:
    \begin{itemize}
        \item \(\BWWd{x}{\delta}\cap \BoundaryN=\emptyset\), or
        \item \(\exists x_0\in \QCompact\) with
            \(\BWWd{x_0}{2\delta}\supseteq \BWWd{x}{\delta}\).
    \end{itemize}
\end{lemma}
\begin{proof}
    Since \(\BoundaryNncWWd\) is a closed subset
    of \(\ManifoldNncWWd\)
    and \(\Compact\Subset \ManifoldNncWWd\) is compact,
    it follows
    \(\Compact \cap\BoundaryN=  \Compact\cap \BoundaryNncWWd\) is a compact subset of \(\BoundaryNncWWd\).
    Let \(V\Subset \BoundaryNncWWd\) be an open, relatively compact set with
    \(\Compact \cap \BoundaryN\subseteq V\). Since \(\BoundaryNncWWd\) is open in \(\BoundaryN\),
    \(V\) is open in \(\BoundaryN\) as well.
    Let \(C:=\BoundaryN\setminus V\), a closed subset of \(\BoundaryN\), with \(C\cap \Compact=\emptyset\).
    As \(C\) is closed and \(\Compact\) is compact, the Picard--Lindel\"of Theorem shows
    \begin{equation*}
        \delta_2:=\inf \left\{ \MetricWWd[x][y] : x\in \Compact, y\in C \right\}>0.
    \end{equation*}

    Set \(\QCompact:=\overline{V}\).
    If \(\delta\in (0,\delta_2]\) then for \(x\in \Compact\),
    \(\BWWd{x}{\delta}\cap \BoundaryN \subseteq V\subseteq \QCompact\).
    We conclude that either \(\BWWd{x}{\delta}\cap \BoundaryN=\emptyset\)
    or \(\exists x_0\in \BWWd{x}{\delta}\cap \QCompact\) and therefore \(\BWWd{x_0}{2\delta}\supseteq \BWWd{x}{\delta}\).
\end{proof}

\begin{proof}[Completion of the proof of Theorem \ref{Thm::Scaling::MainResult}]
    \ref{Item::Scaling::MainResult::InOmega} follows by taking \(\delta_1\) small and using the Picard--Lindel\"of Theorem,
    so we focus on the remainder of the statement of the theorem.

    Let \(\delta_2\in (0,1]\) and \(\QCompact \Subset \BoundaryNncWWd\)
    be as in Lemma \ref{Lemma::Scaling::Proof::TwoCases}. We separate the proof
    which follows into the two cases described in that lemma; we being with the second case, which has
    a more involved proof.

    For each \(x\in \QCompact\), using \(x\in \BoundaryNncWWd\), we may pick \(j_0=j_0(x)\)
    and a neighborhood \(U_x\subseteq \BoundaryN\) of \(x\) such that
    \(\forall y\in U_x\), \(\Wh_{j_0}(y)\not \in \TangentSpace{\BoundaryN}{y}\)
    and \(\Wd_k\leq \Wd_{j_0}\implies \Wh_k(y)\in \TangentSpace{\BoundaryN}{y}\).

    Thus, for each \(x\in \QCompact\), we may find a smooth coordinate chart
    \(g_x:\nUnitBall\xrightarrow{\sim} g_x(\nUnitBall)\subseteq \Omega\)
    such that
    \(g_x(0)=x\),
    \(g_x(\nUnitBall)\) is open in \(\ManifoldM\), \(g_x(\nUnitBallgeq) = g_x(\nUnitBall)\cap \ManifoldN\),
    \(g_x(\nmoUnitBall)=g_x(\nUnitBall)\cap \BoundaryN\), \(g_x(\nUnitBallgeq)\Subset \ManifoldN\) is relatively compact,
    \(g_x^{*}\Wh_1,\ldots, g_x^{*}\Wh_r\in \CinftybSpace[\nUnitBall][\Rn]\),
    \(g_x^{*}\Wh_1,\ldots, g_x^{*}\Wh_r\) satisfy H\"ormander's condition of order \(m=m(x)\in \Zg\) on \(\nUnitBall\),
    if \(b_{x,k}^n\in \CinftybSpace[\nUnitBall]\) is the \(\partial_{t_n}\) component of \(g_x^{*}\Wh_{k}\)
    then
    \begin{equation*}
        \inf_{t'\in \nmoUnitBall} \left| a_{x,j_0}^n(t',0) \right|>0, \quad \Wd_k\leq \Wd_{j_0}\implies a_{x,k}^n(t',0)=0,\forall t'\in\nmoUnitBall,
    \end{equation*}
    and if \(h_x\) is defined by \(g_x^{*} \Volh=h_x(t)\LebDensity\) then
    \(h_x\in \CinftybSpace[\nUnitBall]\) and \(\inf_{t\in \nUnitBall} h_x(t)>0\).

    In short, the hypotheses of Theorem \ref{Thm::Scaling::NearBdry::MainThm}
    hold with \(\Wh_1,\ldots, \Wh_r\) replaced by \(g_x^{*} \Wh_1,\ldots, g_x^{*} \Wh_r\)
    with \(g_x^{*} \Wh_{j_0}\) playing the role of \(\Wh_1\) in that theorem.

    Cover the compact set \(\QCompact\) by a finite collection of sets of the form
    \(g_{x_k}(\nmoBall{1/2})\), \(k=1,\ldots,L\), \(x_k\in \QCompact\).
    As described above, we may apply Theorem \ref{Thm::Scaling::NearBdry::MainThm}
    to each of the lists
    \begin{equation}\label{Eqn::Scaling::Proof::gxkWPullBack}
        \left( g_{x_k}^{*}\Wh_1,\Wd_1 \right),\ldots, \left( g_{x_k}^{*}\Wh_r,\Wd_r \right),
    \end{equation}
    for \(k=1,\ldots, L\).
    Let \(\delta_{0,k}\in (0,1]\), \(\sigma_{1,k}\in (0,1]\),
    and \(\psi_{u,\delta,k}:\nCube{\sigma_{1,k}}\rightarrow \nUnitBall\)
    be 
    \(\delta_0\), \(\sigma_1\), and \(\psi_{u,\delta}\) from
    Theorem \ref{Thm::Scaling::NearBdry::MainThm} when applied to \eqref{Eqn::Scaling::Proof::gxkWPullBack}.

    Let \(\sigma_1:=\min\{\sigma_{1,k}\}\in (0,1)\), \(\sigma_2:=\sigma_1/2\),
    \(\sigma_{3,k}\in (0,\sigma_2/2]\) as in Theorem \ref{Thm::Scaling::NearBdry::MainThm} \ref{Item::Scaling::NearBdry::Containments},
    and
    \(\sigma_3:=\min\{\sigma_{3,k}\}\in (0,\sigma_2/2]\).
    Given \(\eta_1\in (0,1/2)\) (as in \ref{Item::Scaling::MainResult::Containments}), let \(\sigma_4:=\eta_1 \sigma_3\).
    Let \(c_{2,k}\in (0,1)\) be as in Theorem \ref{Thm::Scaling::NearBdry::MainThm} \ref{Item::Scaling::NearBdry::Containments}
    with this choice of \(\sigma_4\) and set \(c_2:=\min\{c_{2,k}\}\in(0,1)\).
    Let \(c_{1,k}\in (0,1)\) be as in Theorem \ref{Thm::Scaling::NearBdry::MainThm} \ref{Item::Scaling::NearBdry::xBallsInImage}
    with \(\sigma_2/2\) playing the role of \(\sigma_2\) there, and set \(c_1:=\min\{c_{1,k}\}\in (0,1)\).

    Set, for \(x\in g_{x_k}(\nmoBall{1/2})\), \(\delta\in (0,\min\{\delta_{0,k}\}]\),
    \begin{equation*}
        \Phi_{x,\delta,k}(t):=g_{x_k}\circ \psi_{g_{x_k}^{-1}(x),\delta,k}(t):\nCube{\sigma_1}\xrightarrow{\sim} \Phi_{x,\delta,k}(\nCube{\sigma_1}).
    \end{equation*}
    By Theorem \ref{Thm::Scaling::NearBdry::MainThm} \ref{Item::Scaling::NearBdry::MapsAreDiffeos} \ref{Item::Scaling::NearBdry::Diffeos},
    \(\Phi_{x,\delta,k}(\nCube{\sigma_1})\) is open in \(\ManifoldM\) and \(\Phi_{x,\delta,k}:\nCube{\sigma_1}\xrightarrow{\sim}\Phi_{x,\delta,k}\)
    is a smooth diffeomorphism.

    We have:
    \begin{enumerate}[(I)]
        \item\label{Item::Scaling::Proofs::Tmp::I} Using Theorem \ref{Thm::Scaling::NearBdry::MainThm} \ref{Item::Scaling::NearBdry::IntersectionCommutes}
            and the properties of \(g_{x_k}\),
            \begin{equation*}
                \Phi_{x,\delta,k}\left( \nCube{\sigma_{1}} \right)\cap \ManifoldN =  \Phi_{x,\delta,k}\left( \nCubegeq{\sigma_{1}} \right),
                \quad
                \Phi_{x,\delta,k}\left( \nCube{\sigma_{1}} \right)\cap \BoundaryN =  \Phi_{x,\delta,k}\left( \nmoCube{\sigma_{1}} \right).
            \end{equation*}
        
        \item\label{Item::Scaling::Proofs::Tmp::II} Using Theorem \ref{Thm::Scaling::NearBdry::MainThm} \ref{Item::Scaling::NearBdry::PullBackW1IsDxn},
            \begin{equation*}
                \Phi_{x,\delta,k}^{*} \delta^{\Wd_{j_0}} \Wh_{j_0}=\pm \partial_{t_n}.
            \end{equation*}
        
        \item\label{Item::Scaling::Proofs::Tmp::III} By Theorem \ref{Thm::Scaling::NearBdry::MainThm} \ref{Item::Scaling::NearBdry::UniformHormander},
            \(\Phi_{x,\delta,k}^{*}\delta^{\Wd_1}\Wh_1,\ldots, \Phi_{x,\delta,k}^{*}\delta^{\Wd_r}\Wh_r\)
            are H\"ormander vector fields on \(\nCube{\sigma_1}\), uniformly for \(k\in \{1,\ldots, L\}\),
            \(x\in g_{x_k}(\nmoBall{1/2})\), \(\delta\in (0,\min\{\delta_{0,l}\}]\).

        \item\label{Item::Scaling::Proofs::Tmp::IV} Define \(\hh_{x,\delta,k}\) by \(\Phi_{x,\delta,k}^{*}\Volh = \hh_{x,\delta,k}\Lambda(x,\delta)\LebDensity\).
            Then, by Theorem \ref{Thm::Scaling::NearBdry::MainThm} \ref{Item::Scaling::NearBdry::PullBackDensity},
            \begin{equation*}
                \left\{ \hh_{x,\delta,k} : k\in \{1,\ldots, L\},  x\in g_{x_k}(\nmoBall{1/2}), \delta\in (0,\min\{\delta_{0,l}\}] \right\}
                \subset \CinftybSpace[\nCube{\sigma_1}]
            \end{equation*}
            is a bounded set and
            \begin{equation*}
                \inf_{\substack{k\in \{1,\ldots, L\}
                \\  x\in g_{x_k}(\nmoBall{1/2})}}
                \inf_{\delta\in (0,\min\{\delta_{0,l}\}]}
                \inf_{t\in \nCube{\sigma_1}}
                \hh_{x,\delta,k}(t)>0.
            \end{equation*}
        
        \item\label{Item::Scaling::Proofs::Tmp::V} 
        By Theorem \ref{Thm::Scaling::NearBdry::MainThm}\ref{Item::Scaling::NearBdry::xBallsInImage},\ref{Item::Scaling::NearBdry::Containments}
        and the properties of \(g_{x_k}\) and the choices of \(c_1\), \(c_2\), \(\sigma_3\), \(\sigma_4\),
        we have
        \(\forall k\in \{1,\ldots,L\}\), \(\forall x\in g_{x_k}\left( \nmoBall{1/2} \right)\), 
            \(\forall \delta\in (0,\min\{\delta_0,l\}]\)
        \begin{equation*}
            \BWhWd{x}{c_1\delta}\cap \ManifoldN \subseteq \Phi_{x,\delta,k}\left( \nCubegeq{\sigma_2/2} \right),
        \end{equation*}
        and \(\forall y\in \BWhWd{x}{c_1\delta}\subseteq \Phi_{x,\delta,k}\left( \nCubegeq{\sigma_2/2} \right)\),
        \begin{equation*}
        \begin{split}
             &\BWhWd{y}{c_2\delta}\cap \ManifoldN
             \subseteq \Phi_{x,\delta,k}\left( \CubeCenteredgeq{\Phi_{x,\delta,k}^{-1}(y)}{\sigma_4} \right)
             \subseteq \Phi_{x,\delta,k}\left( \CubeCenteredgeq{\Phi_{x,\delta,k}^{-1}(y)}{\sigma_3} \right)
             \subseteq \BWWd{y}{\delta}.
        \end{split}
        \end{equation*}
        and 
        \begin{equation*}
            \begin{split}
                 &\BWhWd{y}{c_2\delta}
                 \subseteq \Phi_{x,\delta,k}\left( \CubeCentered{\Phi_{x,\delta,k}^{-1}(y)}{\sigma_4} \right)
                 \subseteq \Phi_{x,\delta,k}\left( \CubeCentered{\Phi_{x,\delta,k}^{-1}(y)}{\sigma_3} \right)
                 \subseteq \BWhWd{y}{\delta}.
            \end{split}
            \end{equation*}

    \end{enumerate}

    For \(x\in g_{x_k}\left( \nmoBall{1/2} \right)\),
    \(\delta\in (0,\min\{\delta_{0,k}]\),
    and \(y\in \BWhWd{x}{c_1\delta}\cap \ManifoldN\subseteq \Phi_{x,\delta,k}(\sigma_2/2)\) set
    \begin{equation*}
        \Psih_{y,x,\delta,k}(t):=\Phi_{x,\delta,k}\left( t+\Phi_{x,\delta,k}^{-1}(y) \right),
    \end{equation*}
    so that \(\Psih_{y,x,\delta,k}(0)=y\).

    Let \(\ch_0=\ch_0(y,x,\delta,k)\) equal \(-1\) times the \(n\)-th component of \(\Phi_{x,\delta,k}^{-1}(y)\).
    Since \(|\Phi_{x,\delta,k}^{-1}(y)|_{\infty}<\sigma_2/2=\sigma_1/4\)
    and \(\sigma_3\leq \sigma_2/2\), we have
    \begin{equation*}
        \Psih_{y,x,\delta,k}:\nCube{\sigma_3}\xrightarrow{\sim} \Psih_{y,x,\delta,k}(\nCube{\sigma_3})\subseteq \ManifoldM.
    \end{equation*}
    Using \ref{Item::Scaling::Proofs::Tmp::I}, we have
    \begin{equation*}
        \Psih_{y,x,\delta,k}(\nCubegeq{\sigma_3}[\ch_0])=\Psih_{y,x,\delta,k}(\nCube{\sigma_3})\cap \ManifoldN,
        \quad
        \Psih_{y,x,\delta,k}(\nCubeeq{\sigma_3}[\ch_0])=\Psih_{y,x,\delta,k}(\nCube{\sigma_3})\cap \BoundaryN.
    \end{equation*}
    By \ref{Item::Scaling::Proofs::Tmp::V}, we have
    \begin{equation*}
    \begin{split}
         &\BWhWd{y}{c_2\delta}\cap \ManifoldN 
         \subseteq \Psih_{y,x,\delta,k}\left( \nCubegeq{\sigma_4}[\ch_0] \right)
         \subseteq \Psih_{y,x,\delta,k}\left( \nCubegeq{\sigma_3}[\ch_0] \right)
         \subseteq \BWWd{y}{\delta}
    \end{split}
    \end{equation*}
    and
    \begin{equation*}
        \begin{split}
             &\BWhWd{y}{c_2\delta}
             \subseteq \Psih_{y,x,\delta,k}\left( \nCube{\sigma_4} \right)
             \subseteq \Psih_{y,x,\delta,k}\left( \nCube{\sigma_3} \right)
             \subseteq \BWhWd{y}{\delta}.
        \end{split}
        \end{equation*}

    For any \(\delta\in (0,\min\{\delta_0,k\}]\), \(y\in \left( \bigcup_{x\in \QCompact} \BWhWd{x}{c_1\delta} \cap \ManifoldN \right)\),
    pick and \(x\in \QCompact\) such that \(y\in \BWhWd{x}{c_1\delta}\cap \ManifoldN\)
    and pick and \(k\) such that \(x\in g_{x_k}(\nmoBall{1/2})\). Set
    \begin{equation*}
        \Psi_{x,\delta}(t):=\Psih_{y,x,\delta,k}(t/\sigma_3).
    \end{equation*}
    With \(c_0=\ch_0/\sigma_3\) and using \(\sigma_4=\eta_1\sigma_3\),
    \ref{Item::Scaling::MainResult::ImageOf0}-\ref{Item::Scaling::MainResults::hxdeltaPositive} follow
    with \(x\) replaced by \(y\).
    Thus, by taking \(\delta_1\leq \min \{\delta_{0,k}\}\), we have proved the entire result,
    except for \ref{Item::Scaling::MainResult::EstimateVolume}, \ref{Item::Scaling::MainResult::EstimateVolumeWedge1}, \ref{Item::Scaling::MainResult::Doubling},
    and \ref{Item::Scaling::MainResult::DoublingWedge} when
    \begin{equation*}
        x\in \left( \bigcup_{x_0\in \QCompact} \BWhWd{x_0}{c_1\delta} \right)\cap \ManifoldN.
    \end{equation*}

    We apply Theorem \ref{Thm::Scaling::WithoutBoundary::MainThm} to \(\WhWd\),
    with \(\Compact\) and \(\Omega\) as given,
     \(\XhXd\)
    as in Remark \ref{Rmk::Scaling::WithoutBoundary::XXdAlwaysExists}, and \(\zeta=1\).
    Let \(\delta_0>0\) and \(\Phi_{x,\delta}\) be as in that theorem.

    Take \(x\in \Compact\), \(\delta\in \left(0, 2\delta_0/c_1\wedge 2\delta_2/c_1\wedge \min\{\delta_{0,k}\}\right]\).
    If \(\exists x_0\in \QCompact\) with \(\BWWd{x}{c_1\delta/2}\subseteq \BWWd{x_0}{c_1\delta}\),
    then \(x\in \BWhWd{x_0}{c_1\delta}\cap \ManifoldN\) and we have established 
    \ref{Item::Scaling::MainResult::ImageOf0}-\ref{Item::Scaling::MainResults::hxdeltaPositive} for this \(x\)
    and \(\delta\).

    Otherwise, by Lemma \ref{Lemma::Scaling::Proof::TwoCases} we have
    \(\BWWd{x}{c_1\delta/2}\cap \BoundaryN=\emptyset\); and therefore
    \(\BWWd{x}{c_1\delta/2}\subseteq \InteriorN\). It follows from the definitions (using that \(\ManifoldN\) is a
    closed co-dimension \(0\)
    embedded submanifold of \(\ManifoldM\)) that in this case
    \(\BWWd{x}{c_1\delta/2}=\BWhWd{x}{c_1\delta/2}\subseteq \InteriorN\).
    Set
    \begin{equation*}
        \psi_{x,\delta}(t):=\Phi_{x,c_1\delta/2}(t/\sqrt{n}).
    \end{equation*}

    Let \(\eta_1\in (0,1/2)\). Let \(\xih_0\) be \(\xi_0\) from 
    Theorem \ref{Thm::Scaling::WithoutBoundary::MainThm} \ref{Item::Scaling::WithoutBdry::HardContainment} with
    \(\eta_0=\eta_1/\sqrt{n}\). We have, with \(\xi_0=c_1\xih_0/2\),
    \begin{equation*}
    \begin{split}
         &\BWhWd{x}{\xi_0 \delta}=\BWhWd{x}{c_1 \xih_0 \delta/2}
         \subseteq \Phi_{x,c_1\delta/2}\left( \nBall{\eta_0} \right)
         =\psi_{x,\delta}(\nBall{\eta_1})
         \\&\subseteq \psi_{x,\delta}(\nCube{\eta_1})
         \subseteq \psi_{x,\delta}(\nCube{1})
         \subseteq \Phi_{x,c_1\delta/2}\left( \nBall{1} \right)
         \subseteq \BWhWd{x}{c_1\delta/2}\subseteq \BWhWd{x}{\delta}.
    \end{split}
    \end{equation*}
    Since \( \BWhWd{x}{c_1\delta/2}=\BWWd{x}{c_1\delta/2}\), we also have
    \begin{equation*}
        \begin{split}
             &\BWhWd{x}{\xi_0 \delta}
            \subseteq \psi_{x,\delta}(\nCube{\eta_1})
             \subseteq \psi_{x,\delta}(\nCube{1})
             \subseteq \BWWd{x}{c_1\delta/2}\subseteq \BWWd{x}{\delta},
        \end{split}
        \end{equation*}
    which establishes \ref{Item::Scaling::MainResult::Containments} with \(c_0=-1\);
    and similarly \ref{Item::Scaling::MainResult::Existencec0} holds with \(c_0=-1\)
    since \(\psi_{x,\delta}(\nCube{1})
             \subseteq \BWWd{x}{c_1\delta/2}\subseteq \InteriorN\).
    \ref{Item::Scaling::MainResult::ImageOf0}, \ref{Item::Scaling::MainResult::Diffeo}, \ref{Item::Scaling::MainResults::UniformHormander},
    \ref{Item::Scaling::MainResults::PulledBackIsPartialxn}, \ref{Item::Scaling::MainResults::hxdeltaSmooth},
    and \ref{Item::Scaling::MainResults::hxdeltaPositive}
    follow from the corresponding results for \(\Phi_{x,c_1\delta/2}\) in Theorem \ref{Thm::Scaling::WithoutBoundary::MainThm}.

    Thus far, we have shown \(\exists \deltah_1>0\) such that all of the result except for
    \ref{Item::Scaling::MainResult::EstimateVolume}, \ref{Item::Scaling::MainResult::EstimateVolumeWedge1}, \ref{Item::Scaling::MainResult::Doubling},
    and \ref{Item::Scaling::MainResult::DoublingWedge} holds for \(x\in \Compact\)
    and \(\delta\in (0,\deltah_1]\).

    \ref{Item::Scaling::MainResult::EstimateVolume}:
    For \(x\in \Compact\), \(\delta\in (0,\deltah_1]\) let \(\xi_1\in (0,1]\) be as in
    \ref{Item::Scaling::MainResult::Containments} with \(\eta_1=1/4\).
    We have,
    \begin{equation*}
    \begin{split}
         & \Vol[\BWWd{x}{\xi_1\delta}]
         \leq \Vol[\BWhWd{x}{\xi_1\delta}\cap \ManifoldN]
         \lesssim \Vol\left( \psi_{x,\delta}\left( \nCubegeq{1}[c_0] \right) \right)
         \lesssim \Vol[\BWWd{x}{\delta}].
    \end{split}
    \end{equation*}
    But, by \ref{Item::Scaling::MainResults::hxdeltaSmooth} and \ref{Item::Scaling::MainResults::hxdeltaPositive},
    \begin{equation*}
        \Vol\left( \psi_{x,\delta}\left( \nCubegeq{1}[c_0] \right) \right)
        \approx \int_{\nCubegeq{1}[c_0]} h_{x,\delta}(t) \Lambda(x,\delta)\: dt
        \approx \Lambda(x,\delta) \LebDensity[\nCubegeq{1}[c_0]]
        \approx \Lambda(x,\delta).
    \end{equation*}
    We conclude
    \begin{equation*}
        \Vol[\BWWd{x}{\xi_1\delta}]\lesssim \Lambda(x,\delta)\lesssim \Vol[\BWWd{x}{\delta}].
    \end{equation*}
    Since \(\Lambda(x,\delta)\approx \Lambda(x,\xi_1\delta)\) (for \(x\in \Compact\) and all \(\delta>0\) by the formula
    for \(\Lambda\)),
    this implies for \(\delta\in (0,\xi_1\deltah_1]\) and \(x\in \Compact\),
    \begin{equation*}
        \Vol[\BWWd{x}{\delta}]\approx \Lambda(x,\delta),
    \end{equation*}
    establishing  \ref{Item::Scaling::MainResult::EstimateVolume} with \(\delta_1:=\xi_1\deltah_1\).

    \ref{Item::Scaling::MainResult::Doubling} follows from \ref{Item::Scaling::MainResult::EstimateVolume}
    and the formula for \(\Lambda(x,\delta)\).

    Using \(x\mapsto \Lambda(x,\delta_1)\) is continuous, we have
    \begin{equation}\label{Eqn::Scaling::Proof::TmpInf}
        \inf_{x\in \Compact}\Lambda(x,\delta_1)>0.
    \end{equation}
    With \eqref{Eqn::Scaling::Proof::TmpInf} in hand, \ref{Item::Scaling::MainResult::EstimateVolumeWedge1}
    follows from \ref{Item::Scaling::MainResult::EstimateVolume}.
    Finally, \ref{Item::Scaling::MainResult::DoublingWedge} follows from \ref{Item::Scaling::MainResult::EstimateVolumeWedge1}
    and the formula for \(\Lambda(x,\delta)\).
\end{proof}

\section{Metrics}
A main consequence of Theorem \ref{Thm::Scaling::MainResult} is that several ways of introducing
metrics from H\"ormander vector fields with formal degrees give locally Lipschitz equivalent metrics both on the
interior and near non-characteristic points on the boundary. In this section, we state and prove these results.
We rephrase these results in the language of sheaves in  Section \ref{Section::Sheaves::Metrics}; and defer discussion of metrics
on the boundary to Section \ref{Section::Sheaves::Metrics} as it is best described in the language of sheaves.

\begin{definition}\label{Defn::Metrics::Intro::Equivalent}
    Let \(X\) be a set and let \(\rho_1,\rho_2:X\times X\rightarrow [0,\infty]\)
    be extended metrics on \(X\). We say \(\rho_1\) and \(\rho_2\) are \textit{Lipschitz equivalent} on \(X\)
    if \(\exists C\geq 1\),
    \begin{equation*}
        C^{-1}\rho_2(x,y)\leq \rho_1(x,y)\leq C\rho_2(x,y),\quad \forall x,y\in X.
    \end{equation*}
\end{definition}

\begin{definition}\label{Defn::Metrics::Intro::LocallyEquivalent}
    Let \(X\) be a topological space and let \(\rho_1,\rho_2:X\times X\rightarrow [0,\infty]\)
    be two extended metrics on \(X\) (not necessarily compatible with the topology).
    We say \(\rho_1\) and \(\rho_2\) are \textit{locally Lipschitz equivalent} on \(X\)
    if \(\forall x\in X\), there exists an open neighborhood \(U\) of \(x\),
    such that \(\rho_1\big|_{U\times U}\) and \(\rho_2\big|_{U\times U}\)
    are equivalent on \(U\).
\end{definition}

\begin{remark}\label{Rmk::Metrics::lemmas::ExtendedMetricTop}
    An extended metric induces a topology on a set \(X\)
    in the same way a metric does: \(U\subseteq X\) is open if and only if \(\forall x\in U\), \(\exists \delta>0\), \(B_{\rho}(x,\delta):=\{y\in X : \rho(x,y)<\delta\}\subseteq U\).
    The extended metric is finite (and therefore a metric) on each connected component (with respect to this topology).
    Indeed, if \(X_0\) is a connected component of \(X\), then
    for each \(x\in X_0\), \(\left\{ y\in X_0 : \rho_1(x,y)<\infty \right\}\)
    and \(\left\{ y\in X_0 :\rho(x,y)=\infty \right\}\) are disjoint open sets whose union is \(X_0\);
    and therefore one is empty.
    Since \(x\in \left\{ y\in X_0 : \rho(x,y)<\infty \right\}\), we conclude
    \(\left\{ y\in X_0 :\rho_1(x,y)<\infty \right\}=X_0\).
\end{remark}


    \subsection{Main results}\label{Section::Metrics::MainResults}
    Let \(\ManifoldN\) be manifold (possibly with boundary\footnote{When \(\ManifoldN\) does not have boundary, then
\(\ManifoldNncWWd=\ManifoldN\): every boundary point is \(\WWd\)-non-characteristic, vacuously.}) and let
\(\WWd=\left\{ \left( W_1,\Wd_1 \right),\ldots, \left( W_r,\Wd_r \right) \right\}\subseteq \VectorFieldsN\times \Zg\)
be H\"ormander vector fields with formal degrees on \(\ManifoldN\).
As in \eqref{Eqn::BasicDefns::CCBallAndMetric} we obtain an extended metric \(\MetricWWd\) on \(\ManifoldN\).

\begin{theorem}\label{Thm::Metrics::Results::GivesUsualTopology}
    \(\MetricWWd\) is a metric on each connected component of \(\ManifoldNncWWd\) and the  topology induced by \(\MetricWWd\)
    agrees with the topology on \(\ManifoldNncWWd\) as a manifold.
\end{theorem}

\begin{theorem}\label{Thm::Metrics::Results::LocalWeakEquivImpliesEquiv}
    Let \(\ManifoldNt\subseteq \ManifoldNncWWd\) be an open subset and let 
    \(\ZZd=\left\{ \left( Z_1,\Zd_1 \right),\ldots, \left( Z_s,\Zd_s \right) \right\}\subset \ManifoldNt\times \Zg\) 
    be H\"ormander vector fields
    with formal degrees on \(\ManifoldNt\). Suppose \(\ZZd\) and \(\WWd\) are locally weakly equivalent on \(\ManifoldNt\).
    Then, \(\MetricWWd\) and \(\MetricZZd\) are locally Lipschitz equivalent on \(\ManifoldNt\).
\end{theorem}

\begin{theorem}\label{Thm::Metrics::Results::ExtendedVFsGiveSameMetric}
    Suppose \(\ManifoldM\) is a manifold without boundary such that \(\ManifoldN\subseteq \ManifoldM\)
    is a closed, co-dimension \(0\) embedded submanifold of \(\ManifoldM\). Suppose \(\WhWd=\left\{ ( \Wh_1,\Wd_1 ),\ldots, (\Wh_r,\Wd_r) \right\}\subset \VectorFieldsM\times \Zg\)
    are H\"ormander vector fields with formal degrees on \(\ManifoldM\) such that \(\Wh_j\big|_{\ManifoldN}=W_j\).
    Then, \(\MetricWhWd\) and \(\MetricWWd\) are locally Lipschitz equivalent on \(\ManifoldNncWWd\).
\end{theorem}

    \subsection{Some lemmas about metrics}
    In this section, we state and prove some simple lemmas about metrics which are used in the proofs of the results
from Section \ref{Section::Metrics::MainResults}.

\begin{lemma}\label{Lemma::Metrics::Lemmas::QuantLocalEst}
    Let \(X\) be a set and \(\rho_1,\rho_2:X\times X\rightarrow [0,\infty)\) 
    be two metrics on \(X\). Suppose 
    \begin{itemize}
        \item \(X\) is compact with respect to \(\rho_1\),
        \item \(\exists \delta_0>0\), \(\exists C_1\geq 1\), \(\forall x,y\in X\)
            \begin{equation*}
                \rho_2(x,y)<\delta_0\implies \rho_1(x,y)\leq C_1\rho_2(x,y).
            \end{equation*}
    \end{itemize}
    Then, \(\exists C_2\geq 1\), such that
    \begin{equation*}
        \rho_1(x,y)\leq C_2\rho_2(x,y),\quad \forall x,y\in X.
    \end{equation*}
\end{lemma}
\begin{proof}
    Since \(X\) is compact with respect to
    \(\rho_1\), 
    \begin{equation*}
        C_3:=\left\{ \rho_1(x,y) :x,y\in X \right\}<\infty.
    \end{equation*}
    We have,
    \begin{equation*}
\begin{split}
            &\rho_1(x,y)
            \leq
            \begin{cases}
                C_1\rho_2(x,y), & \rho_2(x,y)<\delta_0\\
                C_3, &\rho_2(x,y)\geq \delta_0
            \end{cases}
            \\&\leq \max\left\{ C_1, C_3/\delta_0 \right\} \rho_2(x,y).
    \end{split}    
\end{equation*}
\end{proof}

\begin{lemma}\label{Lemma::Metrics::Lemmas::CompactLocalEquivMeansEquiv}
    Let \(X\) be a compact topological space and let \(\rho_1\) and \(\rho_2\) be two metrics on \(X\)
    such that the metric topology induced by either metric is the topology on \(X\).
    Suppose \(\rho_1\) and \(\rho_2\) are locally Lipschitz equivalent on \(X\). Then, \(\rho_1\)
    and \(\rho_2\) are Lipschitz equivalent on \(X\).
\end{lemma}
\begin{proof}
    We will show \(\rho_1(x,y)\lesssim \rho_2(x,y)\), \(\forall x,y\in X\); and the reverse inequality
    follows by symmetry.

    Since \(\rho_1\) and \(\rho_2\) are locally Lipschitz equivalent, \(\forall x\in X\), \(\exists \delta_x>0\), \(\exists C_x\geq 1\),
    such that \(\rho_2(x,y)<\delta_x\implies \rho_1(x,y)\leq C_x\rho_2(x,y)\).
    Extract from the cover \(B_{\rho_2}(x,\delta_x)\) a finite subcover
    \(B_{\rho_2}(x_j, \delta_{x_j})\), \(j=1,\ldots,L\).
    Set \(C_1=\max C_{x_j}\). Then if \(\delta_0>0\) is a Lebesgue number (with respect to \(\rho_2\)) for this subcover,
    we have \(\rho_2(x,y)<\delta_0\implies \rho_1(x,y)\leq C_1\rho_2(x,y)\).

    Since \(X\) is compact with respect to \(\rho_1\), by assumption, the claim \(\rho_1(x,y)\lesssim \rho_2(x,y)\), \(\forall x,y\in X\),
    follows from Lemma \ref{Lemma::Metrics::Lemmas::QuantLocalEst}.
\end{proof}

\begin{lemma}\label{Lemma::Metrics::Lemmas::LocallyEquivIffEquivOnCptSets}
    Let \(X\) be a locally compact topological space and let \(\rho_1\) and \(\rho_2\) be two extended metrics on \(X\)
    such that the topology induced by either extended metric
    is the topology on \(X\).
    Then, \(\rho_1\) and \(\rho_2\) are locally Lipschitz equivalent on \(X\)
    if and only if for each  compact set \(\Compact\Subset X\) lying in a single connected component of \(X\),
    \(\rho_1\) and \(\rho_2\) are Lipschitz equivalent on \(\Compact\).
\end{lemma}
\begin{proof}
    Suppose \(\rho_1\) and \(\rho_2\) are locally Lipschitz equivalent on \(X\),
    \(X_0\) is a connected component of \(X\),
    and \(\Compact\Subset X_0\)
    is a compact set. Then, \(\rho_1\) and \(\rho_2\) are both finite on \(X_0\) (see Remark \ref{Rmk::Metrics::lemmas::ExtendedMetricTop}).
    Thus, \(\rho_1\) and \(\rho_2\) are metrics on \(\Compact\), are locally Lipschitz equivalent, and the metric topologies
    induced \(\rho_1\) and \(\rho_2\) agree with the subspace topology on \(\Compact\).
    Lemma \ref{Lemma::Metrics::Lemmas::QuantLocalEst} shows \(\rho_1\) and \(\rho_2\)
    are Lipschitz equivalent on \(\Compact\).

    The reverse implication follows immediately from the definitions.
\end{proof}

    \subsection{Proof of \texorpdfstring{Theorem \ref{Thm::Metrics::Results::GivesUsualTopology}}{Theorem \ref*{Thm::Metrics::Results::GivesUsualTopology}}}
    Let \(U\subseteq \ManifoldNncWWd\) be open and let \(x\in U\).
The Picard-Lindel\"of Theorem shows \(\exists \delta>0\) with \(\BWWd{x}{\delta}\subseteq U\).
We conclude that the manifold topology on \(\ManifoldNncWWd\) is coarser than the topology
induced by the extended metric \(\MetricWWd\).

Let \(x\in \ManifoldNncWWd\) and \(\delta>0\). We wish to show \(\exists U\subseteq \ManifoldNncWWd\)
open (in the manifold topology) such that \(x\in U\subseteq \BWWd{x}{\delta}\);
it suffices to do this for \(\delta\) small.
The existence of \(U\) follows from Theorem \ref{Thm::Scaling::MainResult}\ref{Item::Scaling::MainResult::Diffeo},\ref{Item::Scaling::MainResult::Existencec0},\ref{Item::Scaling::MainResult::Containments}
(with \(\Compact=\{x\}\)):
here we take \(U= \Psi_{x,\delta}(\nUnitCubegeq[c_0])\).
See also Remark \ref{Rmk::Scaling::MainResult::DontNeedAmbientMfld}.
This shows that the manifold topology on \(\ManifoldNncWWd\) is finer than the topology
induced by the extended metric \(\MetricWWd\); and therefore the topologies are the same.

Alternatively, instead of applying Theorem \ref{Thm::Scaling::MainResult},
one could directly apply Propositions \ref{Prop::Scaling::EquivTop::Interior} and \ref{Prop::Scaling::EquivTop::Boundary}
in local coordinates (with \(y=0\))
to conclude the that the manifold topology on \(\ManifoldNncWWd\) is finer than the topology
induced by the extended metric \(\MetricWWd\). These propositions were used in the proof of Theorem \ref{Thm::Scaling::MainResult},
so this amounts to the same proof.

Finally, that \(\MetricWWd\) is a metric on each connected component (instead of merely an extended metric)
follows from Remark \ref{Rmk::Metrics::lemmas::ExtendedMetricTop}.

    \subsection{Proof of \texorpdfstring{Theorem \ref{Thm::Metrics::Results::ExtendedVFsGiveSameMetric}}{Theorem \ref*{Thm::Metrics::Results::ExtendedVFsGiveSameMetric}}}
    Fix \(\Compact\Subset \ManifoldNncWWd\) a compact set lying in a single connected component of \(\ManifoldNncWWd\).

By Lemma \ref{Lemma::Metrics::Lemmas::LocallyEquivIffEquivOnCptSets}, we wish to show
\(\MetricWWd\) and \(\MetricWhWd\) are Lipschitz equivalent on \(\Compact\).
It follows immediately from the definitions that \(\MetricWhWd[x][y]\leq \MetricWWd[x][y]\), \(\forall x,y\in \ManifoldN\),
so it suffices to show
\begin{equation}\label{Eqn::Metrics::Proof::IntrinsicVsExtrinsic::MainToShow}
    \MetricWWd[x][y]\leq C\MetricWhWd[x][y],\quad \forall x,y\in \Compact.
\end{equation}

We apply Theorem \ref{Thm::Scaling::MainResult} with this choice of \(\Compact\) (any \(\Omega\) will work).
Let \(\delta_1\) be as in that result and \(\xi_1\) be as in Theorem \ref{Thm::Scaling::MainResult} \ref{Item::Scaling::MainResult::Containments}
with \(\eta_1=1/4\).

Suppose \(x,y\in \Compact\) with \(\MetricWhWd[x][y]<\xi_1\delta_1\). Take any \(\delta_0\in (0,\delta_1]\)
with \(\MetricWhWd[x][y]<\xi_1\delta\). Then, by Theorem \ref{Thm::Scaling::MainResult} \ref{Item::Scaling::MainResult::Containments},
\begin{equation*}
    y\in \BWhWd{x}{\xi_1\delta}\cap \ManifoldN\subseteq \BWWd{x}{\delta},
\end{equation*}
and so \(\MetricWWd[x][y]<\delta\). Taking the infimum over such \(\delta\)
shows \(\MetricWWd[x][y]\leq \xi_1^{-1}\MetricWhWd[x][y]\).

Since \(\MetricWWd\) is a metric on \(\Compact\) and \(\Compact\) is compact with respect to \(\MetricWWd\)
(by Theorem \ref{Thm::Metrics::Results::GivesUsualTopology}),
Lemma \ref{Lemma::Metrics::Lemmas::QuantLocalEst} shows \eqref{Eqn::Metrics::Proof::IntrinsicVsExtrinsic::MainToShow}
holds, completing the proof.

    \subsection{Proof of \texorpdfstring{Theorem \ref{Thm::Metrics::Results::LocalWeakEquivImpliesEquiv}}{Theorem \ref*{Thm::Metrics::Results::LocalWeakEquivImpliesEquiv}}}
    Fix \(\Compact\Subset \ManifoldNt\) compact, lying in a single connected component of \(\ManifoldNt\).
In light of Lemma \ref{Lemma::Metrics::Lemmas::LocallyEquivIffEquivOnCptSets} we wish to show that
\(\MetricWWd\) and \(\MetricZZd\) are Lipschitz equivalent on \(\Compact\).
Note that \(\BoundaryNt\subseteq \BoundaryNncWWd\) and so every point of \(\BoundaryNt\) is \(\WWd\)-non-Characteristic.
Since \(\WWd\) and \(\ZZd\) are weakly locally equivalent on \(\ManifoldNt\), Remark \ref{Rmk::BasicDefns::NonCharDependsOnlyOnWeakEquiv}
shows that every point of \(\BoundaryNt\) is also \(\ZZd\)-non-characteristic.

Let \(\ManifoldM\) be a manifold without boundary, such that \(\ManifoldN\) is a closed, co-dimension \(0\)
embedded submanifold of \(\ManifoldM\) and such that there are H\"ormander vector fields with formal degrees
\(\WhWd=\left\{ (\Wh_1,\Wd_1),\ldots, (\Wh_r,\Wd_r) \right\}\subset \VectorFieldsM\times \Zg\)
such that \(\Wh_j\big|_{\ManifoldN}=W_j\). Such an \(\ManifoldM\) and \(\WhWd\)
always exist: see Remark \ref{Rmk::Scaling::MainResult::DontNeedAmbientMfld}.

Fix \(\Omega\Subset \ManifoldM\) open and relatively compact with \(\Compact\Subset \Omega\).
Since \(\overline{\Omega}\) is compact, there is \(m\in \Zg\) such that
\(\Wh_1,\ldots, \Wh_r\) satisfy H\"ormander's condition of order \(m\) on \(\overline{\Omega}\).
Let
\begin{equation*}
\begin{split}
     &\XhXd=\left\{ (\Xh_1,\Xd_1),\ldots, (\Xh_q,\Xd_q) \right\}
     :=\left\{ (Y,e) \in \GenWhWd : e\leq \max\{\Wd_j\} \right\}.
\end{split}
\end{equation*}
As in \eqref{Eqn::BoundaryVfs::NSW::Z},
\begin{equation}\label{Eqn::Metrics::Proofs::XNSW}
    [\Xh_j,\Xh_k]=\sum_{\Xd_l\leq \Xd_j+\Xd_k} c_{j,k}^l \Xh_l,\quad c_{j,k}^l\in \CinftySpace[\Omega].
\end{equation}

\begin{lemma}\label{Lemma::Metrics::Proofs::LocalWeakEquiv::WhWdBoundedByXhXd}
    \(\exists C\geq 1\), \(\forall x,y\in \Compact\), \(\MetricWhWd[x][y]\leq C\MetricXhXd[x][y]\).
\end{lemma}
\begin{proof}
    Theorem \ref{Thm::Scaling::WithoutBoundary::MainThm} applies with \(\WhWd\) and \(\XhXd\) in place
    of \(\WWd\) and \(\XXd\) in that result (see Remark \ref{Rmk::Scaling::WithoutBoundary::XXdAlwaysExists}).
    Let \(\delta_0>0\), be as in that theorem and \(\xi_0>0\) be as in Theorem \ref{Thm::Scaling::WithoutBoundary::MainThm} \ref{Item::Scaling::WithoutBdry::HardContainment}
    with \(\eta_0=1/2\).

    Suppose \(x,y\in \Compact\), \(\MetricXhXd[x][y]<\xi_0\delta_0\). Take \(\delta\in (0,\delta_0]\)
    with \(\MetricXhXd[x][y]<\xi_0\delta\). Then, by Theorem \ref{Thm::Scaling::WithoutBoundary::MainThm} \ref{Item::Scaling::WithoutBdry::HardContainment},
    \begin{equation*}
        y\in \BXhXd{x}{\xi_0\delta}\subseteq \BWhWd{x}{\delta},
    \end{equation*}
    and therefore \(\MetricWhWd[x][y]<\delta\). Taking the infimum over such \(\delta\)
    shows \(\MetricWhWd[x][y]\leq \xi_0^{-1}\MetricXhXd[x][y]\).

    Since \(\Compact\) is a compact subset of \(\ManifoldNt\), it is a compact subset of \(\ManifoldM\),
    and by Theorem \ref{Thm::Metrics::Results::GivesUsualTopology} it is compact with respect to
    \(\MetricWhWd\). From here, the result follows from Lemma \ref{Lemma::Metrics::Lemmas::QuantLocalEst}.
\end{proof}

\begin{lemma}\label{Lemma::Metrics::Proofs::LocalWeakEquiv::WWdBoundedByWhWd}
    \(\exists C\geq 1\), \(\forall x,y\in \Compact\), \(\MetricWWd[x][y]\leq C\MetricWhWd[x][y]\).
\end{lemma}
\begin{proof}

    Since \(\Compact\) lies in a single connected component of \(\ManifoldNt\), it lies in a single connected
    component of \(\ManifoldM\), and by Theorem \ref{Thm::Metrics::Results::GivesUsualTopology}
    both \(\MetricWhWd\) and \(\MetricWWd\) are metrics on \(\Compact\), and the metric topology
    induced by both is the same as the topology of \(\Compact\) as a subset of \(\ManifoldNt\).

    By Theorem \ref{Thm::Metrics::Results::ExtendedVFsGiveSameMetric}, \(\MetricWWd\) and \(\MetricWhWd\)
    are locally Lipschitz equivalent on \(\Compact\), and therefore by Lemma \ref{Lemma::Metrics::Lemmas::CompactLocalEquivMeansEquiv}
    they are Lipschitz equivalent on \(\Compact\) (since it is compact). The result follows.
\end{proof}

\begin{lemma}\label{Lemma::Metrics::Proofs::LocalWeakEquiv::XhXdBoundedByZZd}
    \(\exists C\geq 1\), \(\forall x,y\in \Compact\), \(\MetricXhXd[x][y]\leq C \MetricZZd[x][y]\).
\end{lemma}
\begin{proof}
    Since \(\XhXd\) satisfies \eqref{Eqn::Metrics::Proofs::XNSW}, it follows
    that \(\XhXd\) locally strongly controls \(\GenXhXd=\GenWhWd\) on \(\Omega\).
    Fix \(\Omega_1\Subset\Omega_2\Subset \Omega_3\Subset \Omega\) open in \(\ManifoldM\),
    with \(\Compact \Subset \Omega_1\) (recall, \(A\Subset B\) means \(A\) is relatively compact in \(B\)).

    Since \(\WWd\) locally weakly controls \(\ZZd\) on \(\ManifoldNt\), \(\GenWWd\) locally strongly
    controls \(\ZZd\) on \(\ManifoldN\). And so, for each \(k=1,\ldots,s\), we have 
    \begin{equation}\label{Eqn::Metrics::Proofs::LocalWeakEquiv::XhXdBoundedByZZd::Tmp1}
        Z_k\big|_{\Omega_3\cap \ManifoldNt} = \sum_{\substack{ e\leq \Zd_k \\ (Y,e)\in \GenWWd }} a_{(Y,e)} Y, \quad a_{(Y,e)}\in \CinftySpace[\Omega_3\cap \ManifoldNt].
    \end{equation}
    Extending \(a_{(Y,e)}\) to \(\at_{(Y,e)}\in \CinftySpace[\Omega_3]\), \eqref{Eqn::Metrics::Proofs::LocalWeakEquiv::XhXdBoundedByZZd::Tmp1}
    implies 
    \begin{equation}\label{Eqn::Metrics::Proofs::LocalWeakEquiv::XhXdBoundedByZZd::Tmp2}
        Z_k\big|_{\Omega_3\cap \ManifoldNt} = \sum_{\substack{ e\leq \Zd_k \\ (\Yh,e)\in \GenWhWd }} \at_{(Y,e)}\big|_{\Omega_3 \cap \ManifoldNt} \Yh\big|_{\Omega_3\cap\ManifoldNt}, \quad \at_{(Y,e)}\in \CinftySpace[\Omega_3\cap \ManifoldNt].
    \end{equation}
    Since \(\XXd\) locally strongly controls \(\WhWd\) on \(\Omega\), \eqref{Eqn::Metrics::Proofs::LocalWeakEquiv::XhXdBoundedByZZd::Tmp2}
    implies
    \begin{equation}\label{Eqn::Metrics::Proofs::LocalWeakEquiv::XhXdBoundedByZZd::Tmp3}
        Z_k\big|_{\Omega_2\cap \ManifoldNt} = \sum_{\Xd_l\leq \Zd_k} a_k^l\big|_{\Omega_2\cap \ManifoldNt} \Xh_l\big|_{\Omega_2\cap \ManifoldNt},
        \quad a_k^l\in \CinftySpace[\Omega_2\cap \ManifoldN].
    \end{equation}
    Set
    \begin{equation}\label{Eqn::Metrics::Proofs::LocalWeakEquiv::XhXdBoundedByZZd::Tmp4}
        C_1=\sup_{\substack{k,l \\ x\in \Omega_1}} |a_k^l(x)|<\infty.
    \end{equation}
    Since \(\Compact\) is compact in \(\ManifoldNt\) and \(\ManifoldNt\setminus \Omega_1\) is closed
    in \(\ManifoldNt\), the Picard--Lindel\"of Theorem shows
    \begin{equation*}
        \delta_3:=\inf\left\{ \MetricZZd[x][y] : x\in \Compact, y\in \ManifoldNt\setminus \Omega_1 \right\}>0.
    \end{equation*}

    Suppose \(x,y\in \Compact\), with \(\MetricZZd[x][y]<\min\{\delta_3,1\}\).
    Take and \(\delta\leq \min\{\delta_3,1\}\) with \(\MetricZZd[x][y]<\delta\).
    Then, \(y\in \BZZd{x}{\delta}\), and by definition (see \eqref{Eqn::BasicDefns::UnitCCBall} and \eqref{Eqn::BasicDefns::CCBallAndMetric})
    \(\exists \gamma:[0,1]\rightarrow \ManifoldNt\), absolutely continuous, \(\gamma(0)=x\), \(\gamma(1)=y\),
    \(\gamma'(t)=\sum_{k=1}^s b_k(t) \delta^{\Zd_k} Z_k(\gamma(t))\) for almost every \(t\),
    with \(\sum|b_k(t)|^2<1\) for almost every \(t\).
    Note that \(\gamma([0,1])\subseteq \BZZd{x}{\delta}\subseteq \Omega_1\).
    Combining this with \eqref{Eqn::Metrics::Proofs::LocalWeakEquiv::XhXdBoundedByZZd::Tmp3}, we see for \(C_2\geq 1\) to be chosen later,
    \begin{equation*}
    \begin{split}
         &\gamma'(t)
         =\sum_{k=1}^s \sum_{\Xd_l \leq \Zd_k} b_k(t) a_k^l(\gamma(t)) \delta^{\Zd_k} \Xh_l(\gamma(t))
         =\sum_{l=1}^q \sum_{\Xd_l\leq \Zd_k} \left( b_k(t) a_k^l(\gamma(t)) \delta^{\Zd_k-\Xd_l} \right) \delta^{\Xd_l}\Xh_l(\gamma(t))
         \\&=\sum_{l=1}^q \sum_{\Xd_l\leq \Zd_k} \left( C_2^{-\Xd_l} b_k(t) a_k^l(\gamma(t)) \delta^{\Zd_k-\Xd_l} \right) \left( C_2 \delta \right)^{\Xd_l}\Xh_l(\gamma(t)).
    \end{split}
    \end{equation*}
    By taking \(C_2=C_1^{-1}s^{-1}q^{-1/2}\) and using \eqref{Eqn::Metrics::Proofs::LocalWeakEquiv::XhXdBoundedByZZd::Tmp4}
    we have for almost every \(t\)
    \begin{equation*}
        \sum_{l=1}^q \left| \sum_{\Xd_l\leq \Zd_k} \left( C_2^{-\Xd_l} b_k(t) a_k^l(\gamma(t)) \delta^{\Zd_k-\Xd_l} \right) \right|^2
        < \sum_{l=1}^q \left|  C_2^{-\Xd_l} s C_1   \right|^2<1.
    \end{equation*}
    It follows that \(\MetricXhXd[x][y]< C_2 \delta\). Taking the infimum over such \(\delta\) shows
    \(\MetricXhXd[x][y]\leq C_2 \MetricZZd[x][y]\).

    Since \(\Compact\) is compact with respect to \(\MetricXhXd\) 
    and \(\MetricXhXd\) is a metric on \(\Compact\)
    (see Theorem \ref{Thm::Metrics::Results::GivesUsualTopology}),
    the result follows from Lemma \ref{Lemma::Metrics::Lemmas::QuantLocalEst}.
\end{proof}

\begin{lemma}\label{Lemma::Metrics::Proofs::LocalWeakEquiv::WWdBoundedByZZd}
    \(\exists C\geq 1\), \(\forall x,y\in \Compact\), \(\MetricWWd[x][y]\leq C \MetricZZd[x][y]\).
\end{lemma}
\begin{proof}
    This follows by combining Lemmas \ref{Lemma::Metrics::Proofs::LocalWeakEquiv::WhWdBoundedByXhXd},
    \ref{Lemma::Metrics::Proofs::LocalWeakEquiv::WWdBoundedByWhWd}, and \ref{Lemma::Metrics::Proofs::LocalWeakEquiv::XhXdBoundedByZZd}.
\end{proof}

\begin{lemma}\label{Lemma::Metrics::Proofs::LocalWeakEquiv::ZZdBoundedByWWd}
    \(\exists C\geq 1\), \(\forall x,y\in \Compact\), \(\MetricZZd[x][y]\leq C \MetricWWd[x][y]\).
\end{lemma}
\begin{proof}
    The same proof used to show Lemma \ref{Lemma::Metrics::Proofs::LocalWeakEquiv::WWdBoundedByZZd}
    goes through with the roles of \(\WWd\) and \(\ZZd\) reversed.
\end{proof}

The proof of Theorem \ref{Thm::Metrics::Results::LocalWeakEquivImpliesEquiv} is completed by combining
Lemmas \ref{Lemma::Metrics::Proofs::LocalWeakEquiv::WWdBoundedByZZd} and \ref{Lemma::Metrics::Proofs::LocalWeakEquiv::ZZdBoundedByWWd}.

\section{Sheaves and metrics}\label{Section::Sheaves}
Our results concerning metrics are more naturally stated in terms
of certain \(\Zg\)-filtrations of sheaves of \(\CinftySpace\)-modules of smooth vector fields.
This is particularly true for the boundary metric: we will see (in Section \ref{Section::Sheaves::Restriction::NCBdry}) that H\"ormander vector fields
with formal degrees naturally define a \(\Zg\)-filtration of sheaves of modules of vector fields on
the non-characeteristic boundary satisfying good properties. This will lead directly to a Lipschitz equivalence
class of metrics on the boundary--generalizing and making precise Corollary \ref{Cor::GlobalCor::EquivBoundary}.
Unfortunately, we require a number of definitions and simple lemmas before we can state the relevant results
in Section \ref{Section::Sheaves::Metrics}.
On a first reading, the reader may wish to directly skip to the results in Section \ref{Section::Sheaves::Metrics},
where the basic ideas may be clear, even if the reader has not seen all the relevant definitions.

\begin{remark}\label{Rmk::Sheaves::SheavesAndDistributionsEquiv}
    It is common in the literature on maximal subellipticity to work
    with globally defined \(\CinftySpace[\ManifoldN]\)-modules of vector fields
    on \(\ManifoldN\), instead of sheaves
    of vector fields as in Definition \ref{Defn::Sheaves::Sheaves::Sheaf}.  
    While we could proceed in this way, taking a globally defined
    \(\CinftySpace[\ManifoldN]\)-module of vector fields, \(\sD\),
    requires one to provide more information than we use.
    Indeed, given such a module \(\sD\), one could take the smallest
    sheaf of vector fields \(\SheafF\) on \(\ManifoldN\) with \(\sD\subseteq \SheafF[\ManifoldN]\),
    and apply our theory. Conversely, given a sheaf \(\SheafF\) one obtains a globally
    defined module \(\sD=\SheafF[\ManifoldN]\) which uniquely determines \(\SheafF\)
    (see Lemma \ref{Lemma::Sheaves::Sheaves::CheckGeneratorsAsGlobalSections}).
    However, (when the manifold is not compact) not every \(\sD\) arises as some \(\SheafF[\ManifoldN]\).\footnote{To understand
    this, consider the simpler case of functions. The sheaf, \(\CinftySpace\), of smooth functions on a manifold
    is the smallest sheaf with \(\CzinftySpace[\ManifoldN]\subseteq \CinftySpace[\ManifoldN]\).}
    Thus it is more natural to state our results in the language of sheaves,
    as that is the data which is provided in our applications, and the data which is used in 
    our proofs.
    %
    In any case, one can easily restate our results with modules, with the caveat
    that one would then often restrict the modules to open sets, 
    and work locally with some glueing properties,
    thereby treating the modules like sheaves.
\end{remark}

    \subsection{Sheaves}
    Let \(\ManifoldN\) be a smooth manifold (possibly with boundary).

\begin{definition}\label{Defn::Sheaves::Sheaves::Sheaf}
    We say \(\SheafF\) is a \textit{pre-sheaf of vector fields on \(\ManifoldN\)} if
    \(\SheafF\) is a pre-sheaf of \(\CinftySpace\)-modules of smooth vector fields in \(\ManifoldN\)
    (with the usual restriction map). In particular, for each \(\Omega\subseteq \ManifoldN\) open,
    \(\SheafF[\Omega]\) is a \(\CinftySpace\)-sub-module of \(\VectorFields{\Omega}\)
    such that \(X\in \SheafF[\Omega]\implies X\big|_{\Omega'}\in \SheafF[\Omega']\) for \(\Omega'\subseteq \Omega\)
    (see \cite[\href{https://stacks.math.columbia.edu/tag/006E}{Definition 006E} and \href{https://stacks.math.columbia.edu/tag/006Q}{Definition 006Q}]{stacks-project}).
    We say \(\SheafF\) is a \textit{sheaf of vector fields on \(\ManifoldN\)} if \(\SheafF\)
    is a pre-sheaf of vector fields on \(\ManifoldN\) which satisfies the sheaf gluing axiom
    (see \cite[\href{https://stacks.math.columbia.edu/tag/006T}{Definition 006T} and \href{https://stacks.math.columbia.edu/tag/0077}{Definition 0077}]{stacks-project}).
\end{definition}

\begin{remark}
    In this paper, sheaves are only used as a convenient packaging
    of the properties we need; we do not use any advanced results concerning
    general sheaves. Some of our lemmas could be concluded
    from abstract results in the literature.
    For example, because we have partitions of unity, the sheaves we consider are fine \cite[Chapter 2, Definition 3.3]{WellsDifferentialAnalysisOnComplexManifolds},
    and therefore are soft \cite[Chapter 2, Definition 3.1 and Proposition 3.5]{WellsDifferentialAnalysisOnComplexManifolds};
    however, such results are standard to an analyst in our setting, and
    so we opt to instead directly show what we require.
\end{remark}

Let \(\sI\) be an index set, and for each \(\iota\in \sI\), let \(X_\iota\in \VectorFields{U_\iota}\),
where \(U_\iota\subseteq \ManifoldN\) is open. By \cite[\href{https://stacks.math.columbia.edu/tag/01AP}{Lemma 01AP}]{stacks-project},
there exists a unique smallest sheaf of vector fields \(\SheafF\) with \(X_\iota\in \SheafF[U_\iota]\), \(\forall \iota\in \sI\).

\begin{definition}
    \(\SheafF\) as described in the previous paragraph is called the \textit{sheaf of vector fields generated by} \(\{X_\iota :\iota\in \sI\}\)
    (see \cite[\href{https://stacks.math.columbia.edu/tag/01AQ}{Definition 01AQ}]{stacks-project}).
\end{definition}

\begin{definition}\label{Defn::Sheaves::Sheaves::FiniteType}
    We say a sheaf of vector fields,
    \(\SheafF\),
    is of \textit{finite type} if \(\forall x\in \ManifoldN\), 
    \(\exists U\ni x\) open, and \(F\subset \VectorFields{U}\) finite such that
    \(\SheafF\big|_U\) is generated by \(F\) (see \cite[\href{https://stacks.math.columbia.edu/tag/01B5}{Definition 01B5}]{stacks-project}).\footnote{For an open set \(U\), \(\SheafF\big|_U\) is the sheaf of vector fields on \(U\) given by restricting \(\SheafF\) to open subsets of \(U\).}
\end{definition}

Sheaves of vector fields of finite type are locally easy to understand as the next lemma shows.

\begin{lemma}\label{Lemma::Sheaves::Sheaves::ModuleGenByFiniteSetIsSheaf}
    Let \(\sS\subset \VectorFieldsN\) be a finite set. For each \(\Omega\subseteq \ManifoldN\)
    open, set
    \begin{equation*}
        \SheafGenBy{\sS}[\Omega]:=\CinftySpace[\Omega]\text{-module generated by }\sS\subseteq \VectorFields{\Omega}.
    \end{equation*}
    Then, \(\SheafGenBy{\sS}\) is a sheaf of vector fields on \(\ManifoldN\) and is the sheaf of vector fields
    generated by \(\sS\).
\end{lemma}
\begin{proof}
    \(\SheafGenBy{\sS}\) is clearly contained in the sheaf generated by \(\sS\), and \(\sS\subseteq \SheafGenBy{\sS}[\ManifoldN]\),
    and so to prove the result it suffices to show \(\SheafGenBy{\sS}\) is a sheaf of vector fields.
    It is clearly a pre-sheaf of vector fields (see Definition \ref{Defn::Sheaves::Sheaves::Sheaf}),
    and we therefore only need to verify the gluing property (see \cite[\href{https://stacks.math.columbia.edu/tag/006T}{Definition 006T}]{stacks-project}).

    Let \(U_\alpha\subseteq \ManifoldN\) be a collection of open sets (for \(\alpha\) ranging over some index set),
    and let \(V_\alpha\in \SheafGenBy{\sS}[U_\alpha]\) satisfying
    \begin{equation*}
        V_\alpha\big|_{U_\alpha\cap U_\beta}=V_\beta\big|_{U_\alpha\cap U_\beta},\quad \forall \alpha,\beta.
    \end{equation*}
    Set \(U=\bigcup_{\alpha} U_\alpha\) and define \(V\in \VectorFields{U}\) by \(V\big|_{U_\alpha}=V_\alpha\).
    We wish to show \(V\in \SheafGenBy{S}[U]\).

    Enumerate \(S=\left\{ W_1,\ldots, W_r \right\}\). By hypothesis, we may write
    \begin{equation*}
        V_\alpha=\sum_{l=1}^r a_{\alpha,k} W_k, \quad a_{\alpha,k}\in \CinftySpace[U_\alpha].
    \end{equation*}
    Pick a partition of unity \(\phi_{\kappa}\) on \(U\) such that each \(\phi_{\kappa}\) lies 
    \(\CzinftySpace[U_{\alpha(\kappa)}]\), for some \(\alpha(\kappa)\), and only finitely many \(\phi_\kappa\)
    are non-zero on any compact subset of \(U\).  We have,
    \begin{equation*}
        \begin{split}
            &V=\sum_\kappa \phi_\kappa V = \sum_{\kappa} \phi_\kappa V_{\alpha(\kappa)} 
            =\sum_{\kappa}\sum_{j=1}^r \phi_\kappa a_{\alpha(\kappa),k} W_k
            = \sum_{k=1}^r \left( \sum_\kappa \phi_\kappa a_{\alpha(\kappa),k} \right) W_k \in \SheafGenBy{\sS}[U].
        \end{split}
    \end{equation*}
\end{proof}

\begin{lemma}\label{Lemma::Sheaves::Sheaves::CheckGeneratorsAsGlobalSections}
    Let \(\SheafF\) be a sheaf of vector fields on \(\ManifoldN\) and suppose \(S\subseteq \VectorFieldsN\) is such that
    \begin{equation*}
        \SheafF[\ManifoldN]=\CinftySpace[\ManifoldN]\text{-module generated by }\sS.
    \end{equation*}
    Then, \(\SheafF\) is the sheaf of vector fields generated be \(\sS\). In particular,
    \(\SheafF\) is the sheaf of vector fields generated by \(\SheafF[\ManifoldN]\).
    Moreover, if \(\sS\) is finite, then \(\SheafF=\SheafGenBy{\sS}\).
\end{lemma}
\begin{proof}
    In light of \cite[\href{https://stacks.math.columbia.edu/tag/01AN}{Lemma 01AN}]{stacks-project},
    to show that \(\SheafF\) is the sheaf generated by \(\sS\), it suffices to show:
    \begin{equation}\label{Eqn::Sheaves::Sheaves::ToShowLocalInTmp1}
        \forall x\in \ManifoldN,\text{ if }X\in \SheafF[U], x\in U,\text{ then } \exists V\ni x,
        \text{ with }X\in \CinftySpace[V]\text{-module generated by }\sS.
    \end{equation}
    But if \(X\in \SheafF[U]\), take \(\phi\in \CzinftySpace[U]\) such that \(\phi\equiv 1\) on a neighborhood
    \(V\) of \(x\). Then, \(\phi X\in \SheafF[\ManifoldN]\), and therefore
    \begin{equation*}
        \phi X= \sum_{j=1}^L a_j X_j,\quad X_1,\ldots, X_L\in S,\quad a_j\in \CinftySpace[\ManifoldN].
    \end{equation*}
    Restricting to \(V\) proves \eqref{Eqn::Sheaves::Sheaves::ToShowLocalInTmp1}.
    
    Finally, if \(\sS\) is finite, that \(\SheafF=\SheafGenBy{\sS}\) follows from Lemma \ref{Lemma::Sheaves::Sheaves::ModuleGenByFiniteSetIsSheaf}.
\end{proof}

    \subsection{Filtrations of sheaves}
    Let \(\ManifoldN\) be a smooth manifold (possibly with boundary).

\begin{definition}
    We say \(\FilteredSheafF\) is a \textit{filtration of sheaves of vector fields on \(\ManifoldN\)}
    if for each \(j\in \Zg\), \(\FilteredSheafNoSetF[j]\) is a sheaf of vector fields on \(\ManifoldN\),
    with \(\FilteredSheafNoSetF[j]\subseteq \FilteredSheafNoSetF[j+1]\).
\end{definition}

\begin{definition}
    Let \(\sI\) be an index set, and for \(\iota\in \sI\) let \(X_\iota\in \VectorFields{U_\iota}\), where \(U_\iota\subseteq \ManifoldN\)
    is open, and let \(d_\iota\in \Zg\). We say \(\FilteredSheafF\) is the \textit{filtration of sheaves of vector fields generated by}
    \(\left\{ (X_\iota, d_\iota) : \iota\in \sI \right\}\), if each \(\FilteredSheafNoSetF[j]\)
    is the sheaf generated by \(\left\{ X_\iota : d_\iota\leq j \right\}\).
\end{definition}

\begin{definition}
    We say a filtration of sheaves of vector fields, \(\FilteredSheafF\), is of \textit{finite type}
    if \(\forall x\in \ManifoldN\), \(\exists U\ni x\) open, \(\exists F\subset \VectorFields{U}\times \Zg\) finite,
    such that \(\FilteredSheafF\big|_{U}\) is the filtration of sheaves of vector fields generated by \(F\).
\end{definition}

Filtrations of sheaves of finite type are locally easy to understand as the next lemma shows.

\begin{lemma}\label{Lemma::Sheaves::Filtered::SheafGenByFiniteSet}
    Let \(\sS\subset \VectorFieldsN\times \Zg\) be a finite set. For each \(\Omega\subseteq \ManifoldN\)
    open set
    \begin{equation}\label{Eqn::Sheaves::Filtered::SheafGenByFiniteSet::DefnOfSheafGenByFiniteSet}
        \FilteredSheafGenBy{\sS}[\Omega][j]:=\CinftySpace[\Omega]\text{-module generated by }\left\{ X : \exists (X,k)\in \sS, k\leq j \right\}.
    \end{equation}
    Then, \(\FilteredSheafGenBy{\sS}\) is the filtration of sheaves of vector fields generated by \(\sS\).
\end{lemma}
\begin{proof}
    This follows immediately from Lemma \ref{Lemma::Sheaves::Sheaves::ModuleGenByFiniteSetIsSheaf}.
\end{proof}

\begin{lemma}\label{Lemma::Sheaves::Filtered::CheckGeneratorsAsGlobalSections}
    Let \(\FilteredSheafF\) be a filtration of sheaves of vector fields on \(\ManifoldN\).
    Suppose \(\sS\subseteq \VectorFieldsN\times \Zg\) is such that
    \begin{equation*}
        \FilteredSheafF[\ManifoldN][j]=\CinftySpace[\ManifoldN]\text{-module generated by }\left\{ X : \exists (X,k)\in \sS, k\leq j \right\}.
    \end{equation*}
    Then, \(\FilteredSheafF\) is the filtration of sheaves of vector fields generated by \(\sS\).
    Moreover, if \(\sS\) is finite, then \(\FilteredSheafF=\FilteredSheafGenBy{\sS}\).
\end{lemma}
\begin{proof}
    This follows from Lemmas \ref{Lemma::Sheaves::Sheaves::CheckGeneratorsAsGlobalSections} 
    and \ref{Lemma::Sheaves::Filtered::SheafGenByFiniteSet}.
\end{proof}

\begin{lemma}\label{Lemma::Sheaves::Filtered::FiniteTypeOnCptSets}
    Let \(\FilteredSheafF\) be a filtration of sheaves of vector fields on \(\ManifoldN\).
    The following are equivalent:
    \begin{enumerate}[(i)]
        \item\label{Item::Sheaves::Filtered::FiniteTypeOnCptSets::FiniteType} \(\FilteredSheafF\) is of finite type.
        \item\label{Item::Sheaves::Filtered::FiniteTypeOnCptSets::FinGenOnCpt} \(\forall \Omega\Subset\ManifoldN \) open and relatively compact, \(\FilteredSheafF[\Omega]\) is finitely
            generated as a filtered \(\CinftySpace[\Omega]\)-module. More precisely, \(\exists F\subset \VectorFields{\Omega}\times \Zg\)
            finite such that
            \begin{equation*}
                \FilteredSheafF[\Omega][j]=\CinftySpace[\Omega]\text{-module generated by }\left\{ Y: \exists (Y,k)\in F, k\leq j \right\}.
            \end{equation*}
    \end{enumerate}
    Moreover, when \ref{Item::Sheaves::Filtered::FiniteTypeOnCptSets::FinGenOnCpt} holds, we have
    \begin{equation}\label{Eqn::Sheaves::Filtered::FiniteTypeOnCptSets::EquaLsFiniteGen}
        \FilteredSheafF\big|_{\Omega}=\FilteredSheafGenBy{F}.
    \end{equation}
\end{lemma}
\begin{proof}
    Suppose \ref{Item::Sheaves::Filtered::FiniteTypeOnCptSets::FinGenOnCpt} holds. 
    Lemma \ref{Lemma::Sheaves::Filtered::CheckGeneratorsAsGlobalSections} shows \eqref{Eqn::Sheaves::Filtered::FiniteTypeOnCptSets::EquaLsFiniteGen} holds.
    Since \(F\) is finite, \ref{Item::Sheaves::Filtered::FiniteTypeOnCptSets::FiniteType} follows.

    Suppose \ref{Item::Sheaves::Filtered::FiniteTypeOnCptSets::FiniteType} holds, so that \(\FilteredSheafF\)
    is of finite type and let \(\Omega\Subset \ManifoldN\) be open and relatively compact.
    For each \(x\in \overline{\Omega}\), let \(U_x\ni x\) be open such that
    \(\exists F_x\subset \VectorFields{U_x}\times \Zg\) finite with \(\FilteredSheafF\big|_{U_x}\)
    generated by \(F_x\). Let \(V_x\Subset U_x\) be a neighborhood with \(x\in V_x\) and let \(\phi_x\in \CzinftySpace[U_x]\)
    equal \(1\) on a neighborhood of \(\overline{V_x}\).  Set \(G_x:=\left\{ \left( \phi_x X, d \right) : (X,d)\in F_x \right\}\).
    It follows that \(\FilteredSheafF\big|_{V_x}\)
    is generated by \(G_x\).

    \(\left\{ V_x : x\in \overline{\Omega} \right\}\) is an open cover for \(\overline{\Omega}\); extract
    a finite sub-cover \(V_{x_1},\ldots, V_{x_L}\).
    Set \(H:=G_{x_1}\cup G_{x_2}\cup\cdots \cup G_{x_L}\).
    Since \(G_{x_j}\) generates \(\FilteredSheafF\big|_{V_{x_j}}\), we have \(H\)
    generates \(\FilteredSheafF\big|_{V_{x_j}}\), \(\forall j\). It follows from \cite[\href{https://stacks.math.columbia.edu/tag/01AN}{Lemma 01AN}]{stacks-project}
    that \(H\) generates \(\FilteredSheafF\big|_{\Omega}\).

    Since \(H\) is finite, Lemma \ref{Lemma::Sheaves::Filtered::SheafGenByFiniteSet}
    shows \(\FilteredSheafF\big|_{\Omega}= \FilteredSheafGenBy{H}\big|_{\Omega}\),
    and therefore \(\FilteredSheafF[\Omega]=\FilteredSheafGenBy{H}[\Omega]\),
    establishing \ref{Item::Sheaves::Filtered::FiniteTypeOnCptSets::FinGenOnCpt}.
\end{proof}

\begin{definition}\label{Defn::Sheaves::Filtered::SpansTangentSpace}
    We say \(\FilteredSheafF\) \textit{spans the tangent space at every point}
    if the sheafification of the pre-sheaf of vector fields on \(\ManifoldN\) given by
    \begin{equation*}
        \Omega\mapsto \bigcup_{j} \FilteredSheafF[\Omega][j]
    \end{equation*}
    equals \(\SheafVectorFields\).
\end{definition}

    \subsection{Lie algbera filtrations}
    Let \(\ManifoldN\) be a smooth manifold (possibly with boundary),
and \(\FilteredSheafF\) a filtration of sheaves of vector fields on \(\ManifoldN\).

\begin{definition}
    We say \(\FilteredSheafF\) is a \textit{Lie algebra filtration of sheaves of vector fields on \(\ManifoldN\)}
    if \(\forall \Omega\subseteq \ManifoldN\) open,
    \begin{equation*}
        X\in \FilteredSheafF[\Omega][j], Y\in \FilteredSheafF[\Omega][k]\implies [X,Y]\in \FilteredSheafF[\Omega][j+k].
    \end{equation*}
    We define a \textit{Lie algebra filtration of pre-sheaves of vector fields on \(\ManifoldN\)} by the same
    in the same way, when \(\FilteredSheafF\) is a filtration of presheaves of vector fields on \(\ManifoldN\).
\end{definition}

\begin{notation}
    We write \(\LieFilteredSheafF\) for the smallest Lie algebra filtration of sheaves of vector fields on \(\ManifoldN\)
    containing \(\FilteredSheafF\).
\end{notation}

\begin{definition}\label{Defn::Sheaves::LieAlg::HormandersCondition}
    We say \(\FilteredSheafF\) satisfies \textit{H\"ormander's condition} if \(\LieFilteredSheafF\)
    spans the tangent space at every point (see Definition \ref{Defn::Sheaves::Filtered::SpansTangentSpace}).
    This terminology is justified by Proposition \ref{Prop::Sheaves::LieAlg::HormandersCondiIsEquivToOtherNotions}.
\end{definition}

\begin{proposition}\label{Prop::Sheaves::LieAlg::HormandersCondiIsEquivToOtherNotions}
    The following are equivalent:
    \begin{enumerate}[(i)]
        \item\label{Item::::Sheaves::LieAlg::HormandersCondiIsEquivToOtherNotions::FiniteTypeAndHorCond} \(\FilteredSheafF\) is of finite type and satisfies H\"ormander's condition.
        \item\label{Item::::Sheaves::LieAlg::HormandersCondiIsEquivToOtherNotions::HorCondOnCptSets} \(\forall \Omega\Subset \ManifoldN\) open and relatively compact, there exists
            \(\WWd=\left\{ \left( W_1,\Wd_1 \right),\ldots, \left( W_r,\Wd_r \right) \right\}\subset \VectorFields{\Omega}\times \Zg\),
            H\"ormander vector fields with formal degrees on \(\Omega\), such that
            \begin{equation*}
                \FilteredSheafF\big|_{\Omega}=\FilteredSheafGenBy{\WWd}.
            \end{equation*}
        \item\label{Item::::Sheaves::LieAlg::HormandersCondiIsEquivToOtherNotions::LocallyHorCond} \(\forall x\in \ManifoldN\), \(\exists \Omega\ni x\) open, and H\"ormander vector fields with formal degrees on
            \(\Omega\), \(\WWd=\left\{ \left( W_1,\Wd_1 \right),\ldots, \left( W_r,\Wd_r \right) \right\}\subset \VectorFields{\Omega}\times \Zg\),
            such that
            \begin{equation*}
                \FilteredSheafF\big|_{\Omega}=\FilteredSheafGenBy{\WWd}.
            \end{equation*}
    \end{enumerate}
\end{proposition}

To prove Proposition \ref{Prop::Sheaves::LieAlg::HormandersCondiIsEquivToOtherNotions} we introduce two lemmas
which are useful in their own right.

\begin{lemma}\label{Lemma::Sheaves::LieAlg::LieAlgebraFilIsGenByXXd}
    Suppose \(\Omega\subseteq \ManifoldN\) is open, 
    \(\WWd=\left\{ \left( W_1,\Wd_1 \right),\ldots, (W_r,\Wd_r) \right\}\subset \VectorFields{\Omega}\times \Zg\)
    are such that \(W_1,\ldots, W_r\) satisfy H\"ormander's condition of order \(m\) on \(\Omega\),
    and such that \(\FilteredSheafF[\Omega]=\FilteredSheafGenBy{\WWd}[\Omega]\).
    Set
    \begin{equation}\label{Eqn::Sheaves::LieAlg::DefnXXd}
        \XXd=\left\{ (X_1,\Xd_1),\ldots (X_q,\Xd_q) \right\}:=
        \left\{ (Y,e)\in \GenWWd : e\leq m \max\{\Wd_j\} \right\}.
    \end{equation}
    Then, \(\LieFilteredSheafF\big|_{\Omega}=\FilteredSheafGenBy{\XXd}\).
\end{lemma}
\begin{proof}
    By Lemma \ref{Lemma::Sheaves::Filtered::CheckGeneratorsAsGlobalSections},
    \(\FilteredSheafF\big|_{\Omega}=\FilteredSheafGenBy{\WWd}\) and therefore
    \(\FilteredSheafF\big|_{\Omega}\subseteq\FilteredSheafGenBy{\XXd}\subseteq \LieFilteredSheafF\big|_{\Omega}\).
    Lemma \ref{Lemma::Sheaves::Filtered::SheafGenByFiniteSet}
    shows \(\FilteredSheafGenBy{\XXd}\) is a filtration of sheaves of vector fields on \(\Omega\).
    All that remains to show is that \(\FilteredSheafGenBy{\XXd}\) is a Lie algebra filtration,
    and the result will follow.

    Just as in \eqref{Eqn::BoundaryVfs::NSW::Z}, we have
    \begin{equation}\label{Eqn::Sheaves::LieAlg::NSW}
        [X_j, X_k]=\sum_{\Xd_l\leq \Xd_j+\Xd_k} c_{j,k}^l X_l, \quad c_{j,k}^l \in \CinftySpace[\Omega].
    \end{equation}
    From \eqref{Eqn::Sheaves::LieAlg::NSW} is easily follows for \(\Omega'\subseteq \Omega\) open:
    \begin{equation*}
        Y\in \FilteredSheafGenBy{\XXd}[\Omega'][j],Z\in \FilteredSheafGenBy{\XXd}[\Omega'][k]
        \implies [Y,Z]\in \FilteredSheafGenBy{\XXd}[\Omega'][j+k],
    \end{equation*}
    completing the proof.
\end{proof}

\begin{lemma}\label{Lemma::Sheaves::LieAlg::TwoNotionsOfHormanderAreTheSame}
    Suppose \(\WWd=\left\{ \left( W_1,\Wd_1 \right),\ldots, \left( W_r, \Wd_r \right) \right\}\subset \VectorFieldsN\times \Zg\).
    Then \(\FilteredSheafGenBy{\WWd}\) satisfies H\"ormander's condition if and only if \(W_1,\ldots, W_r\)
    satisfy H\"ormander's condition on \(\ManifoldN\).
\end{lemma}
\begin{proof}
    Suppose \(W_1,\ldots, W_r\) satisfy H\"ormander's condition on \(\ManifoldN\).
    We will show \(\forall x\in \ManifoldN\), \(\exists \Omega\ni x\) open with
    \(\VectorFields{\Omega}\subseteq \bigcup_{j} \LieFilteredSheaf{\FilteredSheafGenBy{\WWd}}[\Omega][j]\), and then it will follow that the sheaf generated
    by \(\Omega\mapsto \bigcup_{j} \LieFilteredSheaf{\FilteredSheafGenBy{\WWd}}[\Omega][j]\) equals \(\SheafVectorFields\), as desired.

    Indeed, fix \(x\in \ManifoldN\), and take \(\Omega\ni x\) open and relatively compact.
    By assumption, \(W_1,\ldots, W_r\) satisfy H\"ormander's condition of some order
    \(m\in \Zg\) on \(\overline{\Omega}\).
    Let \(\XXd\) be as in \eqref{Eqn::Sheaves::LieAlg::DefnXXd}.
    Then, \(X_1(y),\ldots, X_q(y)\) span \(\TangentSpace{\ManifoldN}{y}\), \(\forall y\in \Omega\)
    and therefore
    \begin{equation*}
    \begin{split}
         &\VectorFields{\Omega}\subseteq \CinftySpace[\Omega]\text{-module generated by }X_1,\ldots, X_q
         =\FilteredSheafGenBy{\XXd}[\Omega][m\max{\{\Wd_j\}}]
         =\LieFilteredSheafF[\Omega][m\max\{\Wd_j\}],
    \end{split}
    \end{equation*}
    where the last equality follows from Lemma \ref{Lemma::Sheaves::LieAlg::LieAlgebraFilIsGenByXXd},
    which completes the proof that \(\FilteredSheafGenBy{\WWd}\) satisfies H\"ormander's condition.

    Conversely, suppose  \(\FilteredSheafGenBy{\WWd}\) satisfies H\"ormander's condition.
    The infinite set of vector fields
    \begin{equation}\label{Eqn::Sheaves::LieAlg::TwoNotionsOfHormanderAreTheSame::Tmp1}
        W_1,\ldots, W_r, \ldots, [W_i,W_j],\ldots,\ldots, [W_i,[W_j,W_k]],\ldots,\ldots,\ldots
    \end{equation}
    generates the sheafification of
    \(\Omega\mapsto \bigcup_{j} \LieFilteredSheaf{\FilteredSheafGenBy{\WWd}}[\Omega][j]\),
    which (by hypothesis) equals \(\SheafVectorFields\).

    We wish to show the collection \eqref{Eqn::Sheaves::LieAlg::TwoNotionsOfHormanderAreTheSame::Tmp1}
    spans the tangent space to \(\ManifoldN\) at each point.

    Fix \(x\in \ManifoldN\). By \cite[\href{https://stacks.math.columbia.edu/tag/01AN}{Lemma 01AN}]{stacks-project},
    the stalk \(\SheafVectorFields_x\) equals the \(\CinftySpace_x\)-module generated by the germs of
    \eqref{Eqn::Sheaves::LieAlg::TwoNotionsOfHormanderAreTheSame::Tmp1} at \(x\). It follows that
    the collection \eqref{Eqn::Sheaves::LieAlg::TwoNotionsOfHormanderAreTheSame::Tmp1}
    spans \(\TangentSpace{\ManifoldN}{x}\), completing the proof.
\end{proof}

\begin{proof}[Proof of Proposition \ref{Prop::Sheaves::LieAlg::HormandersCondiIsEquivToOtherNotions}]
    \ref{Item::::Sheaves::LieAlg::HormandersCondiIsEquivToOtherNotions::HorCondOnCptSets}\(\Rightarrow\)\ref{Item::::Sheaves::LieAlg::HormandersCondiIsEquivToOtherNotions::LocallyHorCond}
    is clear.

    \ref{Item::::Sheaves::LieAlg::HormandersCondiIsEquivToOtherNotions::LocallyHorCond}\(\Rightarrow\)\ref{Item::::Sheaves::LieAlg::HormandersCondiIsEquivToOtherNotions::FiniteTypeAndHorCond}:
    That \(\FilteredSheafF\) is of finite type follows immediately. Combining \ref{Item::::Sheaves::LieAlg::HormandersCondiIsEquivToOtherNotions::LocallyHorCond}
    with Lemma \ref{Lemma::Sheaves::LieAlg::TwoNotionsOfHormanderAreTheSame}
    shows \(\forall x\in \ManifoldN\), \(\exists \Omega\ni x\) open such that
    \(\FilteredSheafF\big|_{\Omega}\) satisfies H\"ormander's condition. However, H\"ormander's condition
    is a local property (this follows from the definition and \cite[\href{https://stacks.math.columbia.edu/tag/01AN}{Lemma 01AN}]{stacks-project}).
    We conclude \(\FilteredSheafF\) satisfies H\"ormander's condition.

    \ref{Item::::Sheaves::LieAlg::HormandersCondiIsEquivToOtherNotions::FiniteTypeAndHorCond}\(\Rightarrow\)\ref{Item::::Sheaves::LieAlg::HormandersCondiIsEquivToOtherNotions::HorCondOnCptSets}:
    Let \(\Omega\Subset\ManifoldN\) be open an relatively compact.
    By Lemma \ref{Lemma::Sheaves::Filtered::FiniteTypeOnCptSets}, there exists a finite set
    \(\WWd\subset \VectorFields{\Omega}\times \Zg\) such that
    \begin{equation*}
        \FilteredSheafF\big|_{\Omega}=\FilteredSheafGenBy{\WWd}.
    \end{equation*}
    Lemma \ref{Lemma::Sheaves::LieAlg::TwoNotionsOfHormanderAreTheSame} shows \(\WWd\)
    are H\"ormander vector fields with formal degrees.
\end{proof}

\begin{proposition}\label{Prop::Sheaves::LieAlg::FFiniteTypeAndHorImpliesSameForLie}
    Suppose \(\FilteredSheafF\) is of finite type and satisfies H\"ormander's condition.
    Then, \(\LieFilteredSheafF\) is of finite type and spans the tangent space at every point
    (and therefore satisfies H\"ormander's condition).
\end{proposition}
\begin{proof}
    That \(\LieFilteredSheafF\) spans the tangent space at every point
    (and therefore satisfies H\"ormander's condition) is just a restatement of the definitions
    (see Definition \ref{Defn::Sheaves::LieAlg::HormandersCondition}).

    We turn to showing \(\LieFilteredSheafF\) is of finite type.
    Fix \(x\in \ManifoldN\). 
    By Proposition \ref{Prop::Sheaves::LieAlg::HormandersCondiIsEquivToOtherNotions},
    \(\exists \Omega\ni x\) open and \(\WWd=\left\{ \left( W_1,\Wd_1 \right),\ldots, \left( W_r,\Wd_r \right) \right\}\subset \VectorFields{\Omega}\times\Zg\)
    such that \(\WWd\) are H\"ormander vector fields with formal degrees on \(\Omega\)
    and
    \(\FilteredSheafF\big|_{\Omega}=\FilteredSheafGenBy{\WWd}\).


    Let \(\Omega_1\Subset \Omega\) be an open, relatively compact neighborhood of \(x\).
    By compactness, \(W_1,\ldots, W_r\) satisfy H\"ormander's condition of order \(m\) on \(\Omega_1\)
    for some \(m\in \Zg\).
    Lemma \ref{Lemma::Sheaves::LieAlg::LieAlgebraFilIsGenByXXd} 
    shows \(\LieFilteredSheafF\big|_{\Omega_1}\) is finitely generated, completing the proof.
\end{proof}

    \subsection{Control}\label{Section::Sheaves::Control}
    The notions of local strong/weak control/equivalence from Definition \ref{Defn::BasicDefns::StrongWeakEquivalnce}
can be more simply stated using filtrations of sheaves of vector fields as the next two results show;
\(\ManifoldN\) is a smooth manifold (possibly with boundary).

\begin{proposition}\label{Prop::Sheaves::Control::GlobalEquivWithPreviousDefns}
    Let \(\WWd=\left\{ \left( W_1,\Wd_1 \right),\ldots, \left( W_r,\Wd_r \right) \right\}\subset \VectorFieldsN\times \Zg\)
    and \(\ZZd=\left\{ \left( Z_1,\Zd_1 \right),\ldots, \left( Z_s,\Zd_s \right) \right\}\subset \VectorFieldsN\times \Zg\)
    be two list of H\"ormander vector fields with formal degrees on \(\ManifoldN\).
    Then,
    \begin{enumerate}[(i)]
        \item\label{Item::Sheaves::Control::GlobalEquivWithPreviousDefns::StrongControl} \(\WWd\) locally strongly controls \(\ZZd\) on \(\ManifoldN\) if and only if
            \(\FilteredSheafGenBy{\ZZd}\subseteq \FilteredSheafGenBy{\WWd}\).
        \item\label{Item::Sheaves::Control::GlobalEquivWithPreviousDefns::StrongEquiv} \(\WWd\) and \(\ZZd\) are locally strongly equivalent on \(\ManifoldN\)
             if and only if \(\FilteredSheafGenBy{\WWd}= \FilteredSheafGenBy{\ZZd}\).
        \item\label{Item::Sheaves::Control::GlobalEquivWithPreviousDefns::WeakControl} \(\WWd\) locally weakly controls \(\ZZd\) on \(\ManifoldN\) if and only if
            \(\FilteredSheafGenBy{\ZZd}\subseteq \LieFilteredSheaf{\FilteredSheafGenBy{\WWd}}\).
        \item\label{Item::Sheaves::Control::GlobalEquivWithPreviousDefns::WeakEquiv} \(\WWd\) and \(\ZZd\) are locally weakly equivalent if and only if
            \(\LieFilteredSheaf{\FilteredSheafGenBy{\WWd}}=\LieFilteredSheaf{\FilteredSheafGenBy{\ZZd}}\).
    \end{enumerate}
\end{proposition}
\begin{proof}
    \ref{Item::Sheaves::Control::GlobalEquivWithPreviousDefns::StrongControl} follows from the definitions and
    \ref{Item::Sheaves::Control::GlobalEquivWithPreviousDefns::StrongEquiv} follows from 
    \ref{Item::Sheaves::Control::GlobalEquivWithPreviousDefns::StrongControl}.

    \ref{Item::Sheaves::Control::GlobalEquivWithPreviousDefns::WeakControl}:
    Fix \(\Omega\Subset \ManifoldN\) open an relatively compact. By the compactness of \(\overline{\Omega}\),
    \(W_1,\ldots, W_r\) satisfy H\"ormander's condition of some order \(m\in \Zg\) on \(\Omega\).
    With this choice of \(m\), define \(\XXd=\left\{ \left( X_1,\Xd_1 \right),\ldots, \left( X_q,\Xd_q \right) \right\}\subset \VectorFields{\Omega}\times \Zg\)
    as in \eqref{Eqn::Sheaves::LieAlg::DefnXXd}, 
    so that
    \begin{equation*}
        \LieFilteredSheaf{\FilteredSheafGenBy{\WWd}}\big|_{\Omega}=\FilteredSheafGenBy{\XXd}\big|_{\Omega}.
    \end{equation*}
    Because \(\XXd\) satisfies \eqref{Eqn::Sheaves::LieAlg::NSW}, \(\XXd\) locally strongly
    controls \(\GenXXd=\GenWWd\) on \(\Omega\).
    Thus, \(\XXd\) locally strongly controls \(\ZZd\) on \(\Omega\) if and only if \(\WWd\)
    locally weakly controls \(\ZZd\) on \(\Omega\).

    Using \ref{Item::Sheaves::Control::GlobalEquivWithPreviousDefns::StrongControl}, we conclude
    \(\WWd\) locally weakly controls \(\ZZd\) on \(\Omega\) if and only if
    \begin{equation*}
        \FilteredSheafGenBy{\ZZd}\big|_{\Omega}\subseteq \FilteredSheafGenBy{\XXd}\big|_{\Omega}
        =\LieFilteredSheaf{\FilteredSheafGenBy{\WWd}}\big|_{\Omega}.
    \end{equation*}
    Since \(\Omega\Subset \ManifoldN\) was arbitrary, the result follows.

    \ref{Item::Sheaves::Control::GlobalEquivWithPreviousDefns::WeakEquiv}: By \ref{Item::Sheaves::Control::GlobalEquivWithPreviousDefns::WeakControl},
    \(\WWd\) and \(\ZZd\) are locally weakly equivalent on \(\ManifoldN\) if and only if
    \(\FilteredSheafGenBy{\ZZd}\subseteq \LieFilteredSheaf{\FilteredSheafGenBy{\WWd}}\)
    and 
    \(\FilteredSheafGenBy{\WWd}\subseteq \LieFilteredSheaf{\FilteredSheafGenBy{\ZZd}}\).
    Since \(\LieFilteredSheaf{\FilteredSheafGenBy{\WWd}}\) and \(\LieFilteredSheaf{\FilteredSheafGenBy{\ZZd}}\)
    are Lie algebra filtrations of sheaves of vector fields, we see
    \(\FilteredSheafGenBy{\ZZd}\subseteq \LieFilteredSheaf{\FilteredSheafGenBy{\WWd}}\Leftrightarrow \LieFilteredSheaf{\FilteredSheafGenBy{\ZZd}}\subseteq \LieFilteredSheaf{\FilteredSheafGenBy{\WWd}}\)
    and 
    \(\FilteredSheafGenBy{\WWd}\subseteq \LieFilteredSheaf{\FilteredSheafGenBy{\ZZd}}\Leftrightarrow \LieFilteredSheaf{\FilteredSheafGenBy{\WWd}}\subseteq \LieFilteredSheaf{\FilteredSheafGenBy{\ZZd}}\).
    We conclude \(\WWd\) and \(\ZZd\) are locally weakly equivalent on \(\ManifoldN\) if and only if
     \(\LieFilteredSheaf{\FilteredSheafGenBy{\WWd}}=\LieFilteredSheaf{\FilteredSheafGenBy{\ZZd}}\).
\end{proof}

\begin{proposition}\label{Proposition::Sheaves::Control::LocalEquivWithPreviousDefns}
    Let \(\FilteredSheafF\) and \(\FilteredSheafG\) be filtrations of sheaves of vector fields on \(\ManifoldN\)
    which are finite type and satisfy H\"ormander's condition.
    \begin{enumerate}[label=(\Alph*)]
        \item\label{Item::Sheaves::Control::LocalEquivWithPreviousDefns::StrongControl} The following are equivalent.
        \begin{enumerate}[label=(\Alph{enumi}.\arabic*)]
            \item\label{Item::Sheaves::Control::LocalEquivWithPreviousDefns::StrongControl::Sheaves} \(\FilteredSheafG\subseteq \FilteredSheafF\).
            \item\label{Item::Sheaves::Control::LocalEquivWithPreviousDefns::StrongControl::RelCpt}  \(\forall \Omega\subseteq \ManifoldN\) open, \(\forall F,G\subset \VectorFields{\Omega}\times \Zg\)
                finite sets such that \(\FilteredSheafF\big|_{\Omega}=\FilteredSheafGenBy{F}\) and \(\FilteredSheafG\big|_{\Omega}=\FilteredSheafGenBy{G}\),
                we have \(F\) locally strongly controls \(G\) on \(\Omega\).
            \item\label{Item::Sheaves::Control::LocalEquivWithPreviousDefns::StrongControl::Local} \(\forall x\in \ManifoldN\), 
                \(\exists \Omega\ni x\) open, and \(F,G\subset \VectorFields{\Omega}\times \Zg\) finite
                such that \(\FilteredSheafF\big|_{\Omega}=\FilteredSheafGenBy{F}\) and \(\FilteredSheafG\big|_{\Omega}=\FilteredSheafGenBy{G}\)
                and \(F\) locally strongly controls \(G\) on \(\Omega\).
        \end{enumerate}
        \item\label{Item::Sheaves::Control::LocalEquivWithPreviousDefns::StrongEquiv} The following are equivalent.
        \begin{enumerate}[label=(\Alph{enumi}.\arabic*)]
            \item\label{Item::Sheaves::Control::LocalEquivWithPreviousDefns::StrongEquiv::Sheaves} \(\FilteredSheafF= \FilteredSheafG\).
            \item\label{Item::Sheaves::Control::LocalEquivWithPreviousDefns::StrongEquiv::RelCpt}  \(\forall \Omega\subseteq \ManifoldN\) open, \(\forall F,G\subset \VectorFields{\Omega}\times \Zg\)
                finite sets such that \(\FilteredSheafF\big|_{\Omega}=\FilteredSheafGenBy{F}\) and \(\FilteredSheafG\big|_{\Omega}=\FilteredSheafGenBy{G}\),
                we have \(F\) and \(G\) are locally strongly equivalent on \(\Omega\).
            \item\label{Item::Sheaves::Control::LocalEquivWithPreviousDefns::StrongEquiv::Local} \(\forall x\in \ManifoldN\), 
                \(\exists \Omega\ni x\) open, and \(F,G\subset \VectorFields{\Omega}\times \Zg\) finite
                such that \(\FilteredSheafF\big|_{\Omega}=\FilteredSheafGenBy{F}\) and \(\FilteredSheafG\big|_{\Omega}=\FilteredSheafGenBy{G}\)
                and \(F\) and \(G\) are locally strongly equivalent on \(\Omega\).
        \end{enumerate}

        \item\label{Item::Sheaves::Control::LocalEquivWithPreviousDefns::WeakControl} The following are equivalent.
        \begin{enumerate}[label=(\Alph{enumi}.\arabic*)]
            \item\label{Item::Sheaves::Control::LocalEquivWithPreviousDefns::WeakControl::Sheaves} \(\FilteredSheafG\subseteq \LieFilteredSheaf{\FilteredSheafF}\).
            \item\label{Item::Sheaves::Control::LocalEquivWithPreviousDefns::WeakControl::RelCpt}  \(\forall \Omega\subseteq \ManifoldN\) open, \(\forall F,G\subset \VectorFields{\Omega}\times \Zg\)
                finite sets such that \(\FilteredSheafF\big|_{\Omega}=\FilteredSheafGenBy{F}\) and \(\FilteredSheafG\big|_{\Omega}=\FilteredSheafGenBy{G}\),
                we have \(F\) locally weakly controls \(G\) on \(\Omega\).
            \item\label{Item::Sheaves::Control::LocalEquivWithPreviousDefns::WeakControl::Local} \(\forall x\in \ManifoldN\), 
                \(\exists \Omega\ni x\) open, and \(F,G\subset \VectorFields{\Omega}\times \Zg\) finite
                such that \(\FilteredSheafF\big|_{\Omega}=\FilteredSheafGenBy{F}\) and \(\FilteredSheafG\big|_{\Omega}=\FilteredSheafGenBy{G}\)
                and \(F\) locally weakly controls \(G\) on \(\Omega\).
        \end{enumerate}
        \item\label{Item::Sheaves::Control::LocalEquivWithPreviousDefns::WeakEquiv} The following are equivalent.
        \begin{enumerate}[label=(\Alph{enumi}.\arabic*)]
            \item\label{Item::Sheaves::Control::LocalEquivWithPreviousDefns::WeakEquiv::Sheaves} \(\LieFilteredSheaf{\FilteredSheafF}= \LieFilteredSheaf{\FilteredSheafG}\).
            \item\label{Item::Sheaves::Control::LocalEquivWithPreviousDefns::WeakEquiv::RelCpt}  \(\forall \Omega\subseteq \ManifoldN\) open, \(\forall F,G\subset \VectorFields{\Omega}\times \Zg\)
                finite sets such that \(\FilteredSheafF\big|_{\Omega}=\FilteredSheafGenBy{F}\) and \(\FilteredSheafG\big|_{\Omega}=\FilteredSheafGenBy{G}\),
                we have \(F\) and \(G\) are locally weakly equivalent on \(\Omega\).
            \item\label{Item::Sheaves::Control::LocalEquivWithPreviousDefns::WeakEquiv::Local} \(\forall x\in \ManifoldN\), 
                \(\exists \Omega\ni x\) open, and \(F,G\subset \VectorFields{\Omega}\times \Zg\) finite
                such that \(\FilteredSheafF\big|_{\Omega}=\FilteredSheafGenBy{F}\) and \(\FilteredSheafG\big|_{\Omega}=\FilteredSheafGenBy{G}\)
                and \(F\) and \(G\) are locally weakly equivalent on \(\Omega\).
        \end{enumerate}
    \end{enumerate}
\end{proposition}
\begin{proof}
    \ref{Item::Sheaves::Control::LocalEquivWithPreviousDefns::StrongControl::Sheaves}\(\Rightarrow\)\ref{Item::Sheaves::Control::LocalEquivWithPreviousDefns::StrongControl::RelCpt}
    and
    \ref{Item::Sheaves::Control::LocalEquivWithPreviousDefns::WeakControl::Sheaves}\(\Rightarrow\)\ref{Item::Sheaves::Control::LocalEquivWithPreviousDefns::WeakControl::RelCpt}:
    if \(\Omega\subseteq \ManifoldN\) is open and \(F,G\subset \VectorFields{\Omega}\times \Zg\)
                are finite sets with \(\FilteredSheafF\big|_{\Omega}=\FilteredSheafGenBy{F}\) and \(\FilteredSheafG\big|_{\Omega}=\FilteredSheafGenBy{G}\),
                then Lemma \ref{Lemma::Sheaves::LieAlg::TwoNotionsOfHormanderAreTheSame} shows that
                \(F\) and \(G\) are both lists of H\"ormander vector fields with formal degrees on \(\Omega\).
    Thus, if \ref{Item::Sheaves::Control::LocalEquivWithPreviousDefns::StrongControl::Sheaves} (respectively, \ref{Item::Sheaves::Control::LocalEquivWithPreviousDefns::WeakControl::Sheaves})
    holds, then Proposition \ref{Prop::Sheaves::Control::GlobalEquivWithPreviousDefns} \ref{Item::Sheaves::Control::GlobalEquivWithPreviousDefns::StrongControl} (respectively, \ref{Item::Sheaves::Control::GlobalEquivWithPreviousDefns::WeakControl})
    shows \ref{Item::Sheaves::Control::LocalEquivWithPreviousDefns::StrongControl::RelCpt} (respectively, \ref{Item::Sheaves::Control::LocalEquivWithPreviousDefns::WeakControl::RelCpt}) holds.

    Since \(\FilteredSheafF\) and \(\FilteredSheafG\) are both finite type, \(\forall x\in \ManifoldN\), 
                \(\exists \Omega\ni x\) open, and \(F,G\subset \VectorFields{\Omega}\times \Zg\) finite
                such that \(\FilteredSheafF\big|_{\Omega}=\FilteredSheafGenBy{F}\) and \(\FilteredSheafG\big|_{\Omega}=\FilteredSheafGenBy{G}\) (see Lemma \ref{Lemma::Sheaves::Filtered::SheafGenByFiniteSet}).
    From this, 
    \ref{Item::Sheaves::Control::LocalEquivWithPreviousDefns::StrongControl::RelCpt}\(\Rightarrow\)\ref{Item::Sheaves::Control::LocalEquivWithPreviousDefns::StrongControl::Local},
    \ref{Item::Sheaves::Control::LocalEquivWithPreviousDefns::StrongEquiv::RelCpt}\(\Rightarrow\)\ref{Item::Sheaves::Control::LocalEquivWithPreviousDefns::StrongEquiv::Local},
    \ref{Item::Sheaves::Control::LocalEquivWithPreviousDefns::WeakControl::RelCpt}\(\Rightarrow\)\ref{Item::Sheaves::Control::LocalEquivWithPreviousDefns::WeakControl::Local},
    and
    \ref{Item::Sheaves::Control::LocalEquivWithPreviousDefns::WeakEquiv::RelCpt}\(\Rightarrow\)\ref{Item::Sheaves::Control::LocalEquivWithPreviousDefns::WeakEquiv::Local}    
    are immediate.

    \ref{Item::Sheaves::Control::LocalEquivWithPreviousDefns::StrongControl::Local}\(\Rightarrow\)\ref{Item::Sheaves::Control::LocalEquivWithPreviousDefns::StrongControl::Sheaves}
    and
    \ref{Item::Sheaves::Control::LocalEquivWithPreviousDefns::WeakControl::Local}\(\Rightarrow\)\ref{Item::Sheaves::Control::LocalEquivWithPreviousDefns::WeakControl::Sheaves}:
    By Proposition \ref{Prop::Sheaves::LieAlg::HormandersCondiIsEquivToOtherNotions}, if \(F\) and \(G\)
    are as in either \ref{Item::Sheaves::Control::LocalEquivWithPreviousDefns::StrongControl::Local}
    or \ref{Item::Sheaves::Control::LocalEquivWithPreviousDefns::WeakControl::Local}, then \(F\)
    and \(G\) are H\"ormander vector fields with formal degrees.
    Thus, by Proposition \ref{Prop::Sheaves::Control::GlobalEquivWithPreviousDefns} \ref{Item::Sheaves::Control::GlobalEquivWithPreviousDefns::StrongControl} and \ref{Item::Sheaves::Control::GlobalEquivWithPreviousDefns::WeakControl},
    we see
    \(\forall x\in \ManifoldN\), \(\exists \Omega\ni x\) open with
    \begin{equation*}
    \begin{split}
         & \FilteredSheafG\big|_{\Omega}\subseteq \FilteredSheafF\big|_{\Omega}, \quad \text{if \ref{Item::Sheaves::Control::LocalEquivWithPreviousDefns::StrongControl::Local} holds,}
         \\& \FilteredSheafG\big|_{\Omega}\subseteq \LieFilteredSheaf{\FilteredSheafF}\big|_{\Omega}, \quad \text{if \ref{Item::Sheaves::Control::LocalEquivWithPreviousDefns::WeakControl::Local} holds.}
    \end{split}
    \end{equation*}
    \ref{Item::Sheaves::Control::LocalEquivWithPreviousDefns::StrongControl::Local}\(\Rightarrow\)\ref{Item::Sheaves::Control::LocalEquivWithPreviousDefns::StrongControl::Sheaves}
    and
    \ref{Item::Sheaves::Control::LocalEquivWithPreviousDefns::WeakControl::Local}\(\Rightarrow\)\ref{Item::Sheaves::Control::LocalEquivWithPreviousDefns::WeakControl::Sheaves}
    follow.

    \ref{Item::Sheaves::Control::LocalEquivWithPreviousDefns::StrongEquiv::Sheaves}\(\Rightarrow\)\ref{Item::Sheaves::Control::LocalEquivWithPreviousDefns::StrongEquiv::RelCpt}
    follows from 
    \ref{Item::Sheaves::Control::LocalEquivWithPreviousDefns::StrongControl::Sheaves}\(\Rightarrow\)\ref{Item::Sheaves::Control::LocalEquivWithPreviousDefns::StrongControl::RelCpt},
    \ref{Item::Sheaves::Control::LocalEquivWithPreviousDefns::WeakEquiv::Sheaves}\(\Rightarrow\)\ref{Item::Sheaves::Control::LocalEquivWithPreviousDefns::WeakEquiv::RelCpt}
    follows from 
    \ref{Item::Sheaves::Control::LocalEquivWithPreviousDefns::WeakControl::Sheaves}\(\Rightarrow\)\ref{Item::Sheaves::Control::LocalEquivWithPreviousDefns::WeakControl::RelCpt},
    and
    \ref{Item::Sheaves::Control::LocalEquivWithPreviousDefns::StrongEquiv::Local}\(\Rightarrow\)\ref{Item::Sheaves::Control::LocalEquivWithPreviousDefns::StrongEquiv::Sheaves}
    follows from
    \ref{Item::Sheaves::Control::LocalEquivWithPreviousDefns::StrongControl::Local}\(\Rightarrow\)\ref{Item::Sheaves::Control::LocalEquivWithPreviousDefns::StrongControl::Sheaves}.

    \ref{Item::Sheaves::Control::LocalEquivWithPreviousDefns::WeakEquiv::Local}\(\Rightarrow\)\ref{Item::Sheaves::Control::LocalEquivWithPreviousDefns::WeakEquiv::Sheaves}:
    Suppose \ref{Item::Sheaves::Control::LocalEquivWithPreviousDefns::WeakEquiv::Local} holds.
    By \ref{Item::Sheaves::Control::LocalEquivWithPreviousDefns::WeakControl::Local}\(\Rightarrow\)\ref{Item::Sheaves::Control::LocalEquivWithPreviousDefns::WeakControl::Sheaves},
    we have \(\FilteredSheafG\subseteq \LieFilteredSheaf{\FilteredSheafF}\)
    and \(\FilteredSheafF\subseteq \LieFilteredSheaf{\FilteredSheafG}\).
    Since \(\LieFilteredSheaf{\FilteredSheafF}\) and \(\LieFilteredSheaf{\FilteredSheafG}\)
    are Lie algebra filtrations of sheaves of vector fields, this implies
    \(\LieFilteredSheaf{\FilteredSheafG}\subseteq \LieFilteredSheaf{\FilteredSheafF}\) and 
    and \(\LieFilteredSheaf{\FilteredSheafF}\subseteq \LieFilteredSheaf{\FilteredSheafG}\),
    establishing \ref{Item::Sheaves::Control::LocalEquivWithPreviousDefns::WeakEquiv::Sheaves}.
\end{proof}

    \subsection{Filtrations of sheaves on submanifolds and non-characteristic points}\label{Section::Sheaves::Restriction}
    Let \(S\subseteq \ManifoldN\) be a submanifold.
We wish to restrict a filtration of sheaves of vector fields to \(S\),
keeping only those vector fields which are tangent to \(S\).
In this section, we introduce a general definition, and then study two
special cases which are of interest to us, and in which this definition
becomes particularly simple and satisfies good properties:
co-dimension \(0\), embedded submanifolds (see Section \ref{Section::Sheaves::Restriction::CoDim0}),
and the non-characteristic boundary (see Section \ref{Section::Sheaves::Restriction::NCBdry}).


We begin with the general definition.

\begin{definition}
    Let \(\ManifoldN\) be a smooth manifold (possibly with boundary) and \(S\subseteq \ManifoldN\) an embedded submanifold.
    \begin{enumerate}[(i)]
        \item Let \(\SheafF\) be a pre-sheaf of vector fields on \(\ManifoldN\). We let
            \(\RestrictPreSheafF{S}\) be the presheaf of vector fields on \(S\) given by, for \(U\subseteq S\) open,
            \begin{equation}\label{Eqn::Sheaves::Restrict::DefnOfRestriction}
                \RestrictPreSheafF{S}[U]:=
                \left\{ X\big|_{U} : \exists \Omega\subseteq \ManifoldN\text{ open}, \Omega\cap S=U, X\in \SheafF[\Omega], X(x)\in \TangentSpace{S}{x}, \forall x\in U \right\}.
            \end{equation}
            We let \(\RestrcitSheafF{S}\) be the sheafification of \(\RestrictPreSheafF{S}\).

        \item Let \(\FilteredSheafF\) be a filtration of presheaves of vector fields on \(\ManifoldN\).
            We let \(\RestrictFilteredPreSheafF{S}\) be the filtration of presheaves of vector fields
            on \(S\) given by \(d\mapsto \RestrictPreSheaf{\FilteredSheafNoSetF[d]}{S}\)
            and \(\RestrictFilteredSheaf{\FilteredSheafF}{S}\) the filtration of sheaves of vector fields given by
            \(d\mapsto \RestrictFilteredSheaf{\FilteredSheafNoSetF[d]}{S}\).
    \end{enumerate}
\end{definition}

\begin{remark}\label{Rmk::Sheves::Restrcit::BasicRemarksOnRestricDefn}
    \begin{enumerate}[(i)]
        \item In \eqref{Eqn::Sheaves::Restrict::DefnOfRestriction} we restricted attention to only those \(X\)
            which were tangent to \(S\). Thus, unless \(\ManifoldN\) is co-dimension \(0\), this may not include
            all \(X\in \SheafF[\Omega]\).
        \item\label{Item::Sheves::Restrcit::BasicRemarksOnRestricDefn::CommutesWithRestriction} For \(U\subseteq S\) open, we have
            \(\RestrictFilteredPreSheafF{S}\big|_{U}=\RestrictFilteredPreSheafF{S\cap U}\)
            and \(\RestrictFilteredSheaf{\FilteredSheafF}{S}\big|_{U}=\RestrictFilteredSheaf{\FilteredSheafF}{S\cap U}\).

        \item If \(\SheafF\) is a sheaf, then one can show via a partition of unity argument that
            \(\RestrictPreSheafF{S}\) is a sheaf, and therefore
            \(\RestrictPreSheafF{S}=\RestrcitSheafF{S}\). In the special cases we are interested in this follows
            from the results below, so we do not include the proof of the general case.
    \end{enumerate}
\end{remark}

\begin{lemma}\label{Lemma::Sheaves::Restrict::RestrictionOfLieAlgebraIsLieAlgebra}
    Let \(\ManifoldN\) be a smooth manifold (possibly with boundary) and \(S\subseteq \ManifoldN\) an embedded submanifold.
    If \(\FilteredSheafF\) ie a Lie algebra filtration of pre-sheaves of vector fields on \(\ManifoldN\),
    then \(\RestrictFilteredPreSheafF{S}\) is a Lie algebra filtration of pre-sheaves of vector fields on \(S\)
    and \(\RestrictFilteredSheaf{\FilteredSheafF}{S}\) is a Lie algebra filtration of sheaves of vector fields on \(S\).
\end{lemma}
\begin{proof}
    Let \(V_1\in \RestrictFilteredPreSheafF{S}[U][d_1]\) and \(V_2\in \RestrictFilteredPreSheafF{S}[U][d_2]\).
    Then, \(\exists \Omega_1,\Omega_2\subseteq \ManifoldN\) open, with \(\Omega_1\cap S=U=\Omega_2\cap S\),
    and \(X_1\in \FilteredSheafF[\Omega_1][d_1]\), \(X_2\in \FilteredSheafF[\Omega_2][d_1]\)
    with \(X_j\big|_U=V_j\). By hypothesis, \([X_1,X_2]\in \FilteredSheafF[\Omega_1\cap\Omega_2][d_1+d_2]\)
    and therefore,
    \begin{equation*}
        [V_1,V_2]=[X_1,X_2]\big|_{U} \in \RestrictFilteredPreSheafF{S}[U][d_1+d_2],    
    \end{equation*}
    establishing that \(\RestrictFilteredPreSheafF{S}\) is a Lie algebra filtration of pre-sheaves of vector fields on \(S\).
    Since the sheafification of a Lie algebra filtration of pre-sheaves of vector fields is a Lie algebra filtration
    of sheaves of vector fields, the result follows.
\end{proof}

        \subsubsection{Co-dimension \texorpdfstring{\(0\)}{0} submanifolds}\label{Section::Sheaves::Restriction::CoDim0}
        We are mainly concerned with filtered sheaves of vector fields of finite type satisfying H\"ormander's condition.
The next proposition shows that these properties are preserved when restricting to a co-dimension \(0\), embedded submanifold.

\begin{proposition}\label{Prop::Sheaves::Restrict::Codim0::MainRestrictionResult}
    Suppose \(\ManifoldM\) is a smooth manifold (possibly with boundary) and \(\ManifoldN\subseteq \ManifoldM\)
    is a co-dimension \(0\) embedded submanifold (possibly with boundary). Let \(\FilteredSheafF\)
    be a filtration of sheaves of vector fields on \(\ManifoldM\).
    \begin{enumerate}[(i)]
        \item\label{Item::Sheaves::Restrict::Codim0::MainRestrictionResult::GloballyFG} Suppose \(\exists F\subset \VectorFieldsM\times \Zg\) finite with
            \(\FilteredSheafF=\FilteredSheafGenBy{F}\). Then,
            \(\RestrictFilteredPreSheafF{\ManifoldN}=\FilteredSheafGenBy{F\big|_\ManifoldN}\),
            where 
            \begin{equation}\label{Eqn::Sheaves::Restrict::Codim0::RestrictFSet}
                F\big|_{\ManifoldN}=\left\{ (X\big|_{\ManifoldN},d) :(X,d)\in F \right\}.
            \end{equation}
            In particular, \(\RestrictFilteredPreSheafF{\ManifoldN}\) is a filtration of sheaves
            of vector fields on \(\ManifoldN\) 
            (i.e., \(\RestrictFilteredPreSheafF{\ManifoldN}=\RestrictFilteredSheaf{\FilteredSheafF}{\ManifoldN}\)).
        \item\label{Item::Sheaves::Restrict::Codim0::MainRestrictionResult::FGOnSet} Let \(\Omega\subseteq \ManifoldM\) be an open set such that there exists \(F\subset \VectorFieldsM\times \Zg\)
            finite with \(\FilteredSheafF[\Omega]=\FilteredSheafGenBy{F}[\Omega]\). Then,
            \begin{equation*}
                \RestrictFilteredSheaf{\FilteredSheafF}{\ManifoldN}\big|_{\Omega\cap \ManifoldN}
                =\RestrictFilteredSheaf{\FilteredSheafF}{\Omega\cap \ManifoldN}
                =\RestrictFilteredPreSheafF{\Omega\cap \ManifoldN}
                =\RestrictFilteredPreSheafF{\ManifoldN}\big|_{\Omega}
                =\FilteredSheafGenBy{F\big|_{\Omega\cap \ManifoldN}},
            \end{equation*}
            where \(F\big|_{\Omega\cap \ManifoldN}\) is defined by \eqref{Eqn::Sheaves::Restrict::Codim0::RestrictFSet}.

        \item\label{Item::Sheaves::Restrict::Codim0::MainRestrictionResult::FG} If \(\FilteredSheafF\) is of finite type, then \(\RestrictFilteredSheaf{\FilteredSheafF}{\ManifoldN}\)
            is of finite type.

        \item\label{Item::Sheaves::Restrict::Codim0::MainRestrictionResult::Hormander} If \(\FilteredSheafF\) satisfies H\"ormander's condition, then \(\RestrictFilteredSheaf{\FilteredSheafF}{\ManifoldN}\)
            satisfies H\"ormander's condition.
    \end{enumerate}
\end{proposition}
\begin{proof}
    \ref{Item::Sheaves::Restrict::Codim0::MainRestrictionResult::GloballyFG}:
    Clearly \(\FilteredSheafGenBy{F\big|_\ManifoldN}\subseteq \RestrictFilteredPreSheafF{\ManifoldN}\).
    For the reverse containment, take \(U\subseteq \ManifoldN\) open and
    \(V\in \RestrictFilteredPreSheafF{\ManifoldN}[U][d]\). By definition, there exists
    \(\Omega\subseteq \ManifoldM\) open and \((X,d)\in \FilteredSheafF[\Omega][d]=\FilteredSheafGenBy{F}[\Omega][d]\)
    with \(\Omega\cap \ManifoldN=U\) and \(X\big|_{U}=V\).
    Since \((X,d)\in \FilteredSheafGenBy{F}[\Omega][d]\), using the definition (see \eqref{Eqn::Sheaves::Filtered::SheafGenByFiniteSet::DefnOfSheafGenByFiniteSet}),
    we have
    \begin{equation}\label{Eqn::Sheaves::Restrict::Codim0::MainRestrictionResult::Tmp1}
        X=\sum_{\substack{e\leq d \\ (Y,e)\in F}} a_{(Y,e)} Y, \quad a_{(Y,e)}\in \CinftySpace[\Omega].
    \end{equation}
    Restricting \eqref{Eqn::Sheaves::Restrict::Codim0::MainRestrictionResult::Tmp1} to \(U\)
    shows \(V\in \FilteredSheafGenBy{F\big|_\ManifoldN}[U][d]\), as desired.

    We have shown \(\FilteredSheafGenBy{F\big|_\ManifoldN}= \RestrictFilteredPreSheafF{\ManifoldN}\).
    Since \(\FilteredSheafGenBy{F\big|_\ManifoldN}\) is a filtration sheaves of vector fields on \(\ManifoldN\)
    (see Lemma \ref{Lemma::Sheaves::Filtered::SheafGenByFiniteSet}), we conclude
    \(\RestrictFilteredPreSheafF{\ManifoldN}=\RestrictFilteredSheaf{\FilteredSheafF}{\ManifoldN}=\FilteredSheafGenBy{F\big|_\ManifoldN}\),
    completing the proof of \ref{Item::Sheaves::Restrict::Codim0::MainRestrictionResult::GloballyFG}.

    \ref{Item::Sheaves::Restrict::Codim0::MainRestrictionResult::FGOnSet}:
    Let \(\Omega\) and \(F\) be as in \ref{Item::Sheaves::Restrict::Codim0::MainRestrictionResult::FGOnSet}.
    By Remark \ref{Rmk::Sheves::Restrcit::BasicRemarksOnRestricDefn} \ref{Item::Sheves::Restrcit::BasicRemarksOnRestricDefn::CommutesWithRestriction},
    we have \(\RestrictFilteredSheaf{\FilteredSheafF}{\ManifoldN}\big|_{\Omega\cap \ManifoldN}=\RestrictFilteredSheaf{\FilteredSheafF}{\Omega\cap \ManifoldN}\)
    and \(\RestrictFilteredPreSheafF{\Omega\cap \ManifoldN}=\RestrictFilteredPreSheafF{\ManifoldN}\big|_{\Omega}\).
    By \ref{Item::Sheaves::Restrict::Codim0::MainRestrictionResult::GloballyFG} with \(\ManifoldN\)
    replaced by \(\Omega\cap \ManifoldN\), we have
    \(
        \RestrictFilteredSheaf{\FilteredSheafF}{\Omega\cap \ManifoldN}
        =\RestrictFilteredPreSheafF{\Omega\cap \ManifoldN}
        =\FilteredSheafGenBy{F\big|_{\Omega\cap \ManifoldN}},
    \)
    establishing \ref{Item::Sheaves::Restrict::Codim0::MainRestrictionResult::FGOnSet}.

    \ref{Item::Sheaves::Restrict::Codim0::MainRestrictionResult::FG}
    follows immediately from \ref{Item::Sheaves::Restrict::Codim0::MainRestrictionResult::FGOnSet}.

    \ref{Item::Sheaves::Restrict::Codim0::MainRestrictionResult::Hormander}:
    Let \(\SheafGh\) be the presheaf of vector fields on \(\ManifoldM\) defined by
    \begin{equation*}
        \SheafGh[\Omega]:=\bigcup_{j} \LieFilteredSheaf{\FilteredSheafF}[\Omega][j].
    \end{equation*}
    By H\"ormander's condition (see Definition \ref{Defn::Sheaves::LieAlg::HormandersCondition}), the sheafification
    of \(\SheafGh\) equals the sheaf of all smooth vector fields on \(\ManifoldM\), which we denote by \(\SheafVectorFields\). 
    By \cite[\href{https://stacks.math.columbia.edu/tag/007Z}{Lemma 007Z}]{stacks-project}, the stalk of \(\SheafGh\)
    at each point \(x\in \ManifoldM\) equals the stalk of \(\SheafVectorFields\) at \(x\): \(\SheafGh_x=\SheafVectorFields_x\).
    It follows directly from the definitions that, for \(\Omega\subseteq \ManifoldM\) open,
    \begin{equation*}
        \left\{ X\big|_{\Omega\cap \ManifoldN} : X\in \LieFilteredSheafF[\Omega][j] \right\}
        \subseteq \LieFilteredSheaf{\RestrictFilteredPreSheafF{\ManifoldN}}[\Omega\cap \ManifoldN][j]
        \subseteq \LieFilteredSheaf{\RestrictFilteredSheaf{\FilteredSheafF}{\ManifoldN}}[\Omega\cap \ManifoldN][j].
    \end{equation*}
    Thus, if we define the presheaf of vector fields on \(\ManifoldN\) by
    \begin{equation*}
        \SheafG[U]:=\bigcup_{j} \LieFilteredSheaf{\RestrictFilteredSheaf{\FilteredSheafF}{\ManifoldN}}[U][j],
    \end{equation*}
    we have
    \begin{equation*}
        \left\{ X\big|_{\Omega\cap \ManifoldN} : X\in \SheafGh[\Omega] \right\}
        \subseteq \SheafG[\Omega\cap \ManifoldN].
    \end{equation*}
    For \(x\in \ManifoldN\),
    since the stalk \(\SheafGh_x\) contains the germs at \(x\) of all smooth vector fields on \(\ManifoldM\) near \(x\),
    we conclude that the stalk \(\SheafG_x\) contains the germs of all smooth vector fields on \(\ManifoldN\) near \(x\).
    It follows from \cite[\href{https://stacks.math.columbia.edu/tag/007Z}{Lemma 007Z}]{stacks-project}
    that the sheafification of \(\SheafG\) equals the sheaf of all smooth vector fields on \(\ManifoldN\),
    which is the definition of \(\RestrictFilteredSheaf{\FilteredSheafF}{\ManifoldN}\)
    satisfying H\"ormander's condition.
\end{proof}

Proposition \ref{Prop::Sheaves::Restrict::Codim0::MainRestrictionResult} shows that restricting
a filtration of sheaves of finite type satisfying H\"ormander's condition to a co-dimension \(0\),
embedded submanifold with boundary again yields a filtration of sheaves of finite type satisfying H\"ormander's condition.
The next result shows that every such filtration of sheaves of finite type on \(\ManifoldN\) satisfying H\"ormander's 
condition arises in this way.

\begin{proposition}
    Let \(\ManifoldN\) be a smooth manifold with boundary and \(\FilteredSheafF\) a filtration of sheaves
    of vector fields on \(\ManifoldN\) which is of finite type.
    Then, there exists a smooth manifold \(\ManifoldM\) (without boundary), such that
    \(\ManifoldN\subseteq \ManifoldM\) is a closed, co-dimension \(0\) embedded submanifold,
    and \(\FilteredSheafFh\) a filtration of sheaves of vector fields on \(\ManifoldM\) which is of finite
    type and satisfies \(\RestrictFilteredSheaf{\FilteredSheafFh}{\ManifoldN}=\FilteredSheafF\).
    If \(\FilteredSheafF\) satisfies H\"ormander's condition, then \(\ManifoldM\) and \(\FilteredSheafFh\)
    can be chosen so that \(\FilteredSheafFh\) also satisfies H\"ormander's condition.
\end{proposition}
\begin{proof}
    Let \(\ManifoldMt\) be the double of \(\ManifoldN\).
    Since \(\FilteredSheafF\) is finite type, for each \(x\in \InteriorN\), there exists \(U_x\subseteq \InteriorN\), an open neighborhood of \(x\),
    such that and \(\WWdx:=\left\{ (\Wx_1,\Wdx_1),\ldots, (\Wx_{r_x}, \Wdx_{r_x}) \right\}\subset \VectorFields{U_x}\times \Zg\)
    a finite set with
    \begin{equation*}
        \FilteredSheafF\big|_{U_x} = \FilteredSheafGenBy{\WWdx}.
    \end{equation*}
    Moreover, by Lemma \ref{Lemma::Sheaves::LieAlg::TwoNotionsOfHormanderAreTheSame}, \(\Wx_1,\ldots, \Wx_{r_x}\)
    satisfy H\"ormander's condition on \(U_x\).
    Similarly, for each \(x\in \BoundaryN\) we may choose a neighborhood \(U_x\subseteq \ManifoldMt\) of \(x\)
    and \(\WWdx:=\left\{ (\Wx_1,\Wdx_1),\ldots, (\Wx_{r_x}, \Wdx_{r_x}) \right\}\subset \VectorFields{U_x\cap \ManifoldM}\times \Zg\)
    such that
    \begin{equation*}
        \FilteredSheafF\big|_{U_x\cap \ManifoldM} = \FilteredSheafGenBy{\WWdx},
    \end{equation*}
    and such that there exist \(\Whx_1,\ldots, \Whx_{r_x}\in \VectorFields{U_x}\) with \(\Whx_j\big|_{U_x\cap \ManifoldN}= \Wx_j\).
    By Lemma \ref{Lemma::Sheaves::LieAlg::TwoNotionsOfHormanderAreTheSame}, \(\Wx_1,\ldots, \Wx_{r_x}\)
    satisfy H\"ormander's condition on \(U_x\cap \ManifoldN\), and by possibly shrinking \(U_x\) if necessary,
    we may assume \(\Whx_1,\ldots, \Whx_{r_x}\) satisfy H\"ormander's condition on \(U_x\).

    Pick \(\phi_\alpha,\phit_\alpha\in \CzinftySpace[\ManifoldMt]\), where \(\alpha\) ranges over some index set \(\sI\),
    such that:
    \begin{itemize}
        \item \(\phi_\alpha,\phit_\alpha\in \CzinftySpace[U_{x(\alpha)}]\) for some \(x(\alpha)\in \ManifoldN\).
        \item \(\phit_\alpha=1\) on a neighborhood of the \(\supp[\phi_\alpha]\).
        \item On each compact subset of \(\ManifoldMt\), only finitely many \(\phit_\alpha\) are nonzero.
        \item \(\sum_{\alpha} \phi_\alpha=1\) on a neighborhood of \(\ManifoldN\).
    \end{itemize}
    Let \(\FilteredSheafFt\) be the filtration of sheaves of vector fields on \(\ManifoldMt\)
    generated by
    \begin{equation}\label{Eqn::Sheaves::Restrict::Codim0::LocalizedGeneratorsTmp1}
        \sS:=\left\{ (\phi_{\alpha} \Wh_1^{x(\alpha)}, \Wdxalpha_1), \ldots, (\phi_\alpha \Wh_{r_{x(\alpha)}}^{x(\alpha)}, \Wdxalpha_{r_{x(\alpha)}}) \right\}.
    \end{equation}
    On any compact set, only finitely many of the vector fields appearing in \eqref{Eqn::Sheaves::Restrict::Codim0::LocalizedGeneratorsTmp1}
    are non-zero, so \(\FilteredSheafFt\) is of finite type.
    Proposition \ref{Prop::Sheaves::Restrict::Codim0::MainRestrictionResult} \ref{Item::Sheaves::Restrict::Codim0::MainRestrictionResult::FG}
    shows \(\RestrictFilteredSheaf{\FilteredSheafFt}{\ManifoldN}\) is also of finite type.

    We claim \(\RestrictFilteredSheaf{\FilteredSheafFt}{\ManifoldN}=\FilteredSheafF\).
    We first show \(\RestrictFilteredSheaf{\FilteredSheafFt}{\ManifoldN}\subseteq \FilteredSheafF\).
    Fix \(x\in \ManifoldN\) and take \(\Omega\Subset \ManifoldMt\) a relatively compact open set
    with \(x\in \Omega\).
    Let \(F:=\left\{ (X,d)\in \sS : \supp[X]\cap \Omega=\emptyset \right\}\),
    so that \(F\) is a finite set and \(\FilteredSheafFt\big|_{\Omega}=\FilteredSheafGenBy{F}\big|_{\Omega}\).
    By Proposition \ref{Prop::Sheaves::Restrict::Codim0::MainRestrictionResult} \ref{Item::Sheaves::Restrict::Codim0::MainRestrictionResult::FGOnSet},
    \(\RestrictFilteredSheaf{\FilteredSheafF}{\Omega\cap \ManifoldN}=\FilteredSheafGenBy{F\big|_{\Omega\cap \ManifoldN}}\).
    But it is clear from the construction that
    \(\FilteredSheafGenBy{F\big|_{\Omega\cap \ManifoldN}}\subseteq \FilteredSheafF\big|_{\Omega\cap \ManifoldN}\),
    and so \(\RestrictFilteredSheaf{\FilteredSheafF}{\Omega\cap \ManifoldN}\subseteq \FilteredSheafF\big|_{\Omega\cap \ManifoldN}\).
    Since \(x\in \ManifoldN\) was arbitrary, we see \(\RestrictFilteredSheaf{\FilteredSheafFt}{\ManifoldN}\subseteq \FilteredSheafF\).

    Next we show \(\FilteredSheafF\subseteq \RestrictFilteredSheaf{\FilteredSheafFt}{\ManifoldN}\).
    Fix \(x\in \ManifoldN\), take \(\Omega\subseteq \ManifoldN\) open with \(x\in \Omega\), \(d\in Zg\),
    and \(X\in \FilteredSheafF[\Omega][d]\). 
    Let \(\phi\in \CzinftySpace[\ManifoldMt]\)
    equal \(1\) on a neighborhood of \(x\), and such that \(\supp[\phi]\cap \ManifoldN\subseteq \Omega\).
    We will show \(\phi X\in \RestrictFilteredSheaf{\FilteredSheafNoSetFt[d]}{\ManifoldN}[\ManifoldN]\) and it will follow that
    \(\FilteredSheafF\subseteq \RestrictFilteredSheaf{\FilteredSheafFt}{\ManifoldN}\).

    For each \(\alpha\in \sI\), \(\phit_\alpha \phi X\in \FilteredSheafF[U_{\alpha(x)}][d]\) and so
    \(\phit_\alpha \phi = \sum_{j=1}^{r_{\alpha(x)}} a_{\alpha,j} W_j^{\alpha(x)}\), where
    \(a_\alpha\in \CinftySpace[U_{\alpha(x)}\cap \ManifoldN]\) and \(\phit_\alpha \phi X=0\) for all but finitely many \(\alpha\).

    We have (for some finite subset \(\sI_0\subseteq \sI\)):
    \begin{equation*}
        \phi X = \sum_{\alpha\in \sI_0} \phi_{\alpha} \phit_\alpha \phi X = \sum_{\alpha\in \sI_0} \sum_{j=1}^{r_{\alpha(x)}} \phi_\alpha a_{\alpha,j} W_j^{\alpha(x)}.
    \end{equation*}
    Extend each \(a_{\alpha,j}\) to a function \(\at_{\alpha,j}\in \CinftySpace[U_{\alpha(x)}]\)
    and set
    \begin{equation*}
        Z:=\sum_{\alpha\in \sI_0} \sum_{j=1}^{r_{\alpha(x)}} \phi_{\alpha} \at_{\alpha} \Wh_j^{\alpha(x)}.
    \end{equation*}
    Then, \(Z\in \FilteredSheafFt[\ManifoldMt][d]\) and \(Z\big|_{\ManifoldN}=\phi X\).
    We conclude \(\phi X \in \RestrictFilteredPreSheaf{\FilteredSheafFt}{\ManifoldN}[\ManifoldN][d]   \subseteq \RestrictFilteredSheaf{\FilteredSheafNoSetFt[d]}{\ManifoldN}[\ManifoldN]\).
    This completes the proof that \(\RestrictFilteredSheaf{\FilteredSheafFt}{\ManifoldN}=\FilteredSheafF\).

    \(\FilteredSheafFt\) and \(\ManifoldMt\) satisfy the conclusions of the proposition (in place of \(\FilteredSheafFh\) and \(\ManifoldM\)),
    except for the claim about H\"ormander's condition.
    If \(\FilteredSheafF\) satisfies H\"ormander's condition, then by construction the vector fields in
    \eqref{Eqn::Sheaves::Restrict::Codim0::LocalizedGeneratorsTmp1} also satisfy H\"ormander's condition at every point
    of \(\ManifoldN\). By continuity, the same is true on an open neighborhood \(\ManifoldM\subseteq \ManifoldMt\) of \(\ManifoldN\).
    We set \(\FilteredSheafFh:=\FilteredSheafFt\big|_{\ManifoldM}\); so that \(\FilteredSheafFh\)
    is of finite type and \(\RestrictFilteredSheaf{\FilteredSheafFh}{\ManifoldN}=\FilteredSheafF\).
    By Proposition \ref{Prop::Sheaves::LieAlg::HormandersCondiIsEquivToOtherNotions}, \(\FilteredSheafFh\)
    satisfies H\"ormander's condition, completing the proof.
\end{proof}

        \subsubsection{The non-characteristic boundary}\label{Section::Sheaves::Restriction::NCBdry}
        Throughout this section, \(\ManifoldN\) is a smooth manifold with boundary, and
\(\FilteredSheafF\) is a filtration of sheaves of vector fields on \(\ManifoldN\).

\begin{definition}
    We define \(\degBoundaryNF:\BoundaryN\rightarrow \Zg\cup\{\infty\}\) by,
    for \(x_0\in \BoundaryN\),
    \begin{equation*}
        \degBoundaryNF[x_0]:=\inf \left\{ d\in \Zg : \exists \Omega\subseteq \ManifoldN\text{ open}, x_0\in \Omega, X\in \LieFilteredSheafF[\Omega][d], X(x_0)\not\in \TangentSpace{\BoundaryN}{x_0} \right\}.
    \end{equation*}
\end{definition}

If \(\FilteredSheafF\) satisfies H\"ormander's condition, \(\degBoundaryNF[x]\) is always finite.
\(\degBoundaryNF:\BoundaryN\rightarrow \Zg\cup\{\infty\}\) is clearly upper semi-continuous, and depends on
\(\FilteredSheafF\) only through \(\LieFilteredSheafF\).

\begin{definition}
    We say \(x_0\in \BoundaryN\) is \(\FilteredSheafF\)-non-characteristic if \(x\mapsto \degBoundaryNF(x)\),
    \(\BoundaryN\rightarrow \Zg\),
    is continuous at \(x=x_0\). I.e., if \(\degBoundaryNF\) is constant on a \(\BoundaryN\)-neighborhood of \(x_0\).
\end{definition}

Let \(\BoundaryNncF:=\left\{ x\in \BoundaryN: x\text{ is }\FilteredSheafF\text{-non-characteristic} \right\}\),
which is an open subset of \(\BoundaryN\). Set \(\ManifoldNncF:=\BoundaryNncF\cup \InteriorN\) with the manifold
structure given as an open submanifold of \(\ManifoldN\). Every point of \(\BoundaryNncF\) is
\(\FilteredSheafF\)-non-characteristic.
The next proposition shows that these notions correspond to the similar definitions from
Section \ref{Section::Defns::NonChar}.

\begin{proposition}\label{Prop::Sheaves::Restrict::Bndry::NonCharIsNonChar}
    Let \(x_0\in \BoundaryN\) be such that \(\exists \Omega\subseteq \ManifoldN\) open with \(x_0\in \ManifoldN\),
    and \(\WWd=\left\{ \left( W_1,\Wd_1 \right),\ldots, \left( W_r, \Wd_r \right) \right\}\subset \VectorFields{\Omega}\times \Zg\)
    which are H\"ormander vector fields with formal degrees on \(\Omega\), such that
    \(\FilteredSheafF\big|_{\Omega}=\FilteredSheafGenBy{\WWd}\). Then,
    \begin{enumerate}[(i)]
        \item\label{Item::Sheaves::Restrict::Bndry::NonCharIsNonChar::DegIsSame} \(\degBoundaryOmegaWWd[x]=\degBoundaryNF[x]\), \(\forall x\in \BoundaryOmega\).
        \item\label{Item::Sheaves::Restrict::Bndry::NonCharIsNonChar::NonCharIsSame} \(x\in \BoundaryOmega\) is \(\WWd\)-non-characteristic if and only if \(x\) is \(\FilteredSheafF\)-non-characteristic.
    \end{enumerate}
\end{proposition}
\begin{proof}
    \ref{Item::Sheaves::Restrict::Bndry::NonCharIsNonChar::DegIsSame}: It follows directly from the definitions
    that \(\degBoundaryNF[x]\leq \degBoundaryOmegaWWd[x]\) so we focus on the reverse inequality. 
    Fix 
    \(x_0\in \BoundaryOmega\) and
    \(\Omegat\Subset \Omega\) a neighborhood of \(x_0\) which is relatively compact in \(\Omega\).
    It follows immediately from the definitions that
    \(
        \degBoundaryNF[x_0]\leq \degParams{\partial \Omegat}{\FilteredSheafF\big|_{\Omegat}}[x_0],
    \)
    in fact, a simple argument shows equality holds but we do not require it.
    It also follows immediately from the definitions that
    \(\degBoundaryOmegaWWd[x_0] = \degParams{\partial \Omegat}{\WWd}[x_0]\).
    Thus, we will show
    \begin{equation}\label{Eqn::Sheaves::Restrict::Bndry::NonCharIsNonChar::ToShow}
        \degParams{\partial \Omegat}{\FilteredSheafF\big|_{\Omegat}}[x_0]= \degParams{\partial \Omegat}{\WWd}[x_0],
    \end{equation}
    which will complete the proof.

    Since \(\Omegat\Subset \Omega\), \(W_1,\ldots, W_r\) satisfy H\"ormander's condition or order \(m\in \Zg\)
    for some \(m\).  Define \(\XXd=\left\{ (X_1,\Xd_1),\ldots, (X_q,\Xd_q) \right\}\subset \VectorFields{\Omegat}\times \Zg\)
    as in \eqref{Eqn::Sheaves::LieAlg::DefnXXd}
    so that by Lemma \ref{Lemma::Sheaves::LieAlg::LieAlgebraFilIsGenByXXd} we have
    \begin{equation}\label{Eqn::Sheaves::Restrict::Bndry::NonCharIsNonChar::Tmp1}
        \LieFilteredSheafF\big|_{\Omegat}=\FilteredSheafGenBy{\XXd}.
    \end{equation}
    \(\XXd\) satisfies \eqref{Eqn::Sheaves::LieAlg::NSW}, and it follows that
    \begin{equation*}
        \degParams{\partial \Omegat}{\WWd}[x_0]=\degParams{\partial \Omegat}{\XXd}[x_0]
        =\min \left\{ \Xd_j : X_j(x_0)\not \in \TangentSpace{\BoundaryN}{x_0} \right\}
        =\degParams{\partial \Omegat}{\FilteredSheafF\big|_{\Omegat}}[x_0],
    \end{equation*}
    where the last equality follows by \eqref{Eqn::Sheaves::Restrict::Bndry::NonCharIsNonChar::Tmp1}.
    This proves \eqref{Eqn::Sheaves::Restrict::Bndry::NonCharIsNonChar::ToShow}
    and completes the proof of \ref{Item::Sheaves::Restrict::Bndry::NonCharIsNonChar::DegIsSame}.

    \ref{Item::Sheaves::Restrict::Bndry::NonCharIsNonChar::NonCharIsSame} is an immediate consequence of 
    \ref{Item::Sheaves::Restrict::Bndry::NonCharIsNonChar::DegIsSame}.
\end{proof}

\begin{construction}\label{Construction::Sheaves::Restrict::Bndry::VVd}
    Let \(x_0\in \BoundaryNncF\). Suppose \(\Omega\subseteq \ManifoldN\) is an open set with
    \(\FilteredSheafF\big|_{\Omega}=\FilteredSheafGenBy{\WWd}\) where
    \(\WWd=\left\{ \left( W_1,\Wd_1 \right),\ldots,\left( W_r,\Wd_r \right) \right\}\subset \VectorFields{\Omega}\times \Zg\)
    are H\"ormander vector fields with formal degrees on \(\Omega\).
    On a small neighborhood \(U\subseteq \BoundaryNncF\) of \(x_0\),
    we construct a list of vector fields with formal degrees 
    \(\VVd=\left\{ (V_1,\Vd_1),\ldots, (V_q,\Vd_q) \right\}\subset \VectorFields{U}\times \Zg\)
    such that
    \begin{equation*}
        \RestrictFilteredSheaf{\LieFilteredSheafF}{U}=\FilteredSheafGenBy{\VVd}
    \end{equation*}
    by the procedure in Section \ref{Section::BndryVfsWithFormaLDegrees}. More precisely,
    \begin{enumerate}[(A)]
        \item Let \(\Omega_1\subseteq \Omega\) be an \(\ManifoldN\)-neighborhood of \(x_0\) and \(m\in \Zg\)
            be such that \(W_1,\ldots,W_r\) satisfy H\"ormander's condition of order \(m\) on \(\Omega_1\) (such an \(\Omega_1\)
            always exists; for example, one could take \(\Omega_1\Subset \Omega\)).
        
        \item Define \(\ZZd\) as in Section \ref{Section::BndryVfsWithFormaLDegrees}, Step \ref{Item::BoundaryVfs::DefineZZd}.
            By Lemma \ref{Lemma::Sheaves::LieAlg::LieAlgebraFilIsGenByXXd},
            \(\LieFilteredSheafF\big|_{\Omega_1}=\FilteredSheafGenBy{\ZZd}\big|_{\Omega_1}\).

        \item Define \(\XXd\)
        on some \(\ManifoldN\)-neighborhood \(\Omega_2\subseteq \Omega_1\) of \(x_0\),
        by Section \ref{Section::BndryVfsWithFormaLDegrees}, 
        Steps \ref{Item::BoundaryVfs::FindWj0}-\ref{Item::BoundaryVfs::DefineXj}.
        By the formula for \(\XXd\), we have \(\FilteredSheafGenBy{\XXd}\big|_{\Omega_2}=\FilteredSheafGenBy{\ZZd}\big|_{\Omega_2}\)
        and therefore, \(\LieFilteredSheafF\big|_{\Omega_2}=\FilteredSheafGenBy{\XXd}\big|_{\Omega_2}\).

        \item 
        Let \(U=\Omega_2\cap \BoundaryN\) and
        define \(\VVd=\left\{ (V_1,\Vd_1),\ldots, (V_q,\Vd_q) \right\}\subset \VectorFields{U}\times \Zg\) as in 
        Section \ref{Section::BndryVfsWithFormaLDegrees}, Step \ref{Item::BoundaryVfs::DefineVj}. 
        It follows from the definition of \(\VVd\) that
        \begin{equation}\label{Eqn::Sheaves::Restrict::Bndry::ConstructVVd::PreSheaf}
            \RestrictFilteredPreSheaf{\LieFilteredSheafF}{U}=\FilteredSheafGenBy{\VVd}.
        \end{equation}

        \item Lemma \ref{Lemma::Sheaves::Filtered::SheafGenByFiniteSet} shows \(\FilteredSheafGenBy{\VVd}\) is a
            filtration of sheaves of vector fields on \(U\), and \eqref{Eqn::Sheaves::Restrict::Bndry::ConstructVVd::PreSheaf}
            then implies \(\RestrictFilteredSheaf{\LieFilteredSheafF}{U}=\FilteredSheafGenBy{\VVd}\).
    \end{enumerate}
\end{construction}

\begin{proposition}\label{Prop::Sheaves::Restrict::Bndry::RestritionIsFiniteTypeAndSpandTangent}
    Let \(\FilteredSheafF\) be a filtration of sheaves of vector fields on \(\ManifoldN\)
    which is of finite type and satisfies H\"ormander's condition. Then,
    \(\RestrictFilteredSheaf{\LieFilteredSheafF}{\BoundaryNncF}\) is a Lie algebra filtration of sheaves of vector fields
    on \(\BoundaryNncF\) which is of finite type and spans the tangent space at every point (and therefore satisfies H\"ormander's condition).
\end{proposition}
\begin{proof}
    Lemma \ref{Lemma::Sheaves::Restrict::RestrictionOfLieAlgebraIsLieAlgebra} shows
    \(\RestrictFilteredSheaf{\LieFilteredSheafF}{\BoundaryNncF}\) is a Lie algebra filtration of sheaves of vector fields.,
    so we need to show that it is of finite type and spans the tangent space at every point (once we show
    it spans the tangent space at every point, it trivially satisfies H\"ormander's condition--see Definition \ref{Defn::Sheaves::LieAlg::HormandersCondition}).

    Fix \(x_0\in \BoundaryNncF\). Because \(\FilteredSheafF\) is finite type and satisfies H\"ormander's condition, 
    Proposition \ref{Prop::Sheaves::LieAlg::HormandersCondiIsEquivToOtherNotions} shows that
    there exists an open neighborhood \(\Omega\subseteq \ManifoldN\)
    of \(x_0\) and a finite set \(\WWd\subset \VectorFields{\Omega}\times \Zg\) such that
    \begin{equation*}
        \FilteredSheafF\big|_{\Omega}=\FilteredSheafGenBy{\WWd}.
    \end{equation*}
    and \(\WWd\) are H\"ormander vector fields with formal degrees on \(\Omega\).

    Letting 
    \(U\subseteq \BoundaryNncF\) and 
    \(\VVd=\left\{ (V_1,\Vd_1),\ldots, (V_q,\Vd_q) \right\}\subset \VectorFields{U}\times \Zg\) be as in Construction \ref{Construction::Sheaves::Restrict::Bndry::VVd}, with this choice of \(x_0\),
    we have
    \(\RestrictFilteredSheaf{\LieFilteredSheafF}{U}=\FilteredSheafGenBy{\VVd}\).
    Since \(U\) is a neighborhood of \(x_0\), \(\VVd\) is finite,
    and \(x_0\in \BoundaryNncF\) was arbitrary,
    this shows \(\RestrictFilteredSheaf{\LieFilteredSheafF}{\BoundaryNncF}\) is of finite type.
    Moreover, by Section \ref{Section::BndryVfsWithFormaLDegrees}, Step \ref{Item::BoundaryVfs::DefineVj},
    \(V_1(x),\ldots, V_q(x)\) span \(\TangentSpace{\BoundaryN}{x}\), \(\forall x\in U\), and it follows that
    \(\RestrictFilteredSheaf{\LieFilteredSheafF}{\BoundaryNncF}\) spans the tangent space at every point.
\end{proof}

    \subsection{Metrics}\label{Section::Sheaves::Metrics}
    Let \(\ManifoldN\) be a smooth manifold with boundary and \(\FilteredSheafF\) a filtration
of sheaves of vector fields on \(\ManifoldN\) which is of finite type and satisfies H\"ormander's condition.
Our first main result shows that \(\FilteredSheafF\) induces a unique Lipschitz equivalence class of Carnot-Carath\'eodory
metrics on each compact subset of \(\ManifoldNncF\).
For an extended metric \(\rho\), we write \(\rho\wedge 1\) for the metric \((x,y)\mapsto \min\{\rho(x,y),1\}\).

\begin{theorem}\label{Thm::Sheaves::Metrics::ExistsEquivalenceClass}
    For each compact set, \(\Compact\Subset \ManifoldNncF\) there exists a metric \(\rho_\Compact\)
    on \(\Compact\) such that:
    \begin{enumerate}[label=(\roman*), ref=(\roman*)]
        \item\label{Item::Sheaves::Metrics::ExistsEquivalenceClass::AgreesWithCCInducedByF} If \(\Omega\Subset \ManifoldNncF\) is a relatively compact, open set with \(\Compact\Subset \Omega\),
            and \(\FilteredSheafF\big|_{\Omega}=\FilteredSheafGenBy{\WWd}\), where \(\WWd\subset \VectorFields{\Omega}\times \Zg\)
            is a finite set, then 
            \begin{enumerate}[label=(\roman{enumi}.\alph*), ref=(\roman{enumi}.\alph*)]
                \item\label{Item::Sheaves::Metrics::ExistsEquivalenceClass::AgreesWithCCInducedByF::LocallyEquiv} \(\rho_\Compact\) is locally Lipschitz equivalent to \(\rho_{\WWd}\big|_{\Compact\times \Compact}\), and
                \item\label{Item::Sheaves::Metrics::ExistsEquivalenceClass::AgreesWithCCInducedByF::Equiv} \(\rho_\Compact\) is Lipschitz equivalent to \(\rho_{\WWd}\wedge 1\big|_{\Compact\times \Compact}\).
            \end{enumerate}

        \item\label{Item::Sheaves::Metrics::ExistsEquivalenceClass::AgreesWithCCInducedByG} More generally, if \(\FilteredSheafG\) is a filtration of sheaves of vector fields on 
            \(\ManifoldN\) which is finite type and satisfies \(\LieFilteredSheaf{\FilteredSheafG}=\LieFilteredSheaf{\FilteredSheafF}\),
            and if \(\Omega\Subset \ManifoldNncF\) is a relatively compact, open set with \(\Compact\Subset \Omega\),
            and \(\FilteredSheafG\big|_{\Omega}=\FilteredSheafGenBy{\WWd}\), where \(\WWd\subset \VectorFields{\Omega}\times \Zg\)
            is a finite set, then 
            \begin{enumerate}[label=(\roman{enumi}.\alph*), ref=(\roman{enumi}.\alph*)]
                \item\label{Item::Sheaves::Metrics::ExistsEquivalenceClass::AgreesWithCCInducedByG::LocallyEquiv} \(\rho_\Compact\) is locally Lipschitz equivalent to \(\rho_{\WWd}\big|_{\Compact\times \Compact}\), and
                \item\label{Item::Sheaves::Metrics::ExistsEquivalenceClass::AgreesWithCCInducedByG::Equiv} \(\rho_\Compact\) is Lipschitz equivalent to \(\rho_{\WWd}\wedge 1\big|_{\Compact\times \Compact}\).
            \end{enumerate}

        \item\label{Item::Sheaves::Metrics::ExistsEquivalenceClass::SameTop} The metric topology on \(\Compact\) induced by \(\rho_\Compact\) agrees with the subspace topology on \(\Compact\)
            as a subspace of \(\ManifoldN\).

        \item\label{Item::Sheaves::Metrics::ExistsEquivalenceClass::Unique} If \(\rhot_\Compact\) is any other extended metric on \(\Compact\) satisfying 
        \ref{Item::Sheaves::Metrics::ExistsEquivalenceClass::AgreesWithCCInducedByF::LocallyEquiv},
            then \(\rhot_\Compact\) is locally Lipschitz equivalent to \(\rho_\Compact\). If, in addition, \(\rhot_\Compact\)
            is a metric and \ref{Item::Sheaves::Metrics::ExistsEquivalenceClass::SameTop} holds for \(\rhot_\Compact\), then \(\rhot_\Compact\) is Lipschitz equivalent to \(\rho_\Compact\).

        \item\label{Item::Sheaves::Metrics::ExistsEquivalenceClass::Restrict} \(\Compact_1\Subset \Compact_2\Subset\ManifoldNncF\) are two compact sets, then
            \(\rho_{\Compact_2}\big|_{\Compact_1\times \Compact_1}\) is Lipschitz equivalent to \(\rho_{\Compact_1}\).
        
    \end{enumerate}
\end{theorem}

\begin{notation}\label{Notation::Sheaves::Metrics::WellDefined}
    Theorem \ref{Thm::Sheaves::Metrics::ExistsEquivalenceClass} picks out a unique Lipschitz equivalence class of metrics
    on each compact set in \(\ManifoldNncF\)
    (see Theorem \ref{Thm::Sheaves::Metrics::ExistsEquivalenceClass} \ref{Item::Sheaves::Metrics::ExistsEquivalenceClass::Unique}). Thus, if one only cares about the Lipschitz equivalence class of the metric on a given compact
    set \(\Compact\Subset \ManifoldNncF\), it makes sense to write \(\rhoF\big|_{\Compact\times \Compact}\)
    for a choice of metric from this equivalence class.
    Theorem \ref{Thm::Sheaves::Metrics::ExistsEquivalenceClass} \ref{Item::Sheaves::Metrics::ExistsEquivalenceClass::Restrict}
    shows that this equivalence class behaves as one would expect with respect to restrictions:
    if \(\Compact_1\Subset \Compact_2\Subset \ManifoldNncF\) are compact sets, then
    \(\rhoF\big|_{\Compact_2\times \Compact_2}\big|_{\Compact_1\times \Compact_1}\)
    is Lipschitz equivalent to \(\rhoF\big|_{\Compact_1\times \Compact_1}\).
\end{notation}

\begin{proof}[Proof of Theorem \ref{Thm::Sheaves::Metrics::ExistsEquivalenceClass}]
    Fix \(\Omega\Subset \ManifoldNncF\) open an relatively compact,  with \(\Compact\Subset \Omega\),
    and take
    H\"ormander vector fields with formal degrees on \(\Omega\)
    \(\WWd=\left\{ (W_1,\Wd_1),\ldots, (W_r,\Wd_r) \right\}\subset \VectorFields{\Omega}\times \Zg\)
    with \(\FilteredSheafF\big|_{\Omega}=\FilteredSheafGenBy{\WWd}\) (this is always possible;
    see Proposition \ref{Prop::Sheaves::LieAlg::HormandersCondiIsEquivToOtherNotions}).
    For \(x,y\in \Compact\), set \(\rho_{\Compact}(x,y):=\rho_{\WWd}(x,y)\wedge 1\).
    By Theorem \ref{Thm::Metrics::Results::GivesUsualTopology}, \(\rho_{\WWd}\) is a metric on each connected
    component of \(\Omega\), and the topology induced by \(\rho_{\WWd}\) agrees with the topology on \(\Omega\)
    as a subspace of \(\ManifoldN\). It follows that \(\rho_{\Compact}\) is a metric and \ref{Item::Sheaves::Metrics::ExistsEquivalenceClass::SameTop} holds.

    \ref{Item::Sheaves::Metrics::ExistsEquivalenceClass::AgreesWithCCInducedByG}:
    Since \(\LieFilteredSheaf{\FilteredSheafG}=\LieFilteredSheaf{\FilteredSheafF}\), \(\FilteredSheafG\)
    also satisfies H\"ormander's condition and \(\ManifoldNncF=\ManifoldNncG\).
    Take \(\Omega'\Subset \ManifoldNncG=\ManifoldNncF\) open and relatively compact, and let
    \(\ZZd=\left\{ (Z_1,\Zd_1),\ldots, (Z_s,\Zd_s) \right\}\subset \VectorFields{\Omega'}\times \Zg\) 
    be such that \(\FilteredSheafG\big|_{\Omega'}=\FilteredSheafGenBy{\ZZd}\).
    In light of Lemma \ref{Lemma::Sheaves::LieAlg::TwoNotionsOfHormanderAreTheSame}, \(\ZZd\)
    are H\"ormander vector fields with formal degrees on \(\Omega'\).

    Set \(\Omega_1:=\Omega\cap \Omega'\) and let
    \begin{equation*}
        (W\big|_{\Omega_1},\Wd):=\left\{ (W_1\big|_{\Omega_1},\Wd_1),\ldots, (W_r\big|_{\Omega_1},\Wd_r) \right\}, \quad
        (Z\big|_{\Omega_1},\Zd):=\left\{ (Z_1\big|_{\Omega_1},\Zd_1),\ldots, (Z_s\big|_{\Omega_1},\Zd_s) \right\}.
    \end{equation*}
    Theorem \ref{Thm::Metrics::Results::ExtendedVFsGiveSameMetric} (with \(\ZZd\) replaced by \((W\big|_{\Omega_1},\Wd)\))
    shows \(\rho_{\WWd}\) and \(\rho_{(W\big|_{\Omega_1},\Wd)}\) are locally Lipschitz equivalent on \(\Omega_1\).
    Similarly, Theorem \ref{Thm::Metrics::Results::ExtendedVFsGiveSameMetric}
    also shows \(\rho_{\ZZd}\) and \(\rho_{(Z\big|_{\Omega_1},\Zd)}\) are locally Lipschitz equivalent on \(\Omega_1\).

    Proposition \ref{Proposition::Sheaves::Control::LocalEquivWithPreviousDefns} 
    \ref{Item::Sheaves::Control::LocalEquivWithPreviousDefns::WeakEquiv::Sheaves}\(\Rightarrow\)\ref{Item::Sheaves::Control::LocalEquivWithPreviousDefns::WeakEquiv::RelCpt}
    shows \((W\big|_{\Omega_1},\Wd)\) and \((Z\big|_{\Omega_1},\Zd)\) are locally weakly equivalent on \(\Omega_1\)
    and Theorem \ref{Thm::Metrics::Results::ExtendedVFsGiveSameMetric} shows \(\rho_{(W\big|_{\Omega_1},\Wd)}\)
    and \(\rho_{(Z\big|_{\Omega_1},\Zd)}\) are locally Lipschitz equivalent on \(\Omega_1\).
    We conclude \(\rho_{\WWd}\) and \(\rho_{\ZZd}\) are locally Lipschitz equivalent on \(\Omega_1\)
    and therefore \(\rho_{\ZZd}\big|_{\Compact\times \Compact}\) is locally Lipschitz equivalent to \(\rho_\Compact\);
    in particular the topologies on \(\Compact\) generated by \(\rho_\Compact\), \(\rho_{\ZZd}\), and \(\rho_{\ZZd}\wedge 1\)
    are all the same.
    By \ref{Item::Sheaves::Metrics::ExistsEquivalenceClass::SameTop}, \(\Compact\) is compact with respect to
    this common topology and Lemma \ref{Lemma::Metrics::Lemmas::CompactLocalEquivMeansEquiv} applies
    to show that \(\rho_{\Compact}\) and \(\rho_{\ZZd}\wedge 1\big|_{\Compact}\) are Lipschitz equivalent.
    This establishes \ref{Item::Sheaves::Metrics::ExistsEquivalenceClass::AgreesWithCCInducedByG}.

    \ref{Item::Sheaves::Metrics::ExistsEquivalenceClass::AgreesWithCCInducedByF} follows by taking \(\FilteredSheafG=\FilteredSheafF\)
    in  \ref{Item::Sheaves::Metrics::ExistsEquivalenceClass::AgreesWithCCInducedByG}.

    \ref{Item::Sheaves::Metrics::ExistsEquivalenceClass::Unique}: Suppose \(\rhot_\Compact\)
    is any other extended metric on \(\Compact\) satisfying \ref{Item::Sheaves::Metrics::ExistsEquivalenceClass::AgreesWithCCInducedByF::LocallyEquiv}.
    Then, \(\rhot_\Compact\) is locally Lipschitz equivalent to \(\rho_{\WWd}\), where \(\WWd\) is as chosen in the start of the proof.
    We conclude \(\rhot_\Compact\) is locally Lipschitz equivalent to \(\rho_{\Compact}\).  Now suppose \(\rhot_\Compact\)
    is also a metric such that \ref{Item::Sheaves::Metrics::ExistsEquivalenceClass::SameTop} holds. 
    It follows from Lemma \ref{Lemma::Metrics::Lemmas::CompactLocalEquivMeansEquiv}
    (using \ref{Item::Sheaves::Metrics::ExistsEquivalenceClass::SameTop} for both \(\rho_\Compact\) and \(\rhot_\Compact\))
    that
    that \(\rho_\Compact\) and \(\rhot_\Compact\) are Lipschitz equivalent.

    \ref{Item::Sheaves::Metrics::ExistsEquivalenceClass::Restrict}:
    Take \(\Omega\Subset \ManifoldNncF\) open and relatively compact with \(\Compact_2\Subset \Omega\),
    and take \(\WWd=\left\{ (W_1,\Wd_1),\ldots, (W_r,\Wd_r) \right\}\subset \VectorFields{\Omega}\times \Zg\)
    H\"ormander vector fields with formal degrees on \(\Omega\) (this is always possible;
    see Proposition \ref{Prop::Sheaves::LieAlg::HormandersCondiIsEquivToOtherNotions}). Then, by
    \ref{Item::Sheaves::Metrics::ExistsEquivalenceClass::AgreesWithCCInducedByF::LocallyEquiv},
    \(\rho_{\Compact_1}\) is locally Lipschitz equivalent to \(\rho_{\WWd}\big|_{\Compact_1\times \Compact_1}\)
    and \(\rho_{\Compact_2}\) is locally Lipschitz equivalent to \(\rho_{\WWd}\big|_{\Compact_2\times \Compact_2}\).
    We conclude \(\rho_{\Compact_1}\) is locally Lipschitz equivalent to \(\rho_{\Compact_2}\big|_{\Compact_1\times \Compact_1}\).
    By \ref{Item::Sheaves::Metrics::ExistsEquivalenceClass::SameTop}, 
    \(\rho_{\Compact_1}\) and \(\rho_{\Compact_2}\) both generate the topology of \(\Compact_1\) (which is compact)
    and it follow from Lemma \ref{Lemma::Metrics::Lemmas::CompactLocalEquivMeansEquiv}
    that \(\rho_{\Compact_1}\) and \(\rho_{\Compact_2}\big|_{\Compact_1\times \Compact_1}\) are Lipschitz equivalent.
\end{proof}

As described in Notation \ref{Notation::Sheaves::Metrics::WellDefined},
Theorem \ref{Thm::Sheaves::Metrics::ExistsEquivalenceClass} gives a well-defined Lipschitz equivalence class
of metrics on each compact subset of \(\ManifoldNncF\). The next result shows that this equivalence class
can be either defined intrinsically on \(\ManifoldN\), or extrinsically on an ambient space (i.e., we can
let the paths defining the Carnot--Carath\'eodory metric leave \(\ManifoldN\)).

\begin{theorem}\label{Thm::Sheaves::Metrics::SameMetricFromAmbientSpace}
    Suppose \(\ManifoldM\) is a smooth manifold without boundary with \(\ManifoldN\subseteq \ManifoldM\)
    a closed, co-dimension \(0\) embedded submanifold (with boundary). Let \(\FilteredSheafFh\)
    be a filtration of sheaves of vector fields on \(\ManifoldM\) which is of finite type, satisfies H\"ormander's condition,
    and such that
    \(\RestrictFilteredSheaf{\FilteredSheafFh}{\ManifoldN}=\FilteredSheafF\).
    Then, for any compact set \(\Compact\Subset \ManifoldNncF\), 
    \(\rho_{\FilteredSheafF}\big|_{\Compact\times \Compact}\) is Lipschitz equivalent to
    \(\rho_{\FilteredSheafFh}\big|_{\Compact\times \Compact}\)
    (see Notation \ref{Notation::Sheaves::Metrics::WellDefined} for this notation).
\end{theorem}
\begin{proof}
    Let \(\Compact\Subset \ManifoldNncF\) be compact
    and fix \(\Omega\Subset \ManifoldM\) open and relatively compact
    with \(\Compact\Subset \Omega\) and \(\Omega\cap \ManifoldN\subseteq \ManifoldNncF\),
    and
    \(\WhWd=\left\{ (\Wh_1,\Wd_1),\ldots, (\Wh_r,\Wd_r) \right\}\subset \VectorFields{\Omega}\times \Zg\)
    H\"ormander vector fields with formal degrees such that
    \(\FilteredSheafFh\big|_{\Omega}=\FilteredSheafGenBy{\WhWd}\)
    (this is always possible;
    see Proposition \ref{Prop::Sheaves::LieAlg::HormandersCondiIsEquivToOtherNotions}).
    Theorem \ref{Thm::Sheaves::Metrics::ExistsEquivalenceClass} \ref{Item::Sheaves::Metrics::ExistsEquivalenceClass::AgreesWithCCInducedByF::LocallyEquiv}
    shows \(\rho_{\FilteredSheafFh}\big|_{\Compact\times\Compact}\)
    is locally Lipschitz equivalent to \(\rho_{\WhWd}\big|_{\Compact\times \Compact}\).

    Set \(W_j:=\Wh_j\big|_{\ManifoldN\cap \Omega}\)
    and \(\WWd=\left\{ (W_1,\Wd_1),\ldots, (W_r,\Wd_r) \right\}\), so that by 
    Proposition \ref{Prop::Sheaves::Restrict::Codim0::MainRestrictionResult} \ref{Item::Sheaves::Restrict::Codim0::MainRestrictionResult::FGOnSet}
    we have \(\FilteredSheafF[\Omega\cap\ManifoldN]=\RestrictFilteredSheaf{\FilteredSheafFh}{\ManifoldN}[\Omega\cap\ManifoldN]=\FilteredSheafGenBy{\WWd} \),
    and therefore by 
    Theorem \ref{Thm::Sheaves::Metrics::ExistsEquivalenceClass} \ref{Item::Sheaves::Metrics::ExistsEquivalenceClass::AgreesWithCCInducedByF}
    \(\rho_{\FilteredSheafF}\big|_{\Compact\times \Compact}\) is locally Lipschitz equivalent
    to \(\rho_{\WWd}\big|_{\Compact\times \Compact}\).

    Every point of \(\partial (\ManifoldN\cap \Omega)=\BoundaryN\cap \Omega\) is contained in \(\BoundaryNncF\)
    and is therefore \(\FilteredSheafF\)-non-characteristic; and by 
    Proposition \ref{Prop::Sheaves::Restrict::Bndry::NonCharIsNonChar} \ref{Item::Sheaves::Restrict::Bndry::NonCharIsNonChar::NonCharIsSame}
    it follows that every point of \(\partial (\ManifoldN\cap \Omega)\) is \(\WWd\)-non-characteristic.
    We apply Theorem \ref{Thm::Metrics::Results::ExtendedVFsGiveSameMetric} with \(\ManifoldM\) replaced by \(\Omega\)
    and \(\ManifoldN\) replaced by \(\ManifoldN\cap \Omega\) to see that
    \(\rho_{\WhWd}\) and \(\rho_{\WWd}\) are locally Lipschitz equivalent on \(\Omega\cap \ManifoldN\);
    and therefore are locally Lipschitz equivalent on \(\Compact\).
    We conclude \(\rho_{\FilteredSheafFh}\big|_{\Compact\times\Compact}\) and \(\rho_{\FilteredSheafF}\big|_{\Compact\times \Compact}\)
    are locally Lipschitz equivalent on \(\Compact\).

    By Theorem \ref{Thm::Sheaves::Metrics::ExistsEquivalenceClass} \ref{Item::Sheaves::Metrics::ExistsEquivalenceClass::SameTop},
    the topology on \(\Compact\)
    is induced by both \(\rho_{\FilteredSheafFh}\big|_{\Compact\times\Compact}\) and \(\rho_{\FilteredSheafF}\big|_{\Compact\times \Compact}\)
    and it follows from Lemma \ref{Lemma::Metrics::Lemmas::CompactLocalEquivMeansEquiv}
    that \(\rho_{\FilteredSheafFh}\big|_{\Compact\times\Compact}\) and \(\rho_{\FilteredSheafF}\big|_{\Compact\times \Compact}\)
    are Lipschitz equivalent.
\end{proof}

There are two natural ways to define a Lipschitz equivalence class of Carnot--Carath\'eodory metrics
on compact subsets of the non-characteristic boundary:
by thinking of them as compact subsets of \(\ManifoldNncF\) and using Theorem \ref{Thm::Sheaves::Metrics::ExistsEquivalenceClass}
applied to \(\FilteredSheafF\); or by thinking of them as compact subsetes of \(\BoundaryNncF\)
and using Theorem \ref{Thm::Sheaves::Metrics::ExistsEquivalenceClass}
applied to \(\RestrictFilteredSheaf{\LieFilteredSheafF}{\BoundaryNncF}\)
(which is of finite type and satisfies H\"ormander's condition by Proposition \ref{Prop::Sheaves::Restrict::Bndry::RestritionIsFiniteTypeAndSpandTangent}).
The next result shows that these two ways yield the same Lipschitz equivalence class.

\begin{theorem}\label{Thm::Sheaves::Metrics::BoundaryMetricEquivalence}
    Let \(\Compact\Subset \BoundaryNncF\) be compact. Then,
    \(\rho_{\FilteredSheafF}\big|_{\Compact\times \Compact}\)
    is Lipschitz equivalent to
    \(\rho_{\RestrictFilteredSheaf{\LieFilteredSheafF}{\BoundaryNncF}}\big|_{\Compact\times \Compact}\).
\end{theorem}
\begin{proof}
    By Proposition \ref{Prop::Sheaves::LieAlg::FFiniteTypeAndHorImpliesSameForLie},
    \(\LieFilteredSheafF\) is of finite type and satisfies H\"ormander's condition,
    and therefore Theorem \ref{Thm::Sheaves::Metrics::ExistsEquivalenceClass} \ref{Item::Sheaves::Metrics::ExistsEquivalenceClass::AgreesWithCCInducedByG::Equiv}
    shows \(\rho_{\FilteredSheafF}\big|_{\Compact\times \Compact}\) is Lipschitz equivalent to
    \(\rho_{\LieFilteredSheafF}\big|_{\Compact\times \Compact}\). Thus, it suffices to show
    \(\rho_{\LieFilteredSheafF}\big|_{\Compact\times \Compact}\) is Lipschitz equivalent to
    \(\rho_{\RestrictFilteredSheaf{\LieFilteredSheafF}{\BoundaryNncF}}\big|_{\Compact\times \Compact}\).

    By Theorem \ref{Thm::Sheaves::Metrics::ExistsEquivalenceClass} \ref{Item::Sheaves::Metrics::ExistsEquivalenceClass::SameTop},
    the topology on \(\Compact\) in induced by both  \(\rho_{\LieFilteredSheafF}\big|_{\Compact\times \Compact}\)
    and \(\rho_{\RestrictFilteredSheaf{\LieFilteredSheafF}{\BoundaryNncF}}\big|_{\Compact\times \Compact}\),
    and therefore by Lemma \ref{Lemma::Metrics::Lemmas::CompactLocalEquivMeansEquiv} it suffices to show
    that they are locally Lipschitz equivalent.
    Thus, in light of Theorem \ref{Thm::Sheaves::Metrics::ExistsEquivalenceClass} \ref{Item::Sheaves::Metrics::ExistsEquivalenceClass::Restrict}
    it suffices to show \(\forall x_0\in \BoundaryNncF\), \(\exists\) a compact neighborhood \(\Compact_0\Subset \BoundaryNncF\) of \(x_0\)
    such that \(\rho_{\LieFilteredSheafF}\big|_{\Compact_0\times \Compact_0}\)
    and \(\rho_{\RestrictFilteredSheaf{\LieFilteredSheafF}{\BoundaryNncF}}\big|_{\Compact_0\times \Compact_0}\)
    are equivalent.

    Fix \(x_0\in \BoundaryNncF\).
    Let \(\Omega\Subset \ManifoldN\) be a neighborhood of \(x_0\) such that \(\FilteredSheafF\big|_{\Omega}=\FilteredSheafGenBy{\WWd}\)
    for some H\"ormander vector fields with formal degrees on \(\Omega\) (this is always possible;
    see Proposition \ref{Prop::Sheaves::LieAlg::HormandersCondiIsEquivToOtherNotions}).
    Let \(\Omega_2\Subset \ManifoldNncF\), \(\XXd\), and \(\VVd\) be as in Construction \ref{Construction::Sheaves::Restrict::Bndry::VVd},
    so that \(\Omega_2\) is a neighborhood of \(x_0\) and
    \begin{equation*}
        \LieFilteredSheafF\big|_{\Omega_2}=\FilteredSheafGenBy{\XXd}\big|_{\Omega_2}, \quad
        \RestrictFilteredSheaf{\LieFilteredSheafF}{\BoundaryNncF}\big|_{\Omega_2\cap \ManifoldN}=\FilteredSheafGenBy{\VVd}\big|_{\Omega_2\cap \BoundaryN}.
    \end{equation*}
    For \(\Compact_1\subseteq \Omega_2\cap \BoundaryN\subseteq \BoundaryNncF\) compact,
    Theorem \ref{Thm::Sheaves::Metrics::ExistsEquivalenceClass} \ref{Item::Sheaves::Metrics::ExistsEquivalenceClass::AgreesWithCCInducedByF::LocallyEquiv}
    shows
    \(\rho_{\LieFilteredSheafF}\big|_{\Compact_1\times \Compact_1}\) is locally Lipschitz equivalent to
    \(\rho_{\XXd}\big|_{\Compact_1\times \Compact_1}\) (here we are using that \(\BoundaryNncF=\BoundaryNnc_{\LieFilteredSheafF}\) by the definition)
    and \(\rho_{\RestrictFilteredSheaf{\LieFilteredSheafF}{\BoundaryNncF}}\big|_{\Compact_1\times \Compact_1}\)
    is locally Lipschitz equivalent to \(\rho_{\VVd}\big|_{\Compact_1\times \Compact_1}\).
    Thus, it suffices to show that there is a compact neighborhood \(\Compact_0\Subset \BoundaryNncF\cap \Omega_2\) of \(x_0\)
    such that \(\rho_{\XXd}\big|_{\Compact_0\times \Compact_0}\) is locally Lipschitz equivalent to
    \(\rho_{\VVd}\big|_{\Compact_0\times \Compact_0}\).

    It is clear from the construction (see Section \ref{Section::BndryVfsWithFormaLDegrees}) that \(\rho_{\XXd}(x,y)\leq \rho_{\VVd}(x,y)\), \(\forall x,y\in \BoundaryNncF\cap \Omega_2\),
    so we need only show \(\rho_{\VVd}(x,y)\lesssim \rho_{\XXd}(x,y)\) for \(x,y\) in some \(\BoundaryN\)-neighborhood of \(x_0\).
    In local coordinates, this follows from Lemma \ref{Lemma::Scaling::NearBdry::VVdLessThanXXd}
    (take the neighborhood to be contained in \(\left\{ y : \rho_{\XXd}(x,y)<\delta_1 \right\}\)
    where \(\delta_1>0\) is as in that lemma).
\end{proof}

\begin{example}\label{Example::Sheaves::Metrics::CompactWithNCBdry}
    Suppose \(\ManifoldN\) is a connected, compact manifold with boundary and \(\FilteredSheafF\)
    is a filtration of sheaves of vector fields on \(\ManifoldN\) which is of finite type and satisfies H\"ormander's condition.
    Assume that \textbf{every point of \(\BoundaryN\) is \(\FilteredSheafF\)-non-characteristic}; i.e., assume
    \(\BoundaryN=\BoundaryNncF\).
    \begin{enumerate}[(i)]
        \item\label{Item::Sheaves::Metrics::CompactWithNCBdry::HasMetric} Because \(\ManifoldN\) is compact, Theorem \ref{Thm::Sheaves::Metrics::ExistsEquivalenceClass}
            (see, also, Notation \ref{Notation::Sheaves::Metrics::WellDefined}) shows \(\FilteredSheafF\)
            induces a Lipschitz equivalence class of metrics on \(\ManifoldN\); we denote any choice of such a metric by
            \(\rho_{\FilteredSheafF}\).

        \item\label{Item::Sheaves::Metrics::CompactWithNCBdry::EquivToCCMetric}  By Proposition \ref{Prop::Sheaves::LieAlg::HormandersCondiIsEquivToOtherNotions}, \(\FilteredSheafF=\FilteredSheafGenBy{\WWd}\),
            where \(\WWd=\left\{ (W_1,\Wd_1),\ldots, (W_r,\Wd_r) \right\}\subset \VectorFieldsN\times \Zg\) are 
            H\"ormander vector fields with formal degrees on \(\ManifoldN\).
            Theorem \ref{Thm::Sheaves::Metrics::ExistsEquivalenceClass} \ref{Item::Sheaves::Metrics::ExistsEquivalenceClass::AgreesWithCCInducedByF}
            shows \(\rho_{\FilteredSheafF}\) and \(\rho_{\WWd}\) are locally Lipschitz equivalent; and by
            Theorem \ref{Thm::Metrics::Results::GivesUsualTopology} and Theorem \ref{Thm::Sheaves::Metrics::ExistsEquivalenceClass} \ref{Item::Sheaves::Metrics::ExistsEquivalenceClass::SameTop}
            they both give the usual topology on \(\ManifoldN\).
            Since \(\ManifoldN\) is assumed to be connected, \(\rho_{\WWd}\) is a metric, and Lemma \ref{Lemma::Metrics::Lemmas::CompactLocalEquivMeansEquiv}
            implies \(\rho_{\FilteredSheafF}\) and \(\rho_{\WWd}\) are Lipschitz equivalent.

        \item By Proposition \ref{Prop::Sheaves::Restrict::Bndry::RestritionIsFiniteTypeAndSpandTangent},
        \(\RestrictFilteredSheaf{\LieFilteredSheafF}{\BoundaryN}\) is of finite type and satisfies H\"ormander's condition,
        and therefore using that \(\BoundaryN\) is compact and Proposition \ref{Prop::Sheaves::LieAlg::HormandersCondiIsEquivToOtherNotions},
        we have
        \(\RestrictFilteredSheaf{\LieFilteredSheafF}{\BoundaryN}=\FilteredSheafGenBy{\VVd}\), where \(\VVd=\left\{ (V_1,\Vd_1),\ldots, (V_q,\Vd_q) \right\}\)
        are H\"ormander vector fields with formal degrees on \(\BoundaryN\).

        \item As in \ref{Item::Sheaves::Metrics::CompactWithNCBdry::HasMetric}, \(\RestrictFilteredSheaf{\LieFilteredSheafF}{\BoundaryN}\)
            induces a Lipschitz equivalence class of metrics on \(\BoundaryN\), and we denote any choice of such a metric by
            \(\rho_{\RestrictFilteredSheaf{\LieFilteredSheafF}{\BoundaryN}}\).

        \item As in \ref{Item::Sheaves::Metrics::CompactWithNCBdry::EquivToCCMetric}, \(\rho_{\RestrictFilteredSheaf{\LieFilteredSheafF}{\BoundaryN}}\)
            and \(\rho_{\VVd}\) are locally Lipschitz equivalent. Since \(\BoundaryN\) may not be connected,
            \(\rho_{\VVd}\) might be merely an extended metric and not a metric. It is a metric on connected components
            of \(\BoundaryN\) (see Remark \ref{Rmk::Metrics::lemmas::ExtendedMetricTop}) and as in 
            \ref{Item::Sheaves::Metrics::CompactWithNCBdry::EquivToCCMetric}, \(\rho_{\VVd}\) and \(\rho_{\RestrictFilteredSheaf{\LieFilteredSheafF}{\BoundaryN}}\)
            are Lipschitz equivalent on connected components of \(\BoundaryN\).

        \item By Theorem \ref{Thm::Sheaves::Metrics::BoundaryMetricEquivalence}
        \(\rho_{\FilteredSheafF}\big|_{\BoundaryN}\) and \(\rho_{\RestrictFilteredSheaf{\LieFilteredSheafF}{\BoundaryN}}\)
        are Lipschitz equivalent and therefore \(\rho_{\WWd}\) and \(\rho_{\VVd}\) Lipschitz equivalent on connected components
        of \(\BoundaryN\).
    \end{enumerate}
\end{example}

\bibliographystyle{amsalpha}

\bibliography{bibliography}

\center{\it{University of Wisconsin-Madison, Department of Mathematics, 480 Lincoln Dr., Madison, WI, 53706}}

\center{\it{street@math.wisc.edu}}

\center{MSC 2020: 53C17 (Primary), 51F30 and 42B99 (Secondary)}


\end{document}